\documentclass[reqno]{amsart}
\usepackage[top=80pt,bottom=80pt,left=80pt,right=80pt]{geometry}
\usepackage[utf8]{inputenc}
\usepackage[english]{babel}
\usepackage{amsmath,amssymb,amsfonts,mathrsfs,mathtools}
\mathtoolsset{showonlyrefs}
\usepackage{enumitem}
\usepackage{hyperref}
\usepackage{graphicx}
\usepackage{caption}
\usepackage{subcaption}

\theoremstyle{plain}
\newtheorem{theorem}{Theorem}[section]

\newtheorem{remark}[theorem]{Remark}
\newtheorem{corollary}[theorem]{Corollary}
\newtheorem{definition}[theorem]{Definition}
\newtheorem{lemma}[theorem]{Lemma}
\newtheorem{proposition}[theorem]{Proposition}

\newtheorem{question}{Question}

\newcommand{\C}{\mathbb{C}}

\newcommand{\N}{\mathbb{N}}

\newcommand{\R}{\mathbb{R}}
\newcommand{\Z}{\mathbb{Z}}

\DeclarePairedDelimiter\abs{\lvert}{\rvert}

\renewcommand{\epsilon}{\ensuremath\varepsilon}
\newcommand{\eps}{\epsilon}

\renewcommand{\phi}{\ensuremath{\varphi}}

\newcommand{\Oh}[3]{\tilde{O}_{#2}^{#1}(#3)}
\newcommand{\oh}[3]{O_{#2}^{#1}(#3)}

\renewcommand{\P}[2]{{#2}^{(#1)}}

\newcommand{\M}[2]{\partial_\ast^{#1} #2}
\newcommand{\m}[1]{\partial_\ast #1}

\renewcommand{\C}[1]{[#1]}

\newcommand{\into}{\hookrightarrow}

\newcommand{\mb}{\partial_*}
\newcommand{\mc}{\mathcal}
\newcommand{\comb}{\delta_*}

\newcommand{\st}{\underline{*}}

\newcommand{\und}{\underline}

\newcommand{\pres}{4}
\newcommand{\bij}{5}

\newcommand{\emp}{{7}}

\newcommand{\xgg}{{9}}
\newcommand{\lam}{{10}}
\newcommand{\ftree}{{11}}
\newcommand{\empbij}{{12}}
\newcommand{\mov}{{13}}
\newcommand{\blam}{{14}}
\newcommand{\clam}{{15}}
\newcommand{\comp}{{16}}

\newcommand{\ver}{7}
\newcommand{\dlam}{8}
\newcommand{\dclam}{9}

\newcommand{\tree}{4}

\author{Stefanie Zbinden}
\title{Morse boundaries of 3-manifold groups}

\begin{document}
\maketitle

\begin{abstract}
We classify the homeomorphism types of the Morse boundaries of all 3-manifold groups into 9 different possible homeomorphism types, and show how the Morse boundary depends on the geometric decomposition.  
\end{abstract}

\section{Introduction}

The Morse boundary was defined in \cite{Cor16} and is a generalization of the Gromov boundary for non-hyperbolic groups. It captures the hyperbolic directions of groups and is a quasi-isometry invariant. In \cite{CH17,Charney_2019,Murray_2019,liu_dynamics,Zalloum2018ASC,graph_of_groups1} and \cite{fioravanti2022connected} it was shown that Morse boundaries inherit several properties of the Gromov boundary. However, unlike the Gromov boundary, the Morse boundary is in general neither compact nor metrizable. Before \cite{charney2020complete} managed to describe the topology of certain Morse boundaries (e.g. of fundamental groups of graph manifolds), it was thought to be near impossible to describe the topology comprehensively for non-trivial examples. In this paper, we build on the work of \cite{charney2020complete} and \cite{graph_of_groups1} to classify the Morse boundaries of all 3-manifold groups, that is, the fundamental group of all closed and connected 3-manifolds. 

\begin{theorem}\label{theorem:manifold_light} There exist 9 homeomorphism types of $\mb \pi_1(M)$, where $M$ is a closed and connected 3-manifold. Moreover the homeomorphism type depends only on the geometric decomposition of $M$.
\end{theorem}

In Theorem \ref{theorem:3-manifold} we make the above Theorem more precise, giving both the list of possible homeomorphism types of the Morse boundary of $\pi_1(M)$ and a description of which geometric decomposition leads to which Morse boundary.

The fundamental group, $\pi_1(M)$, has a graph of groups decomposition where all edge groups are either $\Z^2$ or trivial and the vertex groups are fundamental groups of geometric prime manifolds, Seifert fibered manifolds or finite volume hyperbolic manifolds with cusps. In particular, the Morse boundary of the vertex groups are known and Theorem \ref{cor:rel_hyp_graph_of_groups} below can be used to show that the Morse boundary $\mb (\pi_1(M))$ is homeomorphic to the one of a graph of groups with the same vertex groups but trivial edge groups. Graph of groups with trivial edge groups were studied in \cite{graph_of_groups1}, so Theorem \ref{theorem:manifold_light}, up to counting some homeomorphism types multiple times, follows from Theorem \ref{cor:rel_hyp_graph_of_groups} and \cite{graph_of_groups1}. 

\begin{theorem}\label{cor:rel_hyp_graph_of_groups}
Let $\mc G$ be a graph of finitely generated groups with underlying graph $\Gamma$ and a distinguished set of vertices $W\subset \Gamma$ such that:
\begin{enumerate}
    \item All edge groups are undistorted in $\pi_1(\mc{G})$, wide and have infinite index in their corresponding vertex groups. 
    \item For every vertex $w\in W$, the vertex group $G_w$ is hyperbolic relative to a collection of subgroups containing the adjacent edge groups.\label{cor:assumption2}
    \item Every edge has at least one of its endpoints in $W$.\label{cor:assumption3}
\end{enumerate}
Let $\mc G'$ be the graph of groups which is obtained from $\mc G$ by defining all edge groups to be trivial groups. Then $\mb \pi_1(\mc G) \cong \mb \pi_1 (\mc G ')$.
\end{theorem}

To prove Theorem \ref{cor:rel_hyp_graph_of_groups}, we follow the proof strategy of \cite{martin2013hyperbolic} and show that the equivalence class of every Morse ray is determined by its projection to the Bass-Serre tree and (if this is finite) its tail in one vertex group. Since there are in general no cut points in the Cayley graph of $\pi_1(\mc G)$, this step requires more care and is the main reason for the added assumption of relative hyperbolicity. As a second step we adapt the tools from \cite{graph_of_groups1} to construct bijections $q_\alpha$ between vertex groups $G_i$ and their quotient by the adjacent edge group $H_\alpha$. For every coset of the edge group $H_\alpha$, we choose one distinguished element of that coset. With this, we can view the bijections $q_\alpha$ as (non-surjective) maps from $G_i$ to itself. There are two main conditions the bijections $q_\alpha$ need to satisfy. On one hand, the bijections $q_\alpha$ need to preserve the Morseness of elements (in the sense of Definition \ref{definition:morse_elements}). On the other hand, if a sequence of $M$-Morse elements converges to a Morse direction $\xi$, then its image (viewed as a sequence of elements in $G_i$) needs to converge to the Morse direction $\xi$. In Section \ref{sec:local-bijections}, we make the above conditions more precise by introducing the notion of \emph{local bijections}. 
The bijections $q_\alpha$ then allow us to construct a bijection between the Bass-Serre trees of $\mc G$ and $\mc G'$ which induces a homeomorphism between the Morse boundaries $\mb \pi_1(\mc G)$ and $\mb \pi_1(\mc G ')$.

The strategy above to first determine the Morse boundary of all prime manifolds and then use the results of \cite{graph_of_groups1} to determine the Morse boundary of the 3-manifold as a whole will yield that there are 11 possible homeomorphism types of Morse boundaries, which are all the Morse boundary of some closed and connected 3-manifold. However, those types are not necessarily pairwise non-homeomorphic. We can improve this to show that there are exactly 9 possible homeomorphism types which are all pairwise non-homeomorphic. To do so, we use the following theorem. 

\begin{theorem}\label{theorem:empty_morse_boundary}
Let $G$ and $H$ be finitely generated groups such that $G$ is not hyperbolic and $H$ is infinite and has empty Morse boundary. 
Then $\mb (G*G)\cong \mb (G* H)$.
\end{theorem}

We prove Theorem \ref{theorem:empty_morse_boundary} by constructing a bijection from $G$ to $G\times ( H - \{1\})$. Again, this bijection needs to preserve the Morseness of elements and respect sequences that correspond to Morse rays. The bijection constructed then induces a map from the Bass-Serre tree of $G*G$ to the Bass-Serre tree of $G*H$ which in turn induces a homeomorphism between the Morse boundary of $G*G$ and the Morse boundary of $G*H$.

\subsection*{Further Questions} 

Consider the group $G_1 =\langle a, b \mid aba^{-1}b^{-1} \rangle *_{\langle b\rangle} \langle b, c\mid bcb^{-1}c^{-1} \rangle\cong F_2\times \Z$. The subgroup $\langle b\rangle$ is relatively wide in both factors and the Morse boundary of $G_1$ is empty. However, the Morse boundary of $\Z^2*\Z^2$ is an $\omega$-Cantor set (see \cite{charney2020complete}). This shows that if we remove assumption \eqref{cor:assumption2} from Theorem \ref{cor:rel_hyp_graph_of_groups}, then it does not hold anymore. On the other hand, while graphs of groups which are free products do not (in general) satisfy assumption \eqref{cor:assumption2} of Theorem \ref{cor:rel_hyp_graph_of_groups}, they trivially satisfy the conclusion of Theorem \ref{cor:rel_hyp_graph_of_groups}. Thus the questions arises whether and how much we can relax the assumptions in Theorem 1.2. Observe that replacing assumption \eqref{cor:assumption2} with an assumption that does not use the set of distinguished vertices $W$ makes them obsolete. That is, we automatically also get rid of assumption \eqref{cor:assumption3} and the need to define a set of distinguished vertices.

One property that groups satisfying the assumptions of Theorem \ref{cor:assumption2} and free products have in common is that they both act acylindrically on their Bass-Serre trees. Furthermore the group $G_1$ as constructed above does not have this property. Therefore, one might ask the following question. 

\begin{question}\label{q1}
Does Theorem \ref{cor:rel_hyp_graph_of_groups} hold if we replace assumption \eqref{cor:assumption2} by the assumption that $\pi_1(\mc G)$ acts acylindrically on its Bass-Serre tree?
\end{question}

As argued before, if one removes assumption \eqref{cor:assumption2} from Theorem \ref{cor:rel_hyp_graph_of_groups}, it does not hold anymore. But maybe a weaker version does. 

\begin{question}
    If we remove assumption \eqref{cor:assumption2} from Theorem \ref{cor:rel_hyp_graph_of_groups}, is it true that $\mb \pi_1(\mc G)$ embeds topologically into $\mb \pi_1(\mc G')$?
\end{question}

While such a version of the Theorem is much weaker, it would still provide us with with some tools to talk about the Morse boundary of $\pi_1(\mc G)$.

\subsection*{Outline of the paper}

In Section \ref{section:preliminaries} we recall some definitions, most notably the definition of the Morse boundary.
In Section \ref{section:morseless_stars} we introduce a certain subclass of graphs of groups which we call Morse preserving Morseless stars. For those graphs of groups we give a description of Morse rays in terms of their projection to the Bass-Serre tree. In Section \ref{section:homeo_tools} we develop tools to show that certain maps between Bass-Serre trees induce homeomorphisms between the respective Morse boundaries. In Section \ref{section:relativemap} we prove Theorem \ref{theorem:star_of_groups}, which is a special case of Theorem \ref{cor:rel_hyp_graph_of_groups} where the set of distinguished vertices contains only one vertex. In Section \ref{section:doublemap} we prove Theorem \ref{theorem:empty_morse_boundary}. Lastly, in Section \ref{section:main_result} we apply Theorem 
 \ref{theorem:star_of_groups} and \ref{theorem:empty_morse_boundary} to prove  Theorem \ref{cor:rel_hyp_graph_of_groups} and \ref{theorem:3-manifold}, where the latter is a more detailed version of Theorem \ref{theorem:manifold_light}.

\subsection*{Acknowledgements}
I want to thank Elia Fioravanti and Annette Karrer for useful conversations. I would also like to thank my advisor Alessandro Sisto for his ongoing support and helpful comments on the paper.

\section{Preliminaries}\label{section:preliminaries}

\textbf{Notation:}  For this section, unless stated otherwise, $(X, e)$ denotes a pointed proper geodesic metric space and $C$ denotes a constant at least $1$. To ease notation, we denote a $(C, C)$-quasi-geodesic as $C$-quasi-geodesic. Furthermore, we write quasi-geodesic instead of continuous quasi-geodesic, that is, we assume all quasi-geodesics are continuous. Furthermore, we assume that, unless stated otherwise, the domain of any ray is $[0, \infty)$ and that the domain of a segment is $[0, t]$ for some $t\in \R$.

A well known but important way to construct quasi-geodesics is the following Lemma, a proof can be found for example in \cite{graph_of_groups1}.

\begin{lemma}\label{quasi-geodesic}
Let $p, q \in X$ and let $\gamma : [0, T]\to X$ be a $(\lambda, \eps)$-quasi-geodesic. Let $t, t^\prime\in [0, T]$ such that $\gamma(t)$ and $\gamma(t^\prime)$ are points on $\gamma$ closest to $p$ and $q$ respectively. Finally, let $\alpha$ be a geodesic from $p$ to $\gamma(t)$ and let $\beta$ be a geodesic from $\gamma(t^\prime)$ to $q$. The following hold: 
\begin{enumerate}[label= \roman*)] 
    \item $\alpha \cdot \gamma[t, T]$ is a $(3\lambda, \eps)$ -quasi-geodesic.
    \item If $\abs{t - t^\prime} \geq 3\lambda (d_X(p, \gamma(t))+d_X(q, \gamma(t^\prime)))$, then $\alpha \cdot \gamma[t, t^\prime]\cdot \beta$ is a $(3\lambda, \eps)$ -quasi-geodesic.
\end{enumerate}
\end{lemma}

\subsection{Morse boundary}
A Morse gauge is a map $M : \R_{\geq 1}\times \R_{\geq_0}\to \R_{\geq0}$. We will assume, unless stated otherwise, that every Morse gauge is increasing and is continuous in the second argument. We denote the set of all Morse gauges (that satisfy these properties) by $\mathcal{M}$.

The set of all Morse gauges $\mathcal{M}$ has a partial order defined by $M\leq M^\prime$ if and only if $M(\lambda, \eps)\leq M^\prime(\lambda, \eps)$ for all $(\lambda, \eps) \in \R_{\geq 1}\times \R_{\geq 0}$. Observe that with this order, the maximum of a set of Morse gauges is the point-wise maximum.

\begin{definition}[Morse geodesic]
A quasi-geodesic $\gamma$ in $X$ is called $M$-Morse, if for all continuous $(\lambda, \eps)$-quasi-geodesics $\eta$ with endpoints on $\gamma$ we have
\begin{align}
    \eta \subset \mathcal{N}_{M(\lambda, \eps)}(\gamma),
\end{align}
where $\mathcal{N}_{M(\lambda, \eps)}(\gamma)$ denotes the closed $M(\lambda, \eps)$-neighbourhood of $\gamma$ in $X$.
A geodesic is called Morse if it is $M$-Morse for some Morse gauge $M$.
\end{definition}

\begin{remark}
Note that in the usual definition all quasi-geodesics (and not only continuous ones) have to stay close. Thanks to the the taming quasi-geodesics Lemma (Lemma III.H.1.11 \cite{hyperbolic_spaces}) our definition is equivalent to the usual one in the sense that if any quasi-geodesic is $M$-Morse for one of the definitions, it is $M'$ in the other, where $M'( \lambda, \eps) = M(\lambda, 2(\lambda+\eps) )+\lambda+\eps$. 
\end{remark}

\textbf{Notation:} For the rest of the paper, unless stated otherwise, $M, M', M_1$ etc. will denote Morse gauges. We will also write $M(C)$ instead of $M(C, C)$. We define $\delta_M = \max\{4 M (1, 2M(5, 0) ) + 2 M (5, 0), 8 M (3, 0) \}.$ The constant $\delta_M$ can be thought of as the hyperbolicity constant of $M$-Morse subsets. For example, Lemma 2.2 of \cite{Cor16} shows that every triangle where at least two of the sides are $M$-Morse is $\delta_M$-slim. Since we only consider increasing Morse gauges we have that $\delta_{M_1}\leq \delta_{M_2}$ for all $M_1, M_2\in \mc M$ with $M_1\leq M_2$ . 

\begin{definition}[Morse points] \label{definition:morse_elements}Let $x\in X$ be a point, we say that $x$ is $M$-Morse if there exists a geodesic $\gamma$ from the basepoint $e$ to $x$ which is $M$-Morse. Similarly, if $G$ is a group with fixed generating set and $g\in G$ is an element, we say that $g$ is $M$-Morse if there exists an $M$-Morse geodesic from $1$ to $g$. 
\end{definition}

Note that in other literature, the term ``Morse element'' for an element $g\in G$ is used to say that the axis of $g$ is a Morse quasi-geodesic. So being a ``Morse element'' has several consequences for example having infinite order. The definition of ``Morse element'' in the literature differs drastically from an element being $M$-Morse as defined above. In particular, with the definition we use, every element $g\in G$ is $M$-Morse for some Morse gauge $M$. 

\begin{definition}[Morse boundary as a set]
As a set, the Morse boundary of $X$, denoted by $\m X$, is the set of equivalence classes of Morse geodesic rays, where two Morse geodesic rays are equivalent if they have bounded Hausdorff distance. Elements $z\in \m X$ are called Morse directions. For a Morse geodesic ray $\xi$ we denote its equivalence class by $\C{\xi}\in \m X$.
\end{definition}

\begin{definition}[Realisation]
Let $z\in \mb X$ be a Morse direction and $a\in X$. A geodesic ray $\gamma : [0, \infty) \to X $ with $\C{\gamma} = z$ and $\gamma(0) = a$  is called an $a$-realisation of $z$. Similarly, for an element $x\in X$ a geodesic segment $\gamma$ from $a$ to $x$ is called an $a$-realisation of $x$. For a Morse direction $z\in \mb X$, $\bar{z}_a$ denotes a choice of an $a$-realisation of $z$. 
\end{definition}

A realisation of $x\in X\cup \mb X$ is an $e$-realisation of $x$, where $e$ is the basepoint of $X$. We denote by $\P{M}{X}_a\subset X$ the set of all points $x\in X$ that have an $M$-Morse $a$-realisation and we say that points $x\in \P{M}{X}_e$ are $M$-Morse. Note that $ X = \cup_{M\in \mathcal{M}} \P{M}{X}_a$. Similarly, we define the set of all $M$-Morse directions $\M{M}{X}$ as the set of all Morse directions $z\in\m X$ that have an $M$-Morse realisation.

Let $a\in X$, we denote by $\mc G_a^M X$ the set of all $M$-Morse geodesic rays in $X$ starting at $a$. We denote by $\tilde{\mc G}_a^M X$ the set of all $M$-Morse geodesics (that is rays or segments) starting at $a$. Let $\xi$ be an $M$-Morse realisation of $z$ starting at $a$. For any positive $n$ such that $l(\xi)\geq n$ we define
\begin{align}
    O^M_n(\xi, a) &= \{\eta \in \mc G_a^M X \mid d_X(\eta(t), \xi(t)) < \delta_M \text{ for all } t\in [0, n]\},\\
   \Oh{M}{n}{\xi, a} &= \{\eta \in \tilde {\mc G}_a^M X \mid d_X(\eta(t), \xi(t)) < \delta_M \text{ for all } t\in [0, n]\}.
\end{align}
We say that a Morse direction $y\in \mb X$ is in $\oh{M}{n}{\xi, a}$ (resp in $\Oh{M}{n}{\xi, a}$) if there exists an $a$-realisation $\gamma$ of $y$ which is in $\oh{M}{n}{\xi, a}$. For an $M$-Morse direction $z\in \mb^M X$ we define $O_n^M(z, a)$ as the union of all $O_n^M(\xi, a)$, for $M$-Morse $a$-realisations $\xi$ of $z$. We define $\tilde{O}_n^M(z, a)$ analogously. With this notation $\{\oh{M}{n}{\xi, a}\}_{n\in \N}$ is a fundamental system of neighbourhoods of $\mb ^M X$ for any $a\in X$ and $M\in \mc M$. 
\begin{remark}\label{remark:property_of_neighbourhoods}
If $M\leq N$ are Morse gauges, then $\tilde{O}_n^M(z, a)\subset \tilde{O}_n^N(z, a)$ and the same holds if we replace $z$ with a geodesic $\xi$. Furthermore, let $z, z'\in \mb X\cup X$ and let $\xi$ and $\zeta$ be $M$-Morse realisations starting at $a$ of $z$ and $z'$. If $z'\in \Oh{M}{n+18\delta_M}{z, a}$, then $\zeta\in \Oh{M}{n}{\xi, a}$. Indeed, there exist $M$-Morse realisations $\xi'$ and $\zeta'$ of $z$ and $z'$ starting at $a$ such that $\zeta'\in \Oh{M}{n+18\delta_M}{\xi', a}$. Corollary 2.15 of \cite{Cor16} implies that $\zeta'\in \Oh{M}{n+18\delta_M}{\zeta, a}$ and $\xi'\in\Oh{M}{n+18\delta_M}{\xi, a}$. Hence, using Lemma \ref{lemma:get_close_again} we get that $\zeta\in \Oh{M}{n}{\xi, a}$. 
\end{remark}

When $a$ is equal to the basepoint, we usually omit it in the notation. For example, we denote $\oh{M}{n}{\xi, e}$ by $\oh{M}{n}{\xi}$. 

The topology on $\mb X$ is the direct limit topology arising from 
\begin{align}
    \mb X = \lim_{\xrightarrow[\mathcal{M}]{}} \mb^M{X}.
\end{align}

The following Lemma summarises a list of properties of Morse quasi-geodesics that we will use throughout the paper. Most of the various properties are slight variations of well known properties and can be found in \cite{Charney_2019}, \cite{Cor16} and \cite{graph_of_groups1}. 

\begin{lemma}\label{lemma:monster}
There exist increasing functions $f_1: \mc M\to \mc M$, $f_2: \mc M\times \R_{\geq 1}\to \mc M$ and $D_1: M\times \R_{\geq 1}\to \R_{\geq 1}$ such that for any pointed geodesic metric space $(X, e)$ the following properties hold:
\begin{enumerate}[label = \roman*)]
    \item (Subsegments) Let $\gamma : I\to X$ be an $M$-Morse geodesic and $J\subset I$ an interval, then the restriction $\gamma|_J : J\to X$ is $f_1(M)$-Morse. Moreover, if $\gamma' : I\to X$ is an $M$-Morse $C$-quasi-geodesic and $J\subset I$ an interval, then the restriction $\gamma'|_J : J\to X$ is $f_2(M, C)$-Morse.\label{prop:subsegments} 
    \item (Triangles) Let $(\alpha, \beta, \gamma)$ be a geodesic triangle, where the sides are infinite or finite and let $\alpha$ and $\beta$ be $M$-Morse. Then $\gamma$ is $f_1(M)$-Morse. Furthermore, if $(\alpha, \beta, \gamma)$ is a $C$-quasi-geodesic triangle and if $\alpha$ and $\beta$ are $M$-Morse, then $\gamma$ is $f_2(M, C)$-Morse.\label{lemma:triangles}
    \item ($n$-gons) Let $n\geq 2$ and let $(\gamma_1, \ldots, \gamma_n)$ be an $n$-gon with geodesic sides $\gamma_i$. That is, $\gamma_i$ is geodesic and $\gamma_i^+ = \gamma_{i+1}^-$ for all $i\leq n$, where we use the notation $\gamma_{n+1} = \gamma_1$. If $\gamma_i$ is $M$-Morse for all $i\leq n-1$, then $\gamma_n$ is $f_1^{n-2}(M)$-Morse.\label{cond:n_gon}
    \item (Concatenation of two) Let $\gamma_1$ and $\gamma_2$ be $M$-Morse quasi-geodesics with $\gamma_1^+ =\gamma_2^-$. If the concatenation $\gamma = \gamma_1\circ \gamma_2$ is a quasi-geodesic, then $\gamma$ is $f_1(M)$-Morse.\label{cond:concatenation2}
    \item (Multiple Concatenation) Let $\gamma_1, \gamma_2, \ldots, \gamma_n$ be $M$-Morse quasi-geodesics with $\gamma_i^+ = \gamma_{i+1}^-$ for all $1\leq i\leq n-1$. If the concatenation $\gamma = \gamma_1\circ\gamma_2\circ\ldots\circ \gamma_n$ is a quasi-geodesic it is $f_1^{n-1}(M)$-Morse. \label{prop:multiple_concatenation}
    \item (realisation properties) Let $z\in \mb^M X$ be an $M$-Morse direction. Any realisation $\gamma$ of $z$ is $f_1(M)$-Morse. More generally, let $\gamma_1$ and $\gamma_2$ be geodesic rays with $\gamma_1(0)=\gamma_2(0)$ and $[\gamma_1] = [\gamma_2]$. If $\gamma_1$ is $M$-Morse, then $\gamma_2$ is $f_1(M)$-Morse. Furthermore, let $\gamma_1'$ and $\gamma_2'$ be $C$-quasi-geodesic rays with $d_X(\gamma_1'(0),\gamma_2'(0))\leq C$ and $[\gamma_1']=[\gamma_2']$. If $\gamma_1'$ is $M$-Morse, then $\gamma_2'$ and any subsegment or subray of $\gamma_2'$ is $f_2(M, C)$-Morse.  \label{prop:realization1}
    \item (Choice of geodesic) Let $\gamma_1$ and $\gamma_2$ be geodesic segments with the same endpoints. If $\gamma_1$ is $M$-Morse, then $\gamma_2$ is $f_1(M)$-Morse. \label{prop:geodesic_segments}
    \item (Distance of realisations) Let $\gamma_1$ and $\gamma_2$ be $M$-Morse $C$-quasi-geodesics with $\gamma_1(0)=\gamma_2(0)$ and such that $[\gamma_1] = [\gamma_2]$. The Hausdorff distance $d_H(\gamma_1, \gamma_2)\leq D_1(M, C)$ and for every $t\geq 0 $ there exists $s\geq 0$  with $d_H(\gamma_1[0, t], \gamma_2[0, s])\leq D_1(M, C)$ and $d_X(\gamma_1(t), \gamma_2(s))\leq D_1(M, C)$.\label{lemma:close_representatives}
    \item (Morseness of neighbourhoods) Let $\gamma_1$ be an $M$-Morse quasi-geodesic. If $\gamma_2$ is a $C$-quasi-geodesic contained in a $C$-neighbourhood of $\gamma_1$ then $\gamma_2$ is $f_2(M, C)$-Morse. Moreover, if $\gamma_3$ is a finite $C$-quasi-geodesic segment with endpoints in $\mc N_C(\gamma_1)$ then $\gamma_3$ is $f_2(M, C)$-Morse. \label{prop:adapted2.5}
    \item (Quasi-geodesic lines) Let $\gamma_1$ be an $M$-Morse $C$-quasi-geodesic line. Any $C$-quasi-geodesic line $\gamma_2$ with endpoints the same endpoints as $\gamma_1$  is $f_2(M, C)$-Morse and contained in the $D_1(M, C)$-neighbourhood of $\gamma_1$.\label{prop:quasi-geodesic_lines}
\end{enumerate}
\end{lemma}

All the conditions in the above lemma have the following property in common. If they hold for some function $f_1$ (resp $f_2$ or $D_1$) then they also hold for $f_1'$ (resp $f_2'$ or $D_1'$) as long as the new function is larger.

Before proving Lemma \ref{lemma:monster} we show that it is enough to prove that for every property $P$ in \ref{prop:subsegments}- \ref{prop:quasi-geodesic_lines}, there exists a function $f_P$ for which property $P$ is satisfied (the functions $f_P$ might even depend on each other as long as the dependence is not circular). Then, we define $f'$ as the maximum of the functions $f_P$ and make sure that $f'(M)\geq M$ for all Morse gauges $M$. Since we did not require that the functions $f_P$ are increasing, $f'$ is not necessarily increasing. We can construct a function $f$ that is increasing and satisfies all the properties as follows. First we define a function $f''$ via $f''(M)(\lambda, \eps) = \inf_{M'\geq M}\{f'(M')(\lambda, \eps)\}+1$. With this construction, $f''$ is increasing and for every Morse gauge $M$, $f''(M)$ is also increasing. However, $f''(M)$ might not be continuous in the second coordinate. Thus we define $f(M)(\lambda, \eps) = \int_\eps^{\eps+1} f''(M)(\lambda, x)dx$. The function $f$ is increasing and for all Morse gauges $M\in \mc M$, $f(M)$ is an increasing Morse gauge which is continuous in the second coordinate. However, it is not clear that $f$ satisfies all the properties since it not necessarily satisfies that $f\geq f_P$ for every Property $P$. 

We will show that $f$ satisfies Property \ref{cond:concatenation2}, which, like many other properties, has the form ``if certain quasi-geodesics $\gamma_1$ and $\gamma_2$ are $M$-Morse, then another quasi-geodesic $\gamma$ is $f_{\ref{cond:concatenation2}}(M)$-Morse''.

Let $M\leq M'$ be Morse gauges. Observe that the set of $M$-Morse quasi-geodesics is a subset of the set of $M'$-Morse quasi-geodesics. Thus if $\gamma_1$ and $\gamma_2$ are $M$-Morse, they are also $M'$-Morse and hence $\gamma$ has to be $f_{\ref{cond:concatenation2}}(M')$-Morse. Since this is the case for every Morse gauge $M'\geq M$, we have that $\gamma$ is $f''(M)$-Morse and hence $f(M)$-Morse.

We can argue analogously that $f$ satisfies all the other properties.

\textbf{Convention: } We assume that every increasing function $f : \mc M\to \mc M$ satisfies that $f(M)\geq M$. 

If we have to prove that increasing functions $f': \mc M\times \R\to \mc M$, $D: \mc M\times \R\to \R$ or similar functions exist, we can argue analogously. Again we assume that $D(M, C)\geq C$ and $f'(M, C)\geq M$. 

\begin{proof}
We need to use some of the properties to prove the others. Therefore, we denote by $f_{1, P}$, $f_{2, P}$ and $D_{1, P}$ functions that work for property $P$. To ease notation, $f_1$, $f_2$ and $D_1$ will denote $f_{1, P}$, $f_{2, P}$ and $D_{1, P}$ respectively while proving property $P$. 

\ref{prop:subsegments} Note that it is enough to prove the moreover part since it implies the rest of the statement. We prove it in the case that $I$ and $J$ are finite intervals, but the proof if they are not works analogously. Let $I = [0, c]$ and $J = [a, b]$. Let $\lambda: [0, d]\to X$ be a $(C', K')$-quasi-geodesic with endpoints on $\gamma'[a, b]$. Since $\gamma'$ is $M$-Morse, $\lambda$ is contained in the $K_1= M(C', L')$-neighbourhood of $\gamma'$. Assume that $\lambda$ intersects the $K_1$-neighbourhood of $\gamma_1[0, a]$. Since we assume that quasi-geodesics are continuous the supremum $s$ such that $\lambda(s)$ is in the $K_1$-neighbourhood of $\gamma'[0, a]$ satisfies that $\lambda(s)$ is also in the $K_1$-neighbourhood of $\gamma'[a, c]$, say it is $K_1$-close to $\gamma'(t_0)$ and $\gamma'(t_1)$ for some $t_0\in [0 ,a]$ and $t_1\in [a, c]$. Since $\gamma'$ is a quasi-geodesic $t_1-a\leq t_1 -t_0\leq 2CK_1+C^2 = K_2$. Define $s'$ as the infimum such that $\lambda(s')$ is in the $K_1$-neighbourhood of $\gamma'[0, a]$. With the same reasoning as for the supremum we get that $\lambda(s')$ is in the $K_1$ neighbourhood of $\gamma'(t_1')$ for some $t_1'\in [a, c]$ such that $t_1'-a\leq K_2$ and hence $\abs{t_1 - t_1'}\leq K_2$. Thus $d_X(\gamma'(t_1), \gamma'(t'_1))\leq CK_2 +C^2 = K_3$ and also $d_X(\gamma'(a), \gamma'(t_1))\leq K_3$. Since $d_X(\lambda(s), \lambda(s'))\leq 2K_1+K_3$ and since $\lambda$ is a $(C', L')$-quasi-geodesic we have that $s-s'\leq C'(2K_1+K_3+L') = K_4$ and hence $\lambda[s', s]$ is in the $K_5 = C'(K_4+L')$-neighbourhood of $\lambda(s)$ and hence in the $K_6 = K_5+K_1 +K_3$ neighbourhood of $\gamma'(a)$. If we assume that $\lambda$ intersects the $K_1$-neighbourhood of $\gamma'[b, c]$ and define $\tilde s$ (respectively $\tilde{s}'$) as the infimum (respectively supremum) such that $\lambda(s')$ is in the $K_1$-neighbourhood of $\gamma'[b, c]$ we can argue analogously that $\lambda[\tilde {s}, \tilde{s}']$ is contained in the $K_6$-neighbourhood of $\gamma'(b)$. Hence $\lambda$ as a whole is contained in the $K_6 = K_6(C', L')$-neighbourhood of $\gamma'[a, b]$ and $\gamma'[a, b]$ is $f_2(M, C)$-Morse for the Morse gauge $f_2(M, C)(C', L') = K_6(C', L')$ for all $C', L'$. 

\ref{cond:concatenation2} Let $\lambda : [0, t]\to X$ be a $C$-quasi-geodesic with endpoints on $\gamma$. If both endpoints of $\lambda$ are either on $\gamma_1$ or $\gamma_2$, then $\lambda$ is contained in the $M(C)$-neighbourhood of $\gamma$. Thus assume that $\lambda(0)$ lies on $\gamma_1$ and $\lambda(t)$ lies on $\gamma_2$. Let $\lambda(s)$ be the closest point on $\lambda$ to $\gamma_1^+ = \gamma_2^-$. Lemma \ref{quasi-geodesic} implies that both $\lambda[0, s][\lambda(s), \gamma_1^+]$ and $[\gamma_2^-, \lambda(s)]\lambda[s, t]$ are $3C$-quasi-geodesics with endpoints on $\gamma_1$ and $\gamma_2$ respectively. Thus they are both contained in the $M(3C)$-neighbourhood of $\gamma$. Therefore, this property holds for $f_1(M)$ defined by $f_1(M)(K, C)= M(\max\{3K, 3C\})$.

\ref{prop:multiple_concatenation} This follows from \ref{cond:concatenation2} by induction on $n$.

\ref{prop:adapted2.5} The moreover part implies that every finite subsegment of $\gamma_2$ is $f_2(M, C)$-Morse, which implies that $\gamma_2$ is $f_2(M, C)$-Morse as a whole. Thus it is enough to prove the moreover part of \ref{prop:adapted2.5}. We know that there exists $p= \gamma_1(t_0)$ and $q = \gamma_1(t_1)$ such that $p$ and $q$ are $C$-close to the endpoints $\gamma_3^-$ and $\gamma_3^+$ of $\gamma_3$ respectively. We may assume that $t_0\leq t_1$. The concatenation $[\gamma_1(t_0), \gamma_3^-]\gamma_3[\gamma_3^+, \gamma_1(t_1)]$ is a $3C$-quasi-geodesic with endpoints on $\gamma_1[t_0, t_1]$. Property \ref{prop:subsegments} about subsegments implies that $\gamma_1[t_0, t_1]$ is $M_1 = f_{2 ,\text{\ref{prop:subsegments}}}(M, C)$-Morse, which implies that $\gamma_3$ is contained in a $K_1= M_1(3C)$-neighbourhood of $\gamma_1[t_0, t_1]$. Let $K_2 = C(2K_1+C) +1$. We will show that the Hausdorff distance between $\gamma_3$ and $\gamma[t_0, t_1]$ is at most $K_3 = CK_2+C+K_1$. Case 1: for every $t_0\leq s_0\leq s_1, \leq t_1$ with $s_1 - s_0 = K_2$ there exists $s\in [s_0, s_1]$ such that the distance of $\gamma_1(s)$ to $\gamma_3$ is at most $K_1$. In this case, for every $t\in [t_0, t_1]$ the distance of $\gamma_1(t)$ to $\gamma_3$ is at most $CK_2+C +K_1$, where $CK_2+C$ is the distance to a point $\gamma_1(t')$ which has distance at most $K_1$ to $\gamma_3$. Case 2: there exists $t_0\leq s_0\leq s_1\leq t_1$ with $s_1 - s_0 = K_2$ and such that the distance from $\gamma_1[s_0, s_1]$ to $\gamma_3$ is larger than $K_1$. We assume that quasi-geodesics are continuous. Thus, there exists a point $x$ on $\gamma_3$ which is in the $K_1$-neighbourhood of both $\gamma_1[t_0, s_0)$ and $\gamma_1(s_1, t_1]$. But, $s_1 - s_0 = K_2$, that is, $\gamma_1[t_0, s_0)$ and $\gamma_1(s_1, t_1]$ have distance larger than $2K_1$, a contradiction. Hence, this case can never occur and $\gamma_1[t_0, t_1]$ and $\gamma_3$ have Hausdorff distance at most $K_3$. Lemma 2.5 of \cite{CHAR14} is stated for CAT(0) spaces but as noted in \cite{Cor16} works for geodesic metric spaces. It implies that $\gamma_3$ is $M'$-Morse, where $M'$ only depends on $M_1$ and $K_3$. Hence, setting $f_2(M, C) = M'$ works. 

\ref{prop:realization1} Part one of the statement about geodesics follows directly from Lemma 2.23 of \cite{graph_of_groups1} and hence we only need to prove the furthermore part. For any $t\geq 0$, let $s_t\geq 0 $ be such that $\gamma_2'(s_t)$ is the closest point on $\gamma_2'$ to $\gamma_1'(t)$. Since the Hausdorff distance between $\gamma_1'$ and $\gamma_2'$ is finite, $s_t$ gets arbitrarily large for large enough $t$. Lemma \ref{quasi-geodesic} implies that the concatenation $\gamma_2'[0, s_t][\gamma_2'(s_t), \gamma_1'(t)]$ is a $3C$-quasi-geodesic and hence the concatenation $[\gamma_1'(0), \gamma_2'(0)]\gamma_2'[0, s_t][\gamma_2'(s_t), \gamma_1'(t)]$ is a $4C$-quasi-geodesic. Thus $\gamma_2'(s_t)$ is in the $K=M(4C)$-neighbourhood of $\gamma_1'$. Property \ref{prop:adapted2.5} shows that $\gamma_2'[0, s_t)$ is $M_1 = f_{2, \text{\ref{prop:adapted2.5}}}(M, K)$-Morse. Since $s_t$ gets arbitrarily large $\gamma_2'$ is $M_1$-Morse as a whole. By Property \ref{prop:subsegments} any subsegment or subray of $\gamma_2'$ is $f_{2, \text{\ref{prop:subsegments}}}(M_1, C)$-Morse.


\ref{lemma:close_representatives} The proof of \ref{prop:realization1} shows that $\gamma_2$ is in the $K=M(6C)$-neighbourhood of $\gamma_1$. Similarly, $\gamma_1$ is in the $K$-neighbourhood of $\gamma_2$. Thus for every $t\geq 0$ there exists an $s\geq 0$ such that $d_X(\gamma_1(t), \gamma_2(s))\leq K$. Property \ref{prop:subsegments} implies that both $\gamma_1[0, t]$ and $\gamma_2[0, s]$ are $M_1 = f_{2, \text{\ref{prop:subsegments}}}(M, C)$-Morse. In particular, the concatenation $\gamma_1[0, t][\gamma_1(t), \gamma_2(s)]$ is in the $M_1(K+C)$-neighbourhood of $\gamma_2[0, s]$. Conversely, the concatenation $\gamma_2[0, s][\gamma_2(s), \gamma_1(t)]$ is in the $M_1(K+C)$-neighbourhood of $\gamma_1[0, t]$. Thus, this property holds for $D_1(M, C) = M_1(K+C)$.

\ref{prop:geodesic_segments} This follows directly from \ref{prop:adapted2.5}.

\ref{prop:quasi-geodesic_lines} Let $t_1, t_2\in \R$ and let $s_1, s_2\in \R$ be such that $\gamma_2(s_1)$ and $\gamma_2(s_2)$ are closest points on $\gamma_2$ to $\gamma_1(t_1)$ and $\gamma_1(t_2)$ respectively. Since the Hausdorff distance between $\gamma_1$ and $\gamma_2$ is finite, Lemma \ref{quasi-geodesic} implies for far enough apart $t_1$ and $t_2$ that the concatenation $[\gamma_1(t_1), \gamma_2(s_2)]\gamma_2[s_1, s_2][\gamma_2(s_2), \gamma_1(t_2)]$ is a $3C$-quasi-geodesic. Hence it is contained the $K=M(3C)$-neighbourhood of $\gamma_1$. In particular, we can do this for any pair of far apart real numbers $t_1$ and $t_2$ and thus $\gamma_2$ as a whole is contained in the $D_1(M, C)= K$ neighbourhood of $\gamma_1$. Furthermore $\gamma_2$ is $f_{2, \ref{prop:adapted2.5}}(M, K)$-Morse by \ref{prop:adapted2.5}.

\ref{lemma:triangles} The first part of the statement follows directly from Lemma 2.4 of \cite{Charney_2019}. So we only have to prove the moreover part. Consider a geodesic triangle $(\alpha', \beta', \gamma')$ with the same endpoints as the quasi-geodesic triangle $(\alpha, \beta, \gamma)$. In the case that $\alpha$ or $\beta$ are finite segments, Property \ref{prop:adapted2.5} shows that $\alpha'$ or $\beta'$ respectively are $M_1 = f_{2, \text{\ref{prop:adapted2.5}}} (M, C)$-Morse. In the case that $\alpha$ or $\beta$ are rays, $\alpha'$ or $\beta'$ respectively are $M_2 =f_{2, \text{\ref{prop:realization1}}}(M, C)$-Morse by Property \ref{prop:realization1}. If $\alpha$ or $\beta$ are quasi-geodesic lines, then $\alpha'$ or $\beta'$ respectively are $M_3 = f_{2,\text{\ref{prop:quasi-geodesic_lines}}}(M, C)$ by Property \ref{prop:quasi-geodesic_lines}. Hence in any case, $\alpha'$ and $\beta'$ are $M_4 = \max\{M_1, M_2, M_3\}$-Morse. By the first part of \ref{lemma:triangles} about geodesic triangles, the geodesic $\gamma'$ is is $M_5$-Morse, where $M_5$ depends only on $M_4$. Depending on whether $\gamma'$ is a segment, ray or line we can again use Property \ref{prop:adapted2.5}, \ref{prop:realization1} or \ref{prop:quasi-geodesic_lines} to get that $\gamma$ is $f_2(M, C) = M_6$-Morse, where $M_6$ only depends on $M_5$ and $C$. 

\ref{cond:n_gon} For $n=2$ the statement follows from Property \ref{prop:adapted2.5}, \ref{prop:realization1} or \ref{prop:quasi-geodesic_lines} depending on whether $\gamma_1$ is a segment, ray or line. For $n = 3$ the statement follows directly from \ref{lemma:triangles}. For $n\geq 4$ it follows by induction; If $\gamma_1^-$ and $\gamma_2^+$ are equal we can remove $\gamma_1$ and $\gamma_2$ from the $n$-gon and apply the statement to the remaining $(n-2)$-gon to show that $\gamma_n$ is $f_1^{n-4}(M)$ and hence $f_1^{n-2}(M)$-Morse. Otherwise let $\gamma$ be a geodesic with endpoints $\gamma_1^-$ and $\gamma_2^+$. We can use the statement for $n=3$ to get that $\gamma$ is $f_1(M)$-Morse. Now we can apply the statement to the $(n-1)$-gon $(\gamma, \gamma_3, \ldots, \gamma_n)$ and get that $\gamma_n$ is $f_1^{n-3}(f_1(M)) = f_1^{n-2}(M)$-Morse.
\end{proof}

\begin{lemma} \label{lemma:unbounded_crossing}
Let $\gamma_1$ and $\gamma_2$ be quasi-geodesic rays. If $\gamma_1$ is Morse, then $[\gamma_1]=[\gamma_2]$ if and only if there exists a constant $C$ such that $\mc N_C(\gamma_1)\cap \gamma_2$ is unbounded.
\end{lemma}

\begin{proof}
Let $K$ be the distance between $\gamma_1(0)$ and $\gamma_2(0)$ and assume that $\gamma_2$ and $\gamma_1$ are $C'$-quasi-geodesics. If $[\gamma_1] = [\gamma_2]$, then their Hausdorff distance is bounded, say it is at most $C$. Hence, $\mc N_C(\gamma_1)$ contains $\gamma_2$ and is therefore unbounded. On the other hand, if there exists a constant $C$ such that $\mc N_C(\gamma_1)$ is unbounded, then there exist arbitrarily large $t\in \R$ such that $d_X(\gamma_1(t), \gamma_2)\leq C$. For any such $t\in \R$ Let $\gamma_2(s_t)$ be a closest point on $\gamma_2$ to $t$. The concatenation $\eta = [\gamma_1(0), \gamma_2(0)]\gamma_2[0, s_t][\gamma_2(s_t), \gamma_1(t)]$ is a $(C'+K+C)$-quasi-geodesic, and hence $\gamma_2[0, s_t]$ is contained in the $K_1 = M(C'+K+C)$-neighbourhood of $\gamma_1$ and $M'$-Morse for some Morse gauge $M$ by Lemma \ref{lemma:monster}\ref{prop:adapted2.5}. Note that $s_t$ is arbitrarily large for arbitrarily large $t$. As a consequence, the whole ray $\gamma_1$ is contained in the $K_1$-neighbourhood of $\gamma_2$ and $M'$-Morse. The concatenation $[\gamma_2(0), \gamma_1(0)]\gamma_1[0, t][\gamma_1(t), \gamma_2(s_t)]$ is also a $(C'+K+C)$-quasi-geodesic. Hence, $\gamma_1[0, t]$ is contained in the $K_2 = M'(C'+K+C)$-neighbourhood of $\gamma_2$. Since again $t$ can be arbitrarily large, the whole ray $\gamma_1$ is contained in the $K_2$-neighbourhood of $\gamma_2$. Hence, $[\gamma_1] = [\gamma_2]$.
\end{proof}

We will use the following lemma from \cite{Cor16} various times in this paper. The proof of the lemma does not use the fact that $\alpha$ and $\beta$ are rays. Thus we can and will also use it if $\alpha$ and $\beta$ are geodesic segments.

\begin{lemma}[Cor 2.5 of \cite{Cor16}]\label{lemma:get_close_again}
Let X be a geodesic metric space and $\alpha : [0, \infty) \to X$ be an $M$-Morse geodesic ray. Let $\beta: [0, \infty)\to X$ be a ray such that $\beta(0) = \alpha(0)$ and $d_X(\alpha(t), \beta(t)) < K $ for all $t\in [0, D]$ for some $D\geq 6K$. Then $d_X(\alpha(t), \beta(t)) < \delta_M$ for all $t\in [0, D - 2K]$.
\end{lemma}

There are two distinct definition of the relative Morse boundary \cite{fioravanti2022connected} and \cite{AnnettesThesis}, which agree as a set, but use a different topology. We use the one from \cite{fioravanti2022connected}, but only talk about the relative Morse boundary in case it is empty and thus the topology is not important for us. 

\begin{definition}[Relative Morse boundary] Let $H\leq G$ be finitely generated subgroups and let $H$ be undistorted in $G$. The relative Morse boundary of $H$ in $G$, denoted $(\mb H, G)$, is the subset of $\mb  H$ consisting of points represented by rays that remain Morse in the Cayley graphs of $G$.
\end{definition}

\subsection{Relative Hyperbolicity}

Relative hyperbolicity can be thought of as a generalisation of hyperbolicity. Roughly speaking, a group $G$ (space $X$) is relatively hyperbolic to a collection of subgroups (subspaces) $\mc P$ if $G$ ($X$) is hyperbolic after ``collapsing'' all elements of $\mc P$. A precise definition can be found in \cite{bowditch2012relatively}. There are many different but equivalent definitions of hyperbolic groups (and spaces). In this paper, we do not work with any definition in particular but use the following consequences of relative hyperbolicity detailed below. 

The following lemma about projections in relatively hyperbolic groups summarises results of \cite{sisto_rel_hyp_proj}. It is the main tool about relatively hyperbolic groups we will use.

\begin{lemma}[Lemma 1.13 and Theorem 2.14 of \cite{sisto_rel_hyp_proj}]\label{lemma:rel_hyp_projection}
Let $X$ be a geodesic metric space hyperbolic relative to a collection of subset $\mc P$ and let $\{\pi_P : X\to P\mid P \in \mc P\}$ be a collection of closets point projections. There exists a constant $C_1$ and an increasing
function $D_2 : \R_{\geq 1}\to \R_{\geq 1}$ such that the following conditions hold. 
\begin{enumerate}
    \item For all $P, Q\in\mc P$ with $P\neq Q$, $\mathrm{diam}(\pi_P(Q))\leq C_1$; 
    \item For all $P\in \mc P$, $x\in X$ and $p\in P$, $d_X(p, \pi_P(x))\leq C_1+ (d_X(x, p) - d_X(x, P))$. In particular, if $d_X(x, p) = d_X(x, P)$, then $d_X(p, \pi_P(x))\leq C_1$.
    \item For all $P, Q\in\mc P$ with $P\neq Q$ and for every geodesic $\lambda$ from $P$ to $Q$, $d_X(\pi_Q(P),\lambda)\leq C_1$
    \item For all $P, Q\in\mc P$ with $P\neq Q$ and for every $C$-quasi-geodesic $\lambda$ from $P$ to $Q$, $d_X(\pi_Q(P),\lambda)\leq D_2(C)$.
\end{enumerate}
\end{lemma}

The following Theorem is a result of \cite{dahmani2003combination} and is a combination result about relative hyperbolicity. Relevant for this paper are (2) and (3').

\begin{theorem}[Theorem 0.1 of \cite{dahmani2003combination}]\label{dah03}
The following hold.
\begin{enumerate}
    \item Let $\Gamma$ be the fundamental group of an acylindrical finite graph of relatively
hyperbolic groups, whose edge groups are fully quasi-convex subgroups of the
adjacent vertices groups. Let $\mc G$ be the family of the images of the maximal
parabolic subgroups of the vertices groups, and their conjugates in $\Gamma$. Then,
$(\Gamma, \mc G)$ is a relatively hyperbolic group.
    \item Let $G$ be a group which is hyperbolic relative to a family of subgroups $\mc G$,
and let $P$ be a group in $\mc G$. Let $A$ be a finitely generated group in which $P$
embeds as a subgroup. Then, $\Gamma = A \ast_P G$ is hyperbolic relative to the family
$(\mc H \cup \mc A )$, where $\mc H$ is the set of the conjugates of the images of elements of $G$ not conjugated to $P$ in $G$, and where $\mc A$ is the set of the conjugates of $A$ in $\Gamma$.
\item Let $G_1$ and $G_2$ be relatively hyperbolic groups, and let $P$ be a maximal
parabolic subgroup of $G_1$, which is isomorphic to a parabolic (not necessarily
maximal) subgroup of $G_2$. Let $\Gamma = G_1 \ast_P G_2$. Then $\Gamma$ is hyperbolic relative to the family of the conjugates of the maximal parabolic subgroups of $G_1$, except
$P$, and of the conjugates of the maximal parabolic subgroups of $G_2$.
\item [(3')] Let $G$ be a relatively hyperbolic group and let $P$ be a maximal parabolic
subgroup of $G$ isomorphic to a subgroup of another parabolic subgroup $P'$ 
not conjugate to $P$. Let $\Gamma = G\ast_P$ according to the two images. Then $\Gamma$ is
hyperbolic relative to the family of the conjugates of the maximal parabolic
subgroups of $G$, except $P$ (but including the parabolic group $P'$)
\end{enumerate}
\end{theorem}

\subsection{Graphs of groups}\label{section:graph_of_group}

We now recall Bass-Serre theory.

\begin{definition}[Graph] A graph $\Gamma = (V, E)$ consists of a set of vertices $V$ and a set of edges $E$ together with maps 
\begin{align}
    E &\to V\times V\\
    \alpha&\mapsto (\alpha^-, \alpha^+)
\end{align}
and 
\begin{align}
    E&\to E\\
    \alpha&\mapsto \bar{\alpha}
\end{align}
satisfying $\bar{\bar{\alpha}} = \alpha$, $\bar{\alpha}\neq\alpha$ and $\bar{\alpha}^+=\alpha^-$ for all edges $\alpha\in E$. For an edge $\alpha$ we call $\alpha^-$ the source, $\alpha^+$ the target and $\bar{\alpha}$ the reverse edge of $\alpha$. If it is not clear which graph we are talking about, we will write $V(\Gamma)$ and $E(\Gamma)$ instead of $V$ and $E$.
\end{definition}

If we have a set of edges $A$ we denote $\{\alpha^-\mid \alpha\in A\}$ by $A^{-}$ and define $A^+$ accordingly. Given a subset of vertices $U\subset V$ the induced subgraph of $U$, denoted by $[U]$ is the subgraph of $\Gamma$ with vertex set $U$ and edge set $F = \{\alpha\in E\mid \alpha^+,\alpha^-\in U\}$. When it is clear that we are considering a subgraph and not a set of vertices, we will denote $[U]$ simply by $U$. An edge $\alpha$ is incident to a vertex $v$ if $v\in \{\alpha^{-}, \alpha^{+}\}$.\\

For a graph $\Gamma$ we denote its path metric (where every edge has length 1) by $d_{\Gamma}$. In the following definition one does not usually require that vertex and edge groups are finitely generated but it suffices for our purposes. 

\begin{definition}
A graph of groups $\mc{G}$ is a finite connected graph $\Gamma = (V, E)$ together with finitely generated vertex groups $\{G_v\mid v\in V\}$, finitely generated edge groups $\{\tilde{H}_\alpha\mid \alpha\in E\}$ and injective morphisms $f_\alpha: \tilde{H}_\alpha\to G_{\alpha^+}$ such that $\tilde{H}_\alpha = \tilde{H}_{\bar{\alpha}}$ for all $\alpha\in E$.
\end{definition}
We denote $ f_\alpha(\tilde{H}_\alpha)$ by $H_\alpha$.
Given a graph of groups $\mc{G}$ we can define 
\begin{align}
    \mc{FG} =\left(  \ast_{v\in V} G_v\right) \ast\left(\ast_{\alpha\in E} \langle t_\alpha\rangle \right).
\end{align}

Given a spanning tree $ST$ of $\Gamma$ we obtain the fundamental group of $\mc{G}$, denoted by $\pi_1(\mc{G})$ by adding the following relations to $\mc{FG}$:
\begin{itemize}
    \item $t_{\bar{\alpha}} = t_\alpha^{-1}$ for all edges $\alpha\in E$.
    \item $t_\alpha = 1$ for all edges $\alpha\in ST$.
    \item $f_{\bar{\alpha}}(h) = t_{\alpha} f_{\alpha}(h) t_\alpha ^{-1}$ for all $\alpha\in E$ and $h\in \tilde{H}_\alpha$.
\end{itemize}

With this definition, for any $v\in V$, the group $G_v$ is a subgroup of $\pi_1(\mc G)$. Edge groups $\tilde{H}_\alpha$ are technically not a subgroup of $\pi_1(\mc G)$, however their image $H_\alpha$ under $f_\alpha$ is. So for all $\alpha\in E$ we will call the groups $H_\alpha$ edge groups. Whenever we have to view edge groups as a subset of $\pi_1(\mc G)$ we consider the groups $H_\alpha$. So for example, in Condition (1) of Theorem \ref{cor:rel_hyp_graph_of_groups}, we require that the groups $H_\alpha$ be undistorted, wide and have infinite index in $\pi_1(\mc G)$. 

We define the Bass-Serre space $X$ associated to $\mc{G}$ as follows. For every edge $\alpha$ fix a finite symmetric generating set $\tilde{S}_\alpha = \tilde{S}_{\bar{\alpha}}$ of $\tilde{H}_\alpha$. For vertices $v$ fix finite symmetric generating sets $S_v$ of $G_v$ such that $S_\alpha =f_\alpha(\tilde{S}_\alpha)\subset S_{\alpha^+}$ for all edges $\alpha\in E$. The Bass-Serre space $X$ is the graph with vertex set $V(X) = \{(g, v)\mid g\in \pi_1(\mc{G}), v\in V(\Gamma)\}$ and edges from $(g, v)$ to $(gs, v)$, labelled by $s$ and from $(g, \alpha^-)$ to $(g t_{\alpha},  \alpha^+)$ labelled by $\alpha$ for all $v\in V(\Gamma)$, $s\in S_v$, $\alpha\in E(\Gamma)$ and $g\in \pi_1(\mc G)$. For an element $x =(g, v)\in X$ we define $\und x= g$.

We denote the Bass-Serre tree associated to $\pi_1(\mc G)$ by $T$. Its vertices are cosets of vertex groups in $\pi_1 (\mc G)$. We denote the canonical projection from $X$ to the Bass-Serre $T$ defined on vertices by $(g, v) \mapsto gG_v$ by $\pi$. This projection uniquely extends to the edges of $X$ in the following sense. If an edge $\beta$ of $E(X)$ has $\pi(\beta^+) = \pi(\beta^-) = v$, then we define $\pi(\beta) = v$. On the other hand, if $\pi(\beta^-) = v_1\neq \pi(\beta^+) = v_2$, then there exists a (unique) edge $\alpha\in E(T)$ from $v_1$ to $v_2$ and we define $\pi(\beta) = \alpha$. If $\alpha\in E(T)$, then $\pi^{-1}(\alpha)$ is a collection of edges of $X$. In particular, $\pi^{-1}(\alpha)^+$ (and $\pi^{-1}(\alpha)^-$) are ``edge groups''. On the other hand, for a vertex $v\in V(T)$, $\pi^{-1}(v)$ is a subgraph of $X$ which correspond to a vertex group. So for an edge $\alpha\in E(T)$ there is a big difference between $\pi^{-1}(\alpha^+)$ (which is a ``vertex group'') and $\pi^{-1}(\alpha)^+$ (which is an ``edge group'').

Similarly, we have a projection $\hat\cdot : X\to \Gamma$ defined by $(g, v)\mapsto v$, edges of the form $(g, \alpha^-)$ to $(g t_{\alpha},  \alpha^+)$ map to $\alpha$ and edges of the form $(g, v)$ to $(gs, v)$ map to $v$. If we denote by $\check\cdot$ the canonical projection from $T$ to $\Gamma$ we see that the maps commute, that is $\check\cdot \circ \pi = \hat\cdot$. In particular, since the image of both maps is $\Gamma$ they can both be simply thought of as projections to $\Gamma$.

Observe that $\pi_1(\mc G)$ acts on $X$ and $T$ by left multiplication and the maps $\pi, \hat\cdot$ and $\check\cdot$ are equivariant with respect to these actions. We allow vertices of $X$ to act on $X$ in the following sense. If $w = (g, v), w'\in V(X)$ and $\beta\in E(X)$ we define $w\cdot w'$ (respectively $w\cdot \beta)$ as $g\cdot w'$ (respectively $g\cdot \beta$).

There are canonical inclusions $Cay(gG_v, S_v)\into X$ induced by $g\mapsto (g, v)$, which induce a graph isomorphism between $Cay(gG_v, S_v)$ and $[\pi^{-1}(gG_v)]$. Similarly, for an edge $\alpha$ of $T$ with $v = \check\alpha^+$, the inclusion $g\mapsto (g, v)$ induces an isomorphism between $Cay(gH_{\check\alpha}, S_{v}\cap gH_{\check\alpha} )$ and $[\pi^{-1}(\alpha)^+]$. 
In the following, when it is clear which coset we are working with, we will identify elements of $gG_v$ and $H_{\check{\alpha}}$ with their image in $X$.

We will need the following notation to perform inductions on graph of groups.

\begin{definition}
Let $\mc G$ be a graph of groups with underlying graph $\Gamma$. Let $Y$ be a connected subgraph of $\mc G$. 

\begin{itemize}
    \item The graph of groups $\mc G$ restricted to $Y$, denoted by $\mc G|_Y$, is the graph of groups with underlying graph $Y$ and the rest of the data induced by $\mc G$.
    \item We define $\mc G_Y$ as follows: its underlying graph $\Gamma_Y$ is obtained by taking $\Gamma$ and collapsing $Y$ to a single point $y$. The vertex group $G_y $ is defined as $\pi_1(\mc G|_Y)$ and the other data is induced by $\Gamma$.
    \item Let $Y_1, Y_2$ be disjoint connected subgraphs of $\Gamma$, we define $\mc G _{Y_1, Y_2}$ as $(\mc G _ {Y_1})_{Y_2}$ and for a collection of disjoint connected subgraphs $Y_1, \ldots , Y_n\subset \Gamma$, we define $\mc G_{Y_1, \ldots, Y_n}$ accordingly and sometimes write $\mc G_{\{Y_i\}}$.
\end{itemize}

\end{definition}

Note that there exists a connection between $\pi(\mc G)$ and $\pi(\mc G_Y)$ described in the Lemma below. The connection follows directly from the definition of graph of groups.

\begin{lemma}\label{lemma:graph_of_groups_trivia} For every connected subgraph $Y\subset \Gamma$, $\pi_1 (\mc G) = \pi_1(\mc G_Y)$. In particular, for a collection of disjoint connected subgraphs $Y_1, \ldots , Y_n$ of $\Gamma$, $\pi_1(\mc G) =  \pi_1(\mc G_{Y_1, \ldots, Y_n})$
\end{lemma}

\section{Morseless stars and their Morse boundary}\label{section:morseless_stars}

When trying to prove extension results one can often prove the desired result for a free product and then use Papasuglu-Whyte \cite{papasoglu2002quasi}, which states that the fundamental group of two graph of groups with finite edge groups are quasi-isometric if both are infinitely ended and their sets of quasi-isometry types of (non virtually cyclic) vertex groups agree. In our case, this is not quite enough. This is why we introduce and study an object we call ``Morseless star''. 

\subsection{Definition and fundamental properties}\label{section:definitoin_of_morseless_stars}

We say a graph of groups is a \textbf{Morseless star} if 
\begin{itemize}
    \item The vertex set $V$ is equal to $\{0, \ldots n\}$ for some $n\in \N_0$ and every edge is incident to 0.
    \item Every edge group $H_\alpha$ is undistorted in $\pi_1(\mc{G})$.
    \item All edge groups are finitely generated, relatively wide and have infinite index in the corresponding vertex groups. 
    \item All vertex groups are finitely generated.
\end{itemize}

An undistorted subgroup $H\leq G$ is relatively wide in $G$, if for every asymptotic cone $G_\omega$, no
two points of the limit $H_\omega\subset G_\omega$ are separated by a cut point of $G_\omega$. We only use the following implications about relatively wide groups.
\begin{itemize}
    \item If a group $H$ is relatively wide in $G$, then the relative Morse boundary $(\mb H, G)$ is empty.
    \item Morseless stars satisfy the assumptions on graph of groups discussed in \cite{fioravanti2022connected} so we can use their results. In particular, we can use Theorem C, which states that the Morse boundary of every vertex group topologically embeds into the Morse boundary of the ambient group, Corollary 2.16, which states that vertex groups are undistorted, and Corollary 4.3, which states that a $M$-Morse geodesic in a vertex group is $M'$-Morse in the ambient group.
    \item If an undistorted subgroup $H\leq G$ is wide, then $H$ is relatively wide in $G$.
    \item The group $\Z^2$ is wide.
\end{itemize}

If the results in \cite{fioravanti2022connected} can be upgraded, that is, if the assumption of relatively wide can be replaced by the assumption of having relatively empty Morse boundary for their results, then the same holds for our results. 

Note that the underlying graph $\Gamma$ of a Morseless star does not have to be a star (in the graph theoretic sense) since we allow loops (based at 0) and double edges. We call such graphs $\Gamma$ which we allow to be the underlying graph \emph{fake stars}.

There are two subclasses of Morseless stars we are most interested in. Firstly, if all edge groups of a Morseless star are trivial, we call it a \textbf{trivial Morseless star}. Note that any graph of groups whose underlying graph is a fake star, whose vertex groups are finitely generated and infinite and whose edge groups are trivial is a Morseless star and hence a trivial Morseless star.


Secondly, we say a Morseless star is \textbf{relatively hyperbolic} if the vertex group $G_0$ is hyperbolic relative to a collection of subgroups $\mc A\sqcup\{H_{\alpha^+}\mid \alpha \in E(\Gamma), \alpha^+=0\}^{G_0}$. We require that (unless finite) the subgroups $H_{\alpha^+}$ from above are not conjugate, while $\mc A$ can be any collection of subgroups disjoint from $\{H_{\alpha^+}\mid \alpha \in E(\Gamma), \alpha^+=0\}^{G_0}$. Here, if $\mc H$ is a set of subgroups of a group $G$, $\mc H^G$ denotes the set of conjugates of subgroups of $\mc H$ by elements of $G$.\\

In this section we study the Morse boundary of certain Morseless stars and describe their Morse boundary combinatorially. As in \cite{graph_of_groups1}, the combinatorial structure we describe is an analogue to the combinatorial structure used in \cite{martin2013hyperbolic}. In \cite{graph_of_groups1} we could concatenate geodesics of different vertex groups without loosing control of the Morse gauge due to cut points between vertex groups. The control of the Morse gauge when concatenating geodesics is an important part to show the combinatorial structure represents the Morse boundary. It is why we have to restrict our study to a class of Morseless stars we call Morse preserving (see Definition \ref{definition:morse-preserving}), that is, those Morseless stars where we can concatenate geodesics without losing control of the Morse gauge. We prove that the two subclasses of Morseless stars we care about, namely trivial and relatively hyperbolic ones, are Morse preserving. \\

\textbf{Convention:} For the rest of the paper, unless stated otherwise, $\mc{ G}$ is a Morseless star with the notation from Section \ref{section:graph_of_group}. We say that the basepoint of the Bass-Serre space $X$ is $(1, 0)$ and denote it by $e$. We view $G_0$ as the basepoint of the Bass-Serre tree $T$.
Furthermore, unless stated otherwise, we assume neighbourhoods and other properties are induced by the metric $d_X$. That is, for a subgraph $Y$ of the Bass-Serre space $X$ and $x, y\in Y$, $[x, y]$ denotes a geodesic from $x$ to $y$ which is a geodesic with respect to the metric $d_X$, not the metric $d_Y$. Hence $[x, y]$ is not necessarily contained in $Y$. We denote a geodesic (with respect to the metric $d_Y$) from $x$ to $y$ in $Y$ by $[x, y]_{Y}$. Similarly we say that $\gamma$ is a $Y$-(quasi-)geodesic from $x$ to $y$ if it is a (quasi)-geodesic with respect to the metric $d_Y$. If we want to say that a $Y$-geodesic $\gamma$ is $M$-Morse in $Y$, we say that it is $(M, Y)$-Morse.

With this notation, we can study and state some fundamental properties about Morseless stars. 
\begin{lemma}[Properties of Morseless stars]\label{lemma:properties_of_morseless_stars}
Let $\mc G$ be a Morseless star. There exists a constant $C_2\geq 1$ and increasing functions $f_3 : \mc M\to \mc M$ and $D_3: \mc M\times \R_{\geq 1}\to \R_{\geq 1}$ such that for every edge $\alpha\in E(T)$, vertex group $Y=\pi^{-1}(\alpha^+)$ and edge group $W=[\pi^{-1}(\alpha)^+]$, the following properties hold:
\begin{enumerate}
    \item The inclusions $Y\into X, W \into Y$ and $W\into X$ are $C_2$-quasi-isometric embeddings.\label{lemma:undistorted_vertex_groups}
    \item Every $(M, Y)$–Morse $Y$-geodesic is an $f_3(M)$–Morse $C_2$-quasi-geodesic. Conversely, every $M$-Morse geodesic contained in $Y$ is an $(f_3(M), Y)$-Morse $C_2$-$Y$-quasi-geodesic.\label{lemma:vertex_group_morseness}.
    \item If $x, y\in W$ with $d_X(x, y)\geq D_3(M, 1)$ the geodesic $[x, y]$ is not $M$-Morse.\label{lemma:max_length}
    \item Let $\gamma$ be an $M$-Morse $C$-quasi-geodesic. The intersection $\mc N_C(\gamma) \cap \mc N_C(W)$ has diameter at most $D_3(M,C)$. \label{constant:edge_group_intersection2}
\end{enumerate}
\end{lemma}

\begin{remark}\label{remark:uniform_constants}
Unlike $f_1$, $f_2$ and $D_1$, the functions $D_3$ and $f_3$ as well as the constant $C_2$ depend on $\mc G$. If we have finitely many Morseless stars we can make sure that $f_3$, $D_3$ and $C_2$ work for all of them. The same holds for functions we define later.
\end{remark}

\begin{remark}
Property (\ref{constant:edge_group_intersection2}) implies that every Morse quasi-geodesic ray diverges from edge groups in the sense that for every Morse quasi-geodesic ray $\gamma : [0, \infty)\to X$, edge $\alpha\in T$ and $C>0$ there exists some $t>0$ such that for every $s\geq t$ we have $d_X(\gamma(s), \pi^{-1}(\alpha))\geq C$.
\end{remark}

\begin{proof}
As in the proof of Lemma \ref{lemma:monster} we just need to prove that for every Property $x$ there exist (not necessarily increasing) functions $f_{3, x}$ and $D_{3, x}$ such that Property $x$ holds for $f_{3, x}$ and $D_{3, x}$.

Morseless stars satisfy the assumptions on graph of groups discussed in \cite{fioravanti2022connected}. Property \eqref{lemma:undistorted_vertex_groups} and the second implication of Property \eqref{lemma:vertex_group_morseness} follow from Corollary 2.16 of\cite{fioravanti2022connected}. The first implication of Property \eqref{lemma:vertex_group_morseness} follows from Corollary 4.3 of \cite{fioravanti2022connected}, which is stated about geodesic rays, but the proof works for any geodesic, in particular geodesic segments. 

We prove (\ref{lemma:max_length}) by contradiction. Fix a Morse gauge $M$ and assume there exists a sequences of edges $(\alpha_k)$ and sequences of vertices $(x_k)$ and $(y_k)$ with  $x_k, y_k\in \pi^{-1}(\alpha_k)^+$ and $d_X(x_k, y_k)\geq k$ such that $[x_k, y_k]$ is $M$-Morse. Denote $\pi^{-1}(\alpha_1)^+=W'$. We can assume that $\hat{\alpha}_k =\hat\alpha_1$ for all $k$. Define $z_k = x_k^{-1}y_k$. We have that $x_k^{-1}[x_k, y_k] = [e, z_k]$ are $M$-Morse geodesics and are contained in $W'$. By Arzel\`a-Ascoli there exists a subsequence $(z_k')$ of $(z_k)$ such that $[e, z_k]_{W'}$ converge uniformly on compact sets to some $W'$-geodesic $\lambda_1$. Again by Arzel\`a-Ascoli there exists a subsequence $(z_k'')$ of $(z_k')$ such that $[e, z_k'']$ converges uniformly on compact sets to an $X$-geodesic $\lambda_2$. The Hausdorff distance $d_{Hauss}([e, z_k''], [e, z_k'']_{W})$ is bounded by some constant $D$, depending only on $M$ (see Lemma 2.1 \cite{Cor16}). Thus, the Hausdorff distance between $\lambda_1$ and $\lambda_2$ is bounded by $D+1$. By Lemma 2.10 of \cite{Cor16} $\lambda_2$ is $M$-Morse, which implies that $\lambda_1$ is Morse. This is a contradiction to the assumption that the edge groups have empty relative Morse boundary.

Next we prove Property (\ref{constant:edge_group_intersection2}). We denote by $D_{3, 3}: \mc M\times \R_{\geq 1}\to \R$ a function such that Property~(3) is satisfied. Let $\gamma$ be an $M$-Morse $C$-quasi-geodesic and let $x, y\in \mc N_C(\gamma)\cap\mc N_C(W)$. Let $x_1, y_1$ be points on $\gamma$ and $x_2, y_2$ be points on $W$ such that $d_X(x, x_j)\leq C$ and $d_X(y, y_j)\leq C$ for $j\in \{1, 2\}$. The geodesic $[x_2, y_2]$ has endpoints in the $2C$-neighbourhood of $\gamma$ and is thus $M_1 = f_2(M, 2C)$-Morse by Lemma \ref{lemma:monster} \ref{prop:adapted2.5}. By (3), $d_X(x_2 ,y_2)\leq D_{3, 3}(M_1, 1)$ and hence $d_X(x, y)\leq 2C+  D_{3, 3}(M_1, 1)$. Since $x$ and $y$ were chosen arbitrarily, the intersection $\mc N_C(\gamma) \cap \mc N_C(W)$ has diameter at most $D_{3, 4}(M,C) = 2C+  D_{3, 3}(M_1, 1)$.
\end{proof}

Next we use the results of \cite{dahmani2003combination} to study the relatively hyperbolic structure on $X$ if $\mc G$ is a relatively hyperbolic Morseless star.

\begin{definition}\label{def:subsetsassociated}
Let $\mc G$ be a relatively hyperbolic Morseless star. Let $v\in V(T)$ be a vertex which is not a coset of $G_0$, or in other words, where $\check v\neq 0$. Let $F\subset E(T)$ be the set of edges adjacent to $v$. The subset $P(v)\subset X$ associated to $v$ is $P(v) = \pi^{-1}(v)\cup \pi^{-1}(F)$.

Let $\alpha\in E(T)$ be an edge of $T$. We associate a subset $P(\alpha)\subset X$ to $\alpha$ as follows. 
\begin{align}
    P(\alpha) = \begin{cases}
    \pi^{-1}(\alpha) & \text{if $\check\alpha^+ = \check\alpha^-$ = 0,}\\
    P(v) & \text{Otherwise. Here $v$ denotes the endpoint of $\alpha$ with $\check v\neq 0$.}
    \end{cases}
\end{align}

\end{definition}

\begin{remark}\label{rem:p(alpha)are_the_same}
With this definition for every edge $\alpha\in T$ we have that $\pi^{-1}(\alpha)\subset P(\alpha)$. Furthermore if $P(\alpha)=P(\beta)$ for some edges $\alpha$ and $\beta$, then $\alpha$ and $\beta$ share at least one endpoint. Moreover, if $P(\alpha) = P(\beta)$ and they share exactly one endpoint $v$, then $\check v\neq 0$. 
\end{remark}

\begin{lemma}\label{lemma:realhypassoc}
Let $\mc G$ be relatively hyperbolic Morseless star. Then $X$ is relatively hyperbolic with respect to a collection of subsets $\mc P$ such that for every edge $\alpha\in E(T)$, $P(\alpha)\in \mc P$.
\end{lemma}

Lemma \ref{lemma:realhypassoc} follows from the following technical Lemma. Recall that $ST$ is the spanning tree used to define $\pi_1(\mc G)$ and that $\mc A$ is the same set as in the definition of a relatively hyperbolic star.

\begin{lemma}\label{lemma:peripherals}
Let $\mc G$ be a relatively hyperbolic Morseless star. Let $E(\Gamma)^+$ be an orientation of the edges of $\Gamma$, that is, let $E(\Gamma)^+$ be a subset of $E(\Gamma)$ such that for each edge $\alpha\in E(\Gamma)$ exactly one of $\alpha, \bar{\alpha}$ is in $E(\Gamma)$. The group $G = \pi_1(\mc G)$ is hyperbolic relative to $\mc P = \mc A^G\sqcup \mc H'^{G}$, where $\mc H' = \mc H_1'\cup \mc H_2'$ and
\begin{itemize}
    \item $\mc H_1 '  = \{G_{\alpha^-}\mid \alpha\in E(ST), \alpha^+ = 0\}$,
    \item $\mc H_2 ' = \{H_\alpha\mid \alpha\in E(\Gamma)^+, \alpha^+ = \alpha^- = 0\}$.
\end{itemize}
\end{lemma}
Recall that $\mc H'^G$ is the set of conjugates of elements of $\mc H'$ by elements of $G$. 

\begin{proof}
This follows from inductively applying Theorem \ref{dah03}. The induction base case is when $\mc G$ has no edges, and the result follows immediately. Assume the statement holds for every relatively hyperbolic Morseless star with $m$ edges and assume $\Gamma$ has $m+1$ edges. 

Case 1: $\Gamma$ is a star. Let $\beta\in E(\Gamma)$ be an edge with $\beta ^+ = 0$, and define $W = V(\Gamma) - \beta ^-$ and $Y = [W]\subset \Gamma$. The graph of groups $\mc G\mid_Y$ has $m$ edges and is, by the induction hypothesis, hyperbolic relative to $\tilde{\mc P}= (\mc A\cup \{H_{\beta}\}^{\pi_1(\mc G \mid _Y)}) \cup (\mc{H}_1'-\{G_{\beta^-}\}^{G})^{\pi_1(\mc G \mid _Y)}\cup \mc H_2'^{\pi_1(\mc G \mid _Y)}$. Then, by (2) Theorem \ref{dah03}, the group $G$ is hyperbolic relative to $(\tilde{\mc P}\cup G_{\beta^-} )^G- \{H_\beta\}^G = \mc A^G \cup \mc H_1'^G\cup \mc H_2'^G$.

Case 2: $\Gamma$ contains an edge $\beta\in E(\Gamma)^+$ with $\beta^+ = \beta^- = 0$. In this case, we define $Y = \Gamma- \{\beta, \bar{\beta}\}$. Then, $\mc G\mid_Y$ has $m$ edges and is, by the induction hypothesis, hyperbolic relative to $ \tilde{\mc P}= (\mc A \cup \{H_{\beta}, H_{\bar{\beta}}\})^{\pi_1(\mc G \mid _Y)}\cup \mc H_1'^{\pi_1(\mc G \mid _Y)}\cup (\mc H_2' - \{H_\beta, H_{\bar{\beta}}\})^{\pi_1(\mc G \mid _Y)}$. The underlying graph of $\mc G_Y$ is a loop, more precisely $\pi_1 (G_Y) = B*_{\tilde{H}_{\beta}}$. Both images $H_{\beta}$ and $H_{\bar{\beta}}$ of $\tilde{H}_{\beta}$ are in $\tilde{P}$. By ($3'$) of Theorem \ref{dah03}, $G = \pi_1(\mc G_Y)$ is hyperbolic relative to $\tilde P^G - \{H_{\bar\beta}\}^G = A^G\cup \mc H_1'^G\cup \mc H_2 '^G$.

Case 3: $\Gamma$ contains an edge $\beta\in E(\Gamma)\setminus  E(ST)$ with $\beta^-\neq 0$. In this case, we define $Y = \Gamma - \{\beta, \bar{\beta}\}$. The graph of group $\mc G\mid_Y$ has $m$ edges and is, by the induction hypothesis, hyperbolic relative to $ \tilde{\mc P}= (\mc A \cup\{H_{\beta}\})^{\pi_1(\mc G \mid _Y)}\cup \mc H_1'^{\pi_1(\mc G \mid _Y)}\cup \mc H_2'^{\pi_1(\mc G \mid _Y)}$. The underlying graph of $\mc G_Y$ is a loop, more precisely $G =\pi_1 (G_Y) = B*_{\tilde{H}_{\beta}}$. The image $H_\beta$ of $\tilde{H}_\beta$ is contained in $\tilde {\mc P}$ and the other image $H_{\bar\beta}$ is a subgroup of $G_{\beta^-}\in \tilde{\mc P}$. By ($3'$) of Theorem \ref{dah03}, $G$ is hyperbolic relative to $\tilde{\mc P}^G - \{H_{\beta}\}^G =A^G\cup \mc H_1'^G\cup \mc H_2 '^G$.
\end{proof}

\subsection{The combinatorial structure of the Morse boundary}

Similarly as in \cite{graph_of_groups1} and \cite{martin2013hyperbolic} we describe Morse geodesics by how they travel in the Bass-Serre tree and, if they eventually stay in a vertex group, what their tail in this vertex group is. The discussion gets a bit more technical since the projection of a geodesic in $X$ to the Bass-Serre tree might not be a geodesic, even after reparametrization. This is a consequence of the edge group embeddings being quasi-isometries but not necessarily isometries. Thus to give a combinatorial description of equivalence classes of Morse geodesics
in $X$ starting at $e$, we first need to set up some notation. 

Let $\alpha$ be an edge of the Bass-Serre tree $T$. We say the edge $\alpha$ is \textbf{outgoing} if $d_T(G_0, \alpha^-)<d_T(G_0, \alpha^+)$. Denote by $T_\alpha$ the connected component of $T\setminus \{\alpha\}$ that contains $\alpha^+$. The edge $\alpha$ is outgoing if and only if $T_\alpha$ does not contain $G_0$. Define $T_{G_0}$ as $T$ and for every vertex $v\in V(T)- \{G_0\}$ define $T_v$ as $T_{\alpha}$, where $\alpha$ is the outgoing edge with $\alpha^+ = v$. 

With this definition, for two distinct outgoing edges $\alpha$ and $\beta$ we have that $T_\alpha$ and $T_\beta$ are either disjoint or one is contained in the other, where the latter happens if and only if one of the edges lies on the path from $G_0$ to the other edge. Similarly, for two distinct vertices $v$ and $w$, $T_v$ and $T_w$ are either disjoint or one is contained in the other.

For every outgoing edge $\alpha$ we choose a vertex $x$ of $\pi^{-1}(\alpha)^+$ with $d_X(e, x) = d_X(e, \pi^{-1}(\alpha)^+)$ and denoted the chosen vertex by $\alpha^*$. We think of $\alpha^*$ as the ``basepoint'' of (coset of) the edge group $\pi^{-1}(\alpha)^+$.

If $\alpha$ is not outgoing, then $\bar\alpha$ is. Let $(g,i) =\bar\alpha^*$. We define $\alpha^* = (gt_\alpha, \check\alpha^+)$. With this definition, $\bar{\alpha}^*$ and $\alpha^*$ are connected by an edge $\beta$ with $\pi(\beta) = \alpha$. For every vertex $v\neq G_0\in V(T)$ there exists exactly one outgoing edge $\alpha\in E(T)$ with $\alpha^+ = v$. We define $v^*$ as $\alpha^*$ and $G_0^*$ as $e$. We think of $v^*$ as the ``basepoint'' of the (coset of) the vertex group $\pi^{-1}(v)$. 

We want to highlight the following properties of $v^*$ and $\alpha^*$.

\begin{lemma} Let $\alpha$ be an outgoing edge and let $v$ be a vertex of $T$.
\begin{enumerate}[label= \roman*)]
    \item The vertex $\bar{\alpha} ^*$  lies on the geodesic $[e ,\alpha^*]$.
    \item The following equalities hold $d_X(e, \alpha^*) = d_X(e, \pi^{-1}(\alpha)^+)$, $d_X(e, \bar{\alpha}^*) = d_X(e, \pi^{-1}(\alpha)^-)$ and $d_X(\pi^{-1}(v), e) =  d_X(v^*, e)$.
\end{enumerate}
\end{lemma}

Sometimes we think of an outgoing edge $\alpha$ with $\alpha^- = v$ as a connection from $v^*$ to $\alpha^*$. It is then natural to ask how Morse this connection and thus $\alpha$ is.

\begin{definition}[Morseness of the Bass-Serre tree]\label{definition:morseness_edges} Let $\alpha\in E(T)$ be an outgoing edge and $v = \alpha^-$. We say that $\alpha$ is $M$-Morse if the geodesic $[v^*, \alpha^*]$ is $M$-Morse.
\end{definition}

As described in \cite{graph_of_groups1}, in a graph of groups $\mc G$ with trivial edge groups, there are two types of $M$-Morse geodesics $\gamma$; Ones whose projection to the Bass-Serre tree (after reparametrisation) is an infinite geodesic $p$ and ones whose projection to the Bass-Serre tree (after reparametrisation) is a finite geodesic $p$ ending in some vertex $v$. What they both have in common is that the restriction to every vertex group is $M$-Morse. For Morseless stars, we replace ``the intersection of $\gamma$ with every vertex group has to be $M$-Morse" by ``the edges of the geodesic $p$ have to be $M$-Morse".

\begin{definition}[Combinatorial rays]
A combinatorial ray $r$ can be of finite or infinite type. It is defined as follows:
\begin{itemize}
    \item \textbf{Finite type:} A combinatorial ray $r$ of finite type is a sequence 
    \begin{align}
        r =  (\alpha_1, \alpha_2, \ldots, \alpha_n; z),
    \end{align}
    such that $\alpha_j = \alpha_j(r)$ are edges of $T$ and $(\alpha_1, \ldots, \alpha_n)$ is a geodesic path (potentially the empty path) in $T$ starting at $G_0$ and denoted by $p(r)$. We say $\omega= \omega(r)=p(r)^+$ is the end vertex of $r$ and require that $z$, also denoted by $z(r)$ and called the tail of $r$, is a Morse direction of $\partial_* G_{\check{\omega}}$. We say the length of $r$, denoted by $l(r)$, is $n$.

    \item \textbf{Infinite type:} A combinatorial ray $r$ of infinite type is an infinite sequence
    \begin{align}
        r =  (\alpha_1, \alpha_2, \ldots),
    \end{align}
    such that $\alpha_j = \alpha_j(r)$ are edges of $T$ and $(\alpha_1, \alpha_2, \ldots, )$ is a geodesic path starting at $G_0$ denoted by $p(r)$. We say the length of $r$, denoted by $l(r)$, is infinite. 

\end{itemize}

\end{definition}

By \cite{fioravanti2022connected} we know that for every vertex $i\in V(\Gamma)$ the Morse boundary of its vertex group $\mb G_i$ topologically embeds into $\mb X$. Thus, for a combinatorial ray $r$ of finite type, we can view $z(r)$ as an element of $\mb X$ and will do so. Let $w = G_i$ be a vertex of $T$. The realisation $\overline{z(r)}$ of $z(r)$ is not necessarily contained in $\pi^{-1}(w)$ but the following lemma shows that the intersection of $\overline{z(r)}$ and $\pi^{-1}(w)$ is unbounded. 

\begin{lemma}\label{lemma:vertex_group_intersection}
Let $i\in V(\Gamma)$ be a vertex and let $w = G_i$ be a vertex of $T$. Furthermore, let $z\in \mb G_i$ be a Morse direction and let $\gamma$ be a quasi-geodesic with $[\gamma] = z$. The intersection $\pi^{-1}(w)\cap \gamma$ is unbounded.
\end{lemma}

Note that we can also apply the lemma for translations of $z$ and $\gamma$. In particular, if $v\in V(T)$ is a vertex with $\check{v} = i$ and $\gamma'$ is a quasi-geodesic with $[\gamma'] = v^*\cdot z$, then the lemma above implies that the intersection $\pi^{-1}(v)\cap \gamma'$ is unbounded.

\begin{proof}
Let $\gamma'$ be a $\pi^{-1}(w)$-geodesic with $[\gamma']=z$. Assume that there exists $t>0$ such that $w\not\in \pi(\gamma[t, \infty)])$. Since $T$ is a tree there exists an edge $\alpha$ of $T$ such that all paths from $\pi(\gamma[t, \infty))$ to $w$ intersect $\alpha$. In other words, all paths from $\gamma[t, \infty)$ to $\gamma'$ have to intersect $\pi^{-1}(\alpha)^+$. Hence for every constant $C$, $\mc N_C(\gamma')\cap \gamma[t, \infty)\subset \mc N_C(\pi^{-1}(\alpha)^+)\cap \gamma [t, \infty)$. By Lemma \ref{lemma:unbounded_crossing} there exists a constant $C$ such that $\mc N_C(\gamma')\cap \gamma[t, \infty)$ is unbounded but by Lemma \ref{lemma:properties_of_morseless_stars} (\ref{constant:edge_group_intersection2}), $N_C(\pi^{-1}(\alpha)^+)\cap \gamma [t, \infty)$ is always bounded, a contradiction.
\end{proof}

We say that a combinatorial ray $r$ of infinite type is $M$-Morse if for all $j$, $\alpha_j(r)$ is $M$-Morse. If $r$ is a combinatorial ray of finite type we say $r$ is $M$-Morse if its tail $z(r)$ is $M$-Morse and for all $1\leq j \leq l(r)$, $\alpha_j(r)$ is $M$-Morse. We say a combinatorial ray is Morse, if it is $M$-Morse for some Morse gauge $M$.

The set of all $M$-Morse combinatorial rays is denoted by $\delta^M_*X$ and the set of all Morse combinatorial rays is denoted by $\delta_* X$ and called the combinatorial Morse boundary.

\begin{definition}[Realisation of a combinatorial ray]

Let $r$ be a combinatorial ray and $\eps$ the empty path. If $r$ is finite, let $\omega = \omega(r)$ be its end vertex, let $j = \check{\omega}$ and let $\xi$ be a realisation of the tail of $r$, $z(r)$, starting at $(1, j)$. Define 
\begin{align}
    \gamma_i(r) = \begin{cases}
    [e, \alpha_1^*]&\text{if $i=0$ and $l(r)>0$}\\
    [\alpha_{i}^*, \alpha_{i+1}^*]&\text{if $1\leq i < l(r)$}\\
    \omega(r)^*\cdot\xi&\text{if $i = l(r)$}\\
    \eps &\text{if $i >l(r)$}.
    \end{cases}
\end{align}

The realisation of $r$, denoted by $\bar{r}$, is the infinite concatenation
\begin{align}
    \bar{r}=\gamma_0(r)\circ \gamma_1(r)\circ \gamma_2(r)\circ \ldots
\end{align}

\end{definition}

The definition involves a choice of geodesics $[\alpha^*_{j-1}, \alpha_j^*]$. It does not matter how we choose but for simplicity of arguments we assume that the choice is fixed for the rest of the paper and consistent for all combinatorial rays $r$. Hence, for every combinatorial ray $r$, we have chosen a unique realisation $\bar{r}$.

If $r$ is of finite type, then $\bar{r}$ can also be written as the finite concatenation $\bar{r}=\gamma_0(r)\circ\gamma_1(r)\circ\ldots\circ\gamma_{l(r)}(r)$.

\begin{remark}\label{rem:morsenessofrealizationparts}
If $r$ is an $M$-Morse combinatorial ray, then by Lemma \ref{lemma:monster} \ref{prop:geodesic_segments} and \ref{prop:realization1} about choices of realisations, all the geodesics $\gamma_i(r)$ are $f_1(M)$-Morse. Hence the realisation of an $M$-Morse combinatorial ray $r$ is a concatenation of $f_1(M)$-Morse geodesics. 
\end{remark}

We can show that the realisation of a Morse combinatorial ray is a quasi-geodesic.

\begin{lemma}\label{lemma:quasi-geodesic-representatives}
Let $\mc G$ be a Morseless star. There exists an increasing function $D_4: \mc M\to \R_{\geq 1}$ such that for every $M$-Morse combinatorial ray $r$, its realisation $\bar{r}$ is a $D_4(M)$-quasi-geodesic.
\end{lemma}

\begin{figure}
\centering
\begin{minipage}{.5\textwidth}
  \centering
  \includegraphics[width=.9\linewidth]{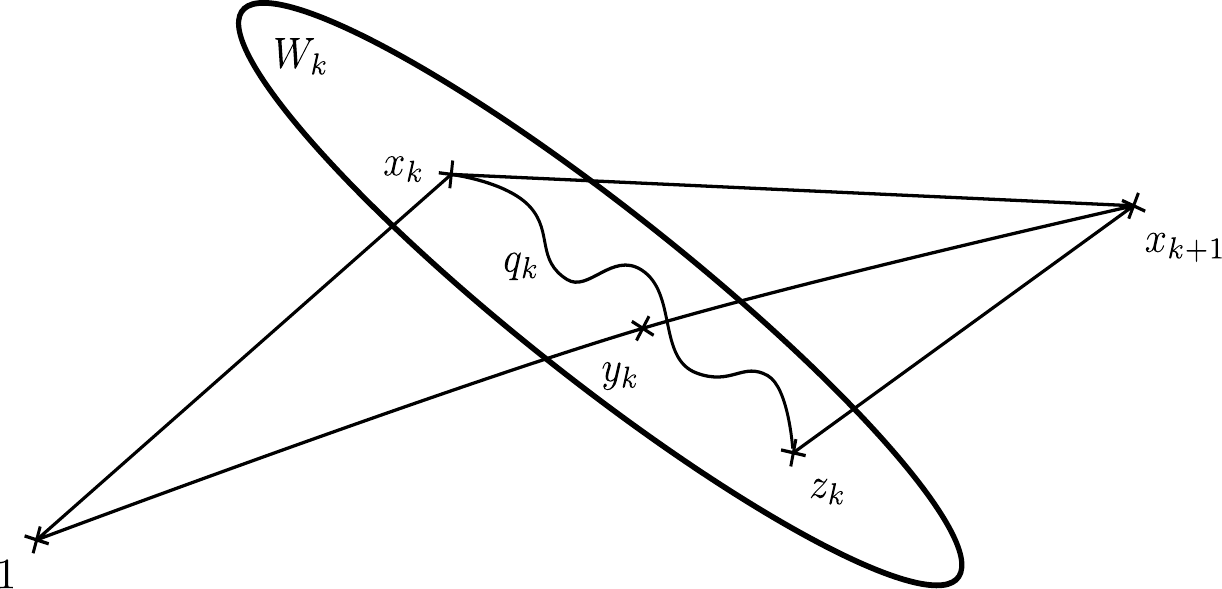}
  \captionof{figure}{Proof of Lemma~\ref{lemma:quasi-geodesic-representatives}}
  \label{picture:lemma3.14}
\end{minipage}%
\begin{minipage}{.5\textwidth}
  \centering
  \includegraphics[width=.9\linewidth]{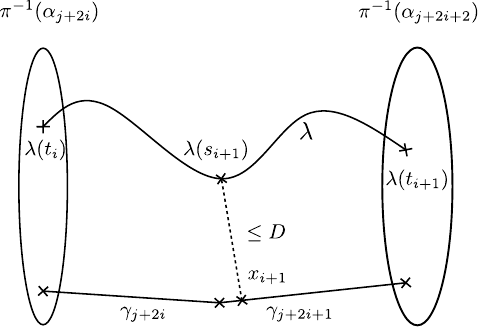}
  \captionof{figure}{Proof of Lemma \ref{lemma:morse_preserving_rel_hyp}}
  \label{picture:morse_preserving_rel_hyp}
\end{minipage}
\end{figure}

\begin{proof}
Let $r$ be an $M$-Morse combinatorial ray. Let $s, t\in [0. \infty)$. We want to show that $\bar{r}$ is a $D = D_4(M)$-quasi-geodesic by showing that
\begin{align} \label{lemma:quasi-geodesic-representatives:eq2}
   \frac{\abs{t-s}}{D} - D\leq d_X(\bar{r}(s), \bar{r}(t))\leq D\abs{t - s} +D.
\end{align} 
We determine the value of $D$ during the proof and it is independent of $r$, $s$ and $t$.

Let $x_0 = e$. For $1\leq k \leq l(r)$, let $W_k$ be the edge group $\pi^{-1}(\alpha_k(r))^+$ and  let $x_k$ be its basepoint $\alpha_k(r)^*$. If $r$ is of finite type we additionally say $x_{l(r)+1}$ is any vertex on $\gamma_{l(r)}(r)$ with $x_{l(r)+1}\in\pi^{-1}(\omega(r))$ and $d_X(e, x_{l(r)+1})\geq d_X(e, x_{l(r)}) +1$. Lemma \ref{lemma:vertex_group_intersection} shows that infinitely many choices for $x_{l(r)+1}$ exist.
Note that for $1\leq k \leq l(r)$ we have 
\begin{itemize}
\item $[e, x_{k+1}]$ intersects $W_k$,
\item $d_X(e, x_{k+1})\geq d_X(e, x_{k})+1$.
\end{itemize}

The points $x_k$ and $x_{k+1}$ both lie on $\gamma_k(r)$, which is $f_1(M)$-Morse by Remark \ref{rem:morsenessofrealizationparts}. Let $\eta$ be the subsegment of $\gamma_k(r)$ from $x_k$ to $x_{k+1}$. Lemma \ref{lemma:monster}\ref{prop:subsegments} about subsegments shows that $\eta$ is $M_1 = f_1^2(M)$-Morse. Let $y_k$ be a point in the intersection $[e, x_{k+1}]\cap W_k$ and let $z_k$ be a closest point to $x_{k+1}$ on $W_k$. We have that
\begin{align}\label{lemma:quasi-geodesic-representatives:eq1}
     d_X(e, x_k)\leq d_X(e, y_k) \qquad \text{and}\qquad d_X(z_k, x_{k+1}) \leq d_X(y_k, x_{k+1}).
\end{align}
Let $q_k$ be a $W_k$-geodesic from $z_k$ to $x_k$. The path $[x_{k+1} ,z_k]q_k$ has endpoints on $[x_k, x_{k+1}]$ and, by Lemma \ref{lemma:properties_of_morseless_stars} (\ref{lemma:vertex_group_morseness}) and Lemma \ref{quasi-geodesic}, is a $3C_2$-quasi-geodesic. Thus $q_k\subset W_k$ is in the $M_1(3C_2)$-neighbourhood of $\eta$ and by Lemma \ref{lemma:properties_of_morseless_stars} (\ref{constant:edge_group_intersection2}), we have that $d_X(z_k, x_k)\leq D_3(M_1, M_1(3C_2)) = K$. 

By the triangle inequality, $d_X(x_k, x_{k+1})\leq d_X(x_{k+1}, z_k) + K$ and thus \eqref{lemma:quasi-geodesic-representatives:eq1} implies
\begin{align}
    d_X(e, x_{k+1}) &= d_X(e, y_k) + d_X(y_k, x_{k+1})\geq d_X(e, x_k) + d_X(z_k, x_{k+1})\geq d_X(e, x_k) + d_X(x_k, x_{k+1}) - K.
\end{align}

Define $a_{k+1}$ via 
\begin{align}
    d_X(e, x_{k+1}) = d_X(e, x_k) + d_X(x_k, x_{k+1}) - a_{k+1}.
\end{align}
Note that $0\leq a_{k+1}\leq K$. Since $d_X(e, x_{k+1})\geq d_X(e, x_k) +1$ we also have $d_X(x_k, x_{k+1}) - a_{k+1}\geq 1$. By induction, we get that 
\begin{align}
    d_X(e, x_{k+1}) = \sum_{j=0}^k(d_X(x_j, x_{j+1}) - a_{j+1}).
\end{align}
Furthermore, for any $b\in [x_k, x_{k+1}]$ we have 
\begin{align}\label{lemma:quasi-geodesic-representatives:eq3}
    d_X(e, x_{k+1}) - d_X(x_{k+1}, b) \leq d_X(e, b) \leq d_X(e, x_k) + d_X(x_k, b).
\end{align}

Let $u = \bar{r}(s)$ and let $v = \bar{r}(t)$. We may assume (by potentially changing $x_{l(r)+1}$) that $u\in [x_n,x_{n+1}]$ and $v\in[x_m, x_{m+1}]$ for some $0\leq n \leq m \leq l(r)$. Let $D=2K\geq 2$. We are ready to start bounding $d_X(u, v)$. If $n = m$, $u$ and $v$ lie on the same geodesic and \eqref{lemma:quasi-geodesic-representatives:eq2} is immediate. From now on we assume $n<m$. Since $\bar{r}$ is a concatenation of geodesics, the upper bound in \eqref{lemma:quasi-geodesic-representatives:eq1} holds. Observe that 
\begin{align}
    t - s = \left( \sum_{j = n}^m d_X(x_j, x_{j+1})\right)  - d_X(x_n, u) - d_X(x_{m+1}, v). 
\end{align}
By \eqref{lemma:quasi-geodesic-representatives:eq3}, 
\begin{align}
     d_X(u, v)\geq  d_X(e, v) - d_X(e, u) &\geq d_X(e, x_{m+1}) - d_X(x_{m+1}, v) - d_X(e, x_n) - d_X(x_n, u) \\
      & = - d_X(x_{m+1}, v)- d_X(x_n, u) + \sum_{j = n}^m (d_X(x_j, x_{j+1}) - a_{j+1}) = A.\\
\end{align}
We can bound $A$ in two ways. On one hand
\begin{align}
    A = t - s - \sum_{j=n}^m a_{j+1}\geq t - s  - K(m - n+1).
\end{align}
If $t - s\geq 2K(m - n - 1)$ we have $(t - s)/2\geq K(m - n -1)$ and hence $A\geq t - s - K(m - n +1)\geq (t-s)/2 - 2K$. Since $D\geq 2$, we are done in this case. On the other hand we know that $d_X(x_j, x_{j+1}) - a_{j+1}\geq 1$ and hence
\begin{align}
     A &= (d_X(x_n, x_{n+1}) - d_X(x_n, u)) - a_{n+1} + (d_X(x_m, x_{m+1}) - d_X(x_{m+1}, v)) - a_{m+1}\\
     & + \sum_{j=n+1}^{m-1} (d_X(x_j, x_{j+1}) - a_{j+1})\\
     &\geq  0 -K + 0 - K + (m - n -1).
\end{align}
Thus, if $t -s\leq 2K(m - n - 1)$, then $d_X(u,v)\geq A\geq (t - s)/D - D$. 
Hence, $\bar{r}$ is indeed a $D=D_4(M)$-quasi-geodesic. 
\end{proof}

\begin{definition} \label{definition:morse-preserving} A graph of groups $\mc G$ is Morse preserving if there exists an increasing function $f_\pres:\mc M\to\mc M$ such that for every $M$-Morse combinatorial ray $r$, its realisation $\bar{r}$ is $f_\pres (M)$-Morse. 
\end{definition}

Not every Morseless star is Morse preserving; the example in the introduction is an example of a Morseless star which is not Morse preserving. 
The Morse preserving property is essential for our proof strategy and it is the main reason Theorem \ref{theorem:star_of_groups} is only stated for relatively hyperbolic Morseless stars. Relative hyperbolicity is a quite strong assumption and it would be interesting to study which weaker assumptions imply that $\mc G$ is Morse preserving.

\begin{lemma}\label{lemma:morse_preserving_rel_hyp}
Every relatively hyperbolic Morseless star $\mc G$ is Morse-preserving.
\end{lemma}

We prove this, by using the following observation. If $r$ is a combinatorial Morse ray, and $\lambda$ is a quasi-geodesic with endpoints on $\bar{r}$, then $\lambda$ has to come close to $\bar{r}$ regularly. This allows us to ``chop" $\lambda$ into segments and prove that each segment stays close.

\begin{proof}
This proof is depicted in Figure \ref{picture:morse_preserving_rel_hyp}. Let $M$ be a Morse gauge, and let $r$ be an $M$-Morse combinatorial ray. Let $K\geq 1$ and let $\lambda : [0, T]\to X$ be a $K$-quasi-geodesic with endpoints on $\bar{r}$. We may assume that $\lambda^-\in\gamma_j(r)$ and $\lambda^+\in \gamma_k(r)$ for some $0\leq j\leq k\leq l(r)$. For all $i\geq 0$ with $j+2i\leq k$ choose $t_i\in [0, T]$ such that $\lambda(t_i)\in \pi^{-1}(\alpha_{j+2i})$. In this case both $\lambda[t_i, t_{i+1}]$ and $\gamma_{j+2i}(r)\gamma_{j+2i+1}(r)$ are quasi-geodesics from $P(\alpha_{j+2i})$ to $P(\alpha_{j+2i+2})$. Lemma \ref{lemma:rel_hyp_projection} about relative hyperbolicity states that they both pass close to the projection $\pi_{P(\alpha_{j+2i})}(P(\alpha_{j+2i+2}))$. In particular, there exists $x_{i+1}\in \gamma_{j+2i}(r)\gamma_{j+2i+1}(r)$ and $s_{i+1}\in [t_i, t_{i+1}]$ with $d_X(x_{i+1}, \lambda(s_{i+1}))\leq C_1 + D_2(K) +D_2 (D_4(M))= D$. Set $n = \lfloor{\frac{k-j}{2}}\rfloor+1$ and define $s_0 =0, s_n = T, x_0 = \lambda(0)$ and $x_n = \lambda(T)$. For every $0\leq i <n $ the concatenation $p=[x_i, \lambda(s_i)]\lambda[s_i, s_{i+1}][\lambda(s_{i+1}), x_{i+1}]$ is a $(K+2D)$-quasi-geodesic with endpoints on $\gamma = \gamma_{j+2i}(r)\circ\gamma_{j+2i+1}(r)\circ\gamma_{j+2i+2}(r)\circ\gamma_{j+2i+3}(r)$. By Remark \ref{rem:morsenessofrealizationparts} and Lemma \ref{lemma:monster} \ref{prop:multiple_concatenation} about concatenations, the quasi-geodesic $\gamma$ is $M_1 = f_1^4(M)$-Morse. Hence $p$ is contained in the $M_1(K+2D)$ neighbourhood of $\gamma\subset \bar{r}$.
Since this holds for every index $i<n$, the realisation $\bar{r}$ is $f_\pres(M)$-Morse, where $f_\pres(M)(L, C) =  M_1(L+C+2D)$. 
\end{proof}

\begin{lemma}
Every Morseless star $\mc G$ with trivial edge groups is Morse-preserving.
\end{lemma}

\begin{proof}
Let $r$ be an $M$-Morse combinatorial ray. We show that $\bar{r}$ is $f_1(M)$-Morse. Let $\lambda: [0, T]\to X$ be a $(K, C)$-quasi-geodesic with endpoints $\lambda^-\in \gamma_j(r)$ and $\lambda^+\in \gamma_k(r)$ for some $0\leq j\leq k \leq l(r)$. For $j+1\leq i \leq k$ removing $\alpha_i^*$ disconnects $X$ and $\lambda^-$ and $\lambda^+$ are in different connected components. Thus, there exists $s_i\in [0, T]$ with $\lambda(s_i) = \alpha_i^*$. Define $s_j = 0$ and $s_{k+1} = T$. With this definition for $j\leq i \leq k$, $\lambda[s_i, s_{i+1}]$ has endpoints on $\gamma_i(r)$, which is $f_1(M)$-Morse by Remark \ref{rem:morsenessofrealizationparts}. Hence $\lambda[s_i, s_{i+1}]\subset \mc N_{f_1(M)(K, C)}(\gamma_i(r))\subset\mc N_{f_1(M)(K, C)}(\bar{r})$. Since this holds for all $j\leq i \leq k$ we have $\lambda[s_j, s_{k+1}] = \lambda\subset \mc N_{f_1(M)(K, C)}(\bar{r})$.
\end{proof}

Now we we are ready to state the main result of this section.

\begin{proposition}\label{prop:equivalence_morse_boundary}
Let $\mc {G}$ be a Morse-preserving Morseless star. The map $\Phi : \delta_* X\to \partial_* X$ defined by $\Phi(r)=[\bar{r}]$ is a well-defined bijection. Moreover, there exists an increasing function $f_\bij: \mc M \to \mc M$ such that if $r\in \comb X$ (resp $\gamma \in \mb X$) is $M$-Morse, then $\Phi(r)$ (resp $\Phi^{-1}(\gamma)$) is $f_\bij(M)$-Morse. 
\end{proposition}

Every Morse combinatorial ray $r$ has an associated geodesic path $p(r)$ in $T$, which to a large extent determines $r$. To show that $\Phi$ is surjective one has to, given a Morse quasi-geodesic ray $\gamma$, recover the geodesic path in $T$ it is associated to. Thus we introduce the following definition.

\begin{definition}
Let $\gamma$ be a Morse quasi-geodesic starting at $e$. We say an outgoing edge $\alpha\in E(T)$ lies on $\gamma$ if there exists $t>0$ such that $\pi(\gamma[t, \infty))\subset T_\alpha$.
\end{definition}

\begin{lemma}\label{lemma:representatives_in_tree}
Let $\gamma : [0, \infty) \to X$ be a Morse quasi-geodesic with $\gamma(0)=e$. The following properties hold:

\begin{enumerate}
    \item The edges lying on $\gamma$ form a geodesic $p_\gamma$ in $T$ that starts at $G_0$. 
    \item If there are only finitely many edges lying on $\gamma$, then for every $t>0$, $\pi(\gamma [t,\infty)) \cap p_\gamma^+\neq \emptyset$.
    \item If an outgoing edge $\alpha$ does not lie on $\gamma$ there exists $t>0$ such that $\pi(\gamma[t,  \infty))\cap T_\alpha = \emptyset$
    \item Let $r$ be an $M$-Morse combinatorial ray. The edges lying on its realisation $\bar{r}$ are precisely the edges of $p(r)$.
\end{enumerate}
\end{lemma}

This lemma states that we can indeed recover a geodesic path in $T$ associated to Morse quasi-geodesics.

\begin{proof}
(1): Let $\beta$ be an outgoing edge lying on $\gamma$ and let $\alpha$ be an outgoing edge with $d_T(\alpha, G_0)\leq d_T(\beta, G_0)$. If $\alpha$ is an edge on $[G_0, \beta^+]_T$, then $T_\beta\subset T_\alpha$ and thus $\alpha$ lies on $\gamma$. Otherwise $T_\alpha$ and $T_\beta$ are disjoint and thus $\alpha$ does not lie on $\gamma$. Let $\{\alpha_1, \alpha_2, \ldots\}$ be the outgoing edges of $T$ that lie on $\gamma$ and assume $d_T(G_0, \alpha_k)\leq d_T(G_0, \alpha_{k+1})$ for all $k$. By the observation above $(\alpha_1, \ldots, \alpha_k) = [G_0, \alpha_k^+]_T$. Hence $(\alpha_1, \alpha_2, \ldots)$ is a geodesic path starting at $G_0$.

(2): Let $v = p_\gamma^+$. Assume there exists some $ t > 0$ such that $v\not \in\pi(\gamma[t, \infty))$. Then $\pi(\gamma[t, \infty))$ is contained in a connected component $C$ of $T - \{v\}$. By the definition of $p_\gamma$, $C$ cannot contain $G_0$. Therefore $C$ is of the form $C = T_\alpha\cup \{\alpha\}$ for some outgoing edge $\alpha$ with $\alpha^- = v$. By Lemma \ref{lemma:properties_of_morseless_stars}(\ref{constant:edge_group_intersection2}), $\pi^{-1}(\alpha)\cap \gamma$ is bounded. Hence, there exists some $s>t$ such that $\pi(\gamma[s, \infty))\subset T_\alpha$, implying that $\alpha$ lies on $\gamma$. This is a contradiction to the definition of $v$. Hence $v\in \pi(\gamma[t, \infty))$ for all $t>0$.

(3): Observe: for every outgoing edge $\alpha$ in $T$, $\gamma\cap \pi^{-1}(\alpha)^+$ is bounded by Lemma \ref{lemma:properties_of_morseless_stars} (\ref{constant:edge_group_intersection2}). Hence, for large enough $t$, $\pi(\gamma[t, \infty))$ is contained in a connected component of $T-\{\alpha\}$. This implies $\pi(\gamma[t, \infty))$ is either contained in $T_\alpha$ or disjoint from $T_\alpha$, which implies (3). 

(4): Case 1: Assume that $r$ is of infinite type. Let $\alpha = \alpha_k(r)$ for some $k>0$. For $j\geq k$, $\pi(\alpha_j^*(r))=\alpha_j(r)^+\in T_{\alpha}$. Property (3) implies that $\alpha$ lies on $\gamma$. The edges $\alpha_k(r)$ form an infinite geodesic path, so by (1) there is no other edge that lies on $\gamma$, which concludes (4) in this case. 

Case 2: Assumet that $r$ is of finite type. Lemma \ref{lemma:vertex_group_intersection} implies that $\pi^{-1}(\omega(r))\cap \bar{r}[t, \infty)$ is non-empty for every $t>0$. Thus, every edge $\alpha$ lying on $\bar{r}$ satisfies that $\omega(r)\in T_{\alpha}$. Thus by (3) $\alpha$ lies on $\bar{r}$ if and only if $\omega(r)\in T_{\alpha}$. The edges $\alpha$ such that $\omega(r)\in T_{\alpha}$ are precisely the ones lying on the geodesic path $p(r)$. 
\end{proof}

Before we can start with the proof of Proposition \ref{prop:equivalence_morse_boundary} we need one last technical Lemma. 

\begin{lemma}\label{lemma:morse_edges}
Let $\mc G$ be a Morseless star. There exists an increasing function $f_6: \mc M\to \mc M$ such that the following holds. Let $q = (\alpha_1, \ldots, \alpha_n)$ be a geodesic path in $T$ starting at $G_0$. If there exists $x\in \pi^{-1}(T_{\alpha_n})$ such that $[e, x]$ is a subsegment of an $M$-Morse geodesic, then the edges $\alpha_i$ are $f_6(M)$-Morse for all $1\leq i \leq n$.
\end{lemma}

\begin{proof}
Let $y_0 = z_0=e$ and for $1\leq j \leq n$ define $y_j = \alpha_j^*$ and choose $z_j\in \pi^{-1}(\alpha_j)^+\cap[e, x]$. With these definitions, $\alpha_j$ is $N$-Morse if and only if $[y_{j-1}, y_j]$ is $N$-Morse.  For $1\leq j \leq n$, let $\lambda_j$ be a $\pi^{-1}(\alpha_j)^+$-geodesic from $y_j$ to $z_j$. Since $y_j$ is the closest point to $e$ in $\pi^{-1}(\alpha_j)^+$ the path $[e, y_j]\lambda_j$ is a $3C_2$-quasi-geodesic by Lemma \ref{lemma:properties_of_morseless_stars} (\ref{lemma:vertex_group_morseness}) and Lemma \ref{quasi-geodesic}. Therefore $y_j$ is contained in the $C = M(3C_2)$-neighbourhood of $[e, x]$. Then, by Lemma \ref{lemma:monster} \ref{prop:adapted2.5}, the geodesic $[y_{j-1}, y_j]$ is $f_6(M) = f_2(M, C)$-Morse.
\end{proof}

Now we are ready to prove Proposition \ref{prop:equivalence_morse_boundary}.

\begin{proof}[Proof of Proposition \ref{prop:equivalence_morse_boundary}]
Note that Lemma \ref{lemma:quasi-geodesic-representatives} implies that for every Morse combinatorial Morse ray $r$, its realisation $\bar{r}$ is a quasi-geodesic. Since we assume the $\mc G$ is Morse-preserving $\bar{r}$ is Morse and hence the map $\Phi$ is well-defined.

Injectivity: Let $r_1, r_2$ be Morse combinatorial rays with $r_1\neq r_2$ and we show that $[\bar{r}_1]\neq [\bar{r}_2]$. 

Case 1: $p(r_1)\neq p(r_2)$: By Lemma \ref{lemma:representatives_in_tree} there exists without loss of generality an edge $\alpha$ such that $\alpha$ lies on $\bar{r}_1$ but not on $\bar{r}_2$. Hence, there exist $t > 0$ such that every path from $\bar{r}_1[t, \infty)$ to $\bar{r}_2[t, \infty)$ goes through $\pi^{-1}(\alpha)$. Hence for every constant $C>0$, the intersection $A_C = \mc N_C(\bar{r}_1[t, \infty))\cap \bar{r}_2[t, \infty)$ is contained in the intersection $B_C = \mc N_C (\pi^{-1}(\alpha))\cap \bar{r}_2[t, \infty)$. By Lemma \ref{lemma:properties_of_morseless_stars} (\ref{constant:edge_group_intersection2}) $B_C$ is bounded for every constant $C$. Hence $A_C$ is bounded for every constant $C$. Lemma \ref{lemma:unbounded_crossing} implies that $[\bar{r}_1]\neq [\bar{r}_2]$.

Case 2: $p(r_1) = p(r_2)$: If $p(r_1)= p(r_2)$ they are either both of finite or of infinite type. If $r_1$ and $r_2$ are of infinite type and $p(r_1) = p(r_2)$, they are equal. Thus, we know that $r_1$ and $r_2$ are of finite type. Let $n = l(r_1)=l(r_2)$. We have that $\bar{r}_1$ and $\bar{r}_2$ have bounded Hausdorff distance if and only if $[\gamma_n(r_1)] = [\gamma_n(r_2)]$. Since $r_1\neq r_2$ we have that $[\gamma_n(r_1)] \neq [\gamma_n(r_2)]$ and hence $[\bar{r}_1]\neq [\bar{r}_2]$. 

Surjectivity: Let $\gamma$ be an $M$-Morse geodesic ray starting at $e$. We construct a combinatorial ray $r$ such that $[\gamma] = [\bar{r}]$. 

Case 1: There are infinitely many edges that lie on $\gamma$. Let $p_\gamma = (\alpha_1, \alpha_2, \ldots)$. Let $r$ be the combinatorial ray $r = (\alpha_1, \alpha_2, \ldots)$. By Lemma \ref{lemma:morse_edges}, $r$ is $f_6(M)$-Morse and hence $\bar{r}$ is an $N_1 = f_\pres(f_6(M))$-Morse $D_4(f_6(M))$-quasi-geodesic. For every $j>0$, define $y_j = \alpha_j^*$ and choose $x_j \in \pi^{-1}(\alpha_j)\cap \gamma$. Since the geodesic $[e, y_j]$ has endpoints on $\bar{r}$ it is $N_2 = f_2(N_1, D_4(f_6(M)))$-Morse by Lemma \ref{lemma:monster} \ref{prop:adapted2.5}. Similarly, the geodesic $[e, x_i]$ has endpoints on $\gamma$ and is thus $N_3 = f_2(M, 1)$-Morse by Lemma \ref{lemma:monster} \ref{prop:adapted2.5}. By \ref{lemma:monster} \ref{lemma:triangles} applied to the triangle $([e, x_j][x_j, y_j][y_j, e])$ the geodesic $[x_j, y_j]$ is $N_4 = f_1(\max\{N_2, N_3\})$-Morse. But then, by Lemma \ref{lemma:properties_of_morseless_stars} (\ref{constant:edge_group_intersection2}), $d_X(x_j, y_j)\leq D_3(N_4, 1) = C$. Thus $y_j\in \mc N_C(\gamma)$ for all $j>0$. Lemma \ref{lemma:unbounded_crossing} implies that $[\bar {r}] = \gamma$.

Case 2: There are only finitely many edges lying on $\gamma$. Let $p_\gamma = (\alpha_1, \ldots, \alpha_n)$, $v = p_\gamma^+$ and $Y=\pi^{-1}(v)$. For $n\in \N$ choose $x_n\in \gamma[n,\infty)\cap Y$. Such $x_n$ exist by Lemma \ref{lemma:representatives_in_tree}. Let $\lambda_n = [v^*, x_n]_Y$ and let $M_1 =f_3(f_2(M, C_2))$. By Lemma \ref{lemma:monster} \ref{prop:adapted2.5} and Lemma \ref{lemma:properties_of_morseless_stars}, $\lambda_n$ is $(M_1,Y)$-Morse. By Arzel\`a-Ascoli, a subsequence $(\lambda_{n_k})_k$ of the sequence $(\lambda_n)_n$ converges in $Y$ (and hence in $X$) to a $Y$-geodesic $\lambda$ and by Lemma 2.10 of \cite{Cor16} the $Y$-geodesic $\lambda$ is $(M_1, Y)$-Morse. Let $D = d_X(e, v^*)$. The paths $[e, v^*]\lambda_{n}$ are $(C_2 +D)$-quasi-geodesics with endpoints on $\gamma$ and thus lie in the $M(C_2+D)$-neighbourhood of $\gamma$. Since this holds for all $n$, $\lambda$ also lies in the $M(C_2+D)$-neighbourhood of $\gamma$. By Lemma \ref{lemma:unbounded_crossing} we have that $[\gamma] = \lambda$. Define $z = [(v^*)^{-1}\cdot \lambda]$ and let $r=(\alpha_1, \ldots, \alpha_n; z)$. We then have $[\bar{r}] = [\gamma]$.

Moreover part: Let $r$ be a combinatorial $M$-Morse ray. Since $\mc G$ is Morse-preserving and by Lemma \ref{lemma:quasi-geodesic-representatives}, $\bar{r}$ is an $f_\pres(M)$-Morse $D_4(M)$-quasi-geodesic. Lemma \ref{lemma:monster} \ref{prop:realization1} implies that $\Phi(r)= [\bar{r}]$ is $f_2(f_\pres(M), D_4(M))$-Morse. Conversely, let $\gamma\in \mb X$ be $M$-Morse. If the combinatorial ray $r = \Phi^{-1}(\gamma)$ is of infinite type, then Lemma \ref{lemma:morse_edges} implies that $r$ is $f_6(M)$-Morse. If the combinatorial ray $r = \Phi^{-1}(\gamma)$ is of finite type and length $n$, Lemma \ref{lemma:morse_edges} implies that all edges $\alpha_i(r)$ are $f_6(M)$-Morse. It is left to show that the tail $z(r)$ is suitably Morse. Let $t>0$ such that $\gamma(t)\in \pi^{-1}(\alpha_n(r))^+$. Let $\lambda$ be a $\pi^{-1}(\alpha_n(r))^+$-geodesic from $\omega(r)^*$ to $\gamma(t)$. By Lemma \ref{quasi-geodesic} and Lemma \ref{lemma:properties_of_morseless_stars}(\ref{lemma:undistorted_vertex_groups}) the path $[e, \omega(r)^*]\lambda$ is a $3C_2$-quasi-geodesic and hence $\omega(r)^*$ is in the $C= M(3C_2)$-neighbourhood of $\gamma$. Let $s>0$ be such that $d_X(\gamma(s), \omega(r)^*)\leq C$. By Lemma \ref{lemma:monster} \ref{prop:subsegments} and \ref{prop:realization1}, $\gamma_n(r)$ (and hence $z(r)$) is $f_2(f_1(M), C)$-Morse. This concludes the moreover part and hence the proof of the proposition.
\end{proof}

During the construction of $r = \Phi^{-1}(\gamma)$ for a Morse geodesic ray $\gamma$ we made sure that the edges lying on $\gamma$ were precisely the edges lying on $\bar{r}$. Hence, the bijectivity of $\Phi$ implies the following.

\begin{corollary}
    Let $\gamma$ and $\gamma'$ be Morse geodesic rays starting at $e$ with $[\gamma]=[\gamma']$. Then an edge $\alpha$ lies on $\gamma$ if and only if it lies on $\gamma'$.
\end{corollary}

\subsection{Further properties of Morseless stars}

In following sections we need some more properties of Morseless stars and relatively hyperbolic Morseless stars. We will state and prove them here.

\begin{lemma}\label{lemma:cosets_morseness_vs_closeness}
Let $\mc G$ be a Morseless star. There exist increasing functions $D_5: \mc M \times \R_{\geq 1}\to \R_{\geq 1}$ and $f_{\emp} : \mc M\times \R_{\geq 1}\to \mc M$ such that the following holds. Let $\alpha\in E(T)$ be an outgoing edge, $v=\alpha^-$ and let $q$ be a closets point to $v^*$ on $\pi^{-1}(\alpha)^-$. If $\gamma$ is an $M$-Morse $C$-quasi-geodesic from $v^*$ to some $p\in\pi^{-1}(\alpha)$, then $d_X(p,q)\leq D_5(M, C)$ and $[v^*, q]$ is $f_{\emp}(M, C)$-Morse.
\end{lemma}

\begin{proof}
Let $\lambda$ be a $\pi^{-1}(\alpha)$-geodesic from $q$ to $p$. By Lemma \ref{lemma:properties_of_morseless_stars} (\ref{lemma:undistorted_vertex_groups}) and Lemma \ref{quasi-geodesic}, the concatenation $\gamma' = [v^*, q]\lambda$ is a $3C_2$-quasi-geodesic and thus is in the $M(3C_2)$-neighbourhood of $\gamma$. By Lemma \ref{lemma:properties_of_morseless_stars} (\ref{constant:edge_group_intersection2}) about intersections of Morse quasi-geodesics with edge groups applied to $\gamma$, $d_X(p, q)\leq \mathrm{diam}(\mc N_{M(3C_2)}(\gamma)\cap \mc N_{M(3C_2)}(\pi^{-1}(\alpha)^+) \leq D_3(M, \max\{C, M(3C_2)\}) = D_5(M, C)$. By Lemma \ref{lemma:monster} \ref{prop:adapted2.5}, the geodesic $[v^*, q]$ is $f_\emp(M) = f_2(M, \max\{C,M(3C_2)\})$-Morse.
\end{proof}

\begin{lemma}\label{lemma:close_projections}
Let $\mc G$ be a relatively hyperbolic Morseless star. There exists an increasing function $D_6 : \mc M\to \R_{\geq 1}$ such that the following holds. Let $\alpha\in E(T)$ be an outgoing edge, $v = \alpha^-$ and let $q$ be a closets point to $v^*$ on $\pi^{-1}(\alpha)^-$. If $[v^*, q]$ is $M$-Morse then $d_X(\bar{\alpha}^*, q)\leq D_6(M)$.
\end{lemma}

\begin{proof}

Let $\beta$ be an outgoing edge adjacent to $G_0$ such that $\pi^{-1}(\beta)$ contains $e$. Let $\gamma$ be a geodesic from $e$ to $\bar\alpha^*$. Define $A=P(\alpha)$ and $B=P(\beta)$.
We distinguish three cases:

Case 1: $v = G_0$. In this case, $\bar\alpha^*$ is a closest point of $v^*$ to $\pi^{-1}(\alpha)^-$. We apply Lemma \ref{lemma:cosets_morseness_vs_closeness} with $p = q$ and $q = \bar\alpha^*$. So in this case $D_6(M)\geq D_5(M, 1)$ works.

Case 2: $\check v =0$. Let $\beta_1$ be the outgoing edge of $T$ with $\beta_1^+ = v$. By Remark \ref{rem:p(alpha)are_the_same}, $P(\beta_1)\neq A$. Let $p\in \gamma\cap \pi^{-1}(\beta_1)\subset P(\beta_1)$. By Lemma \ref{lemma:rel_hyp_projection}, $d_X(\pi_A(v^*), q)\leq C_1$, $d_X(\pi_A(p), \bar\alpha^*)\leq C_1$ and $d_X(\pi_A(p), \pi_A(v^*))\leq C_1$. By the triangle inequality, $D_6(M)\geq 3C_1$ works.

\begin{figure}\centering
\includegraphics[width= \linewidth]{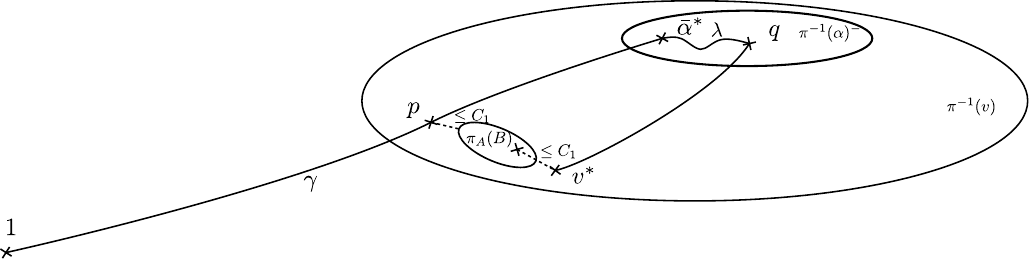}
\caption{Proof of Lemma \ref{lemma:close_projections}}
\label{picture:proof_of_lemma_close_projections}
\end{figure}

Case 3: $\check v\neq 0$. This is case is depicted in Figure \ref{picture:proof_of_lemma_close_projections}.  First we show that there exists a point $p$ on $\gamma$ such that $d_X(p, v^*)\leq 3C_1+1$. If $\beta^+ = v$, then $d_X(e, v^*) = 1$ and hence choosing $p = \gamma(0)$ works. If $\beta^+\neq v$, then $\alpha$ and $\beta$ do not share an endpoint and thus $ A\neq B$. Furthermore, $A=P(v)$. Note that $\gamma$ is a geodesic from $B$ to $A$. By Lemma \ref{lemma:rel_hyp_projection} there exists a point $p$ on $\gamma$ such that $d_X(p, \pi_A(B))\leq C_1$. Also by Lemma \ref{lemma:rel_hyp_projection}, $\mathrm{diam}(\pi_A(B))\leq C_1$ and $d_X(v^*, \pi_A(e))\leq C_1 +1$. By the triangle inequality, $d_X(p, v^*)\leq 3C_1+1$. Thus in both cases, a point $p$ on $\gamma$ with the desired properties indeed exists.

Let $\lambda$ be a $\pi^{-1}(\alpha)^-$-geodesic from $\bar\alpha^*$ to $q$. By definition, $\bar {\alpha} ^*$  is the closest point on $\lambda$ to $p$. Thus by Lemma \ref{quasi-geodesic} and Lemma \ref{lemma:properties_of_morseless_stars} (\ref{lemma:undistorted_vertex_groups}), $[p, \bar{\alpha}^*]\lambda$ is a $3C_2$-quasi-geodesic and hence $\gamma' = [v^*, p][p, \bar\alpha^*]\lambda$ is a $(3C_1+3C_2+1)$-quasi-geodesic with endpoints on $[v^*, q]$. Thus by Lemma \ref{lemma:properties_of_morseless_stars} (\ref{constant:edge_group_intersection2}), $D_6(M)\geq  D_3(M, M(3C_1+3C_2+1))$ works.
\end{proof}

\begin{lemma}\label{lemma:compare_morseness_of_repr}
Let $\mc G$ be a relatively hyperbolic Morseless star. There exists an increasing function $f_8: \mc M\to \mc M$ such that the following holds. Let $\alpha\in E(T)$ be an outgoing edge, $v=\alpha^-$ and let $x\in \pi^{-1}(\alpha)^-$. If $[v^*, x]$ is $M$-Morse, then $[v^*, \bar\alpha^*]$ and $[v^*, \alpha^*]$ are $f_8(M)$-Morse.
\end{lemma}

\begin{proof}
Let $q$ be the closet point of $v^*$ on $\pi^{-1}(\alpha)^-$. By Lemma \ref{lemma:cosets_morseness_vs_closeness} the geodesic $[v^*, q]$ is $M_1 = f_\emp(M, 1)$-Morse. By Lemma \ref{lemma:close_projections}, $d_X(\bar{\alpha}^*, q)\leq D_6(M_1)$ and hence $d_X(\alpha^*, q)\leq D_6(M_1)+1$. By Lemma \ref{lemma:monster} \ref{prop:adapted2.5} the geodesics $[v^*, \bar{\alpha}^*]$ and $[v^*, \alpha^*]$ are $f_8(M) = f_2(M_1, D_6(M_1)+1)$-Morse. 
\end{proof}

\begin{lemma} \label{lemma:distance_of_geo} Let $\mc G$ be a Morse-preserving Morseless star, let $M$ be a Morse gauge and let $(\alpha_1, \ldots, \alpha_n)$ be a geodesic path in $T$. There exists a constant $C_3$ such that the following holds. If $\gamma$ is an $M$-Morse geodesic ray starting at $e$ such that the combinatorial ray $r = \Phi^{-1}([\gamma])$ has length at least $n$ and $\alpha_n(r) = \alpha_n$, then $d_X(\gamma(t), \gamma_n(r)(t))\leq C_3$ for all $0\leq t \leq l(\gamma_n(r))$.
\end{lemma}

Note that the constant $C_3$ is allowed to depend on the geodesic path $(\alpha_1, \ldots, \alpha_n)$ (and hence on $n$) and the Morse gauge $M$.

\begin{proof}
Let $K_1$ be the distance from $e$ to $\alpha_n^*$, that is, $K_1$ is the distance from $\gamma(0)$ to $\gamma_n(r)(0)$. The geodesic $\gamma$ is $M$-Morse, and thus $r$ is $M_1 = f_\bij(M)$-Morse by Proposition \ref{prop:equivalence_morse_boundary}. Lemma \ref{lemma:quasi-geodesic-representatives} shows that $\bar{r}$ is a $K_2 = D_4(M_1)$-Morse $M_2 =f_\pres(M_1)$-quasi-geodesic. We have that $[\bar{r}] = [\gamma]$. Hence, Lemma \ref{lemma:monster}\ref{lemma:close_representatives} about the distance of representatives shows that the Hausdorff distance between $\bar{r}$ and $\gamma$ is at most $K_3 =D_1(M_2, K_2)$. Let $C_3 = 2K_3 +2K_1$. We have that for every $s$ in the domain of $\gamma_n(r)$ there exists $t\geq$ such that $d_X(\gamma_n(r)(s), \gamma(t))\leq K_3$. Since $\gamma_n(r)$ and $\gamma$ are geodesics, we also have that $s-K_1-K_3 \leq t \leq K_1+K_3+s$ and hence $d_X(\gamma(s), \gamma_n(r)(s))\leq C_3$.
\end{proof}

\begin{lemma}\label{lemma:join_vertex_group}
Let $\mc G$ be a Morseless star. There exists an increasing function $D_\ver : \mc M\to \R_{\geq 1}$ such that the following holds. If $\gamma_1$ and $\gamma_2$ are $M$-Morse geodesics starting at $e$ and such that $\gamma_1(t_1)$ and $\gamma_2(t_2)$ are in $\pi^{-1}(\alpha)^+$ for some $t_1, t_2\geq 0$ and an edge $\alpha\in E(T)$, then $d_X(\gamma_1(s), \gamma_2(s))\leq D_\ver(M)$ for all $0\leq s \leq \max\{t_1, t_2\}$.
\end{lemma}

\begin{proof}
Let $\eta$ be a geodesic from $\gamma_1(t_1)$ to $\gamma_2(t_2)$. The geodesics $\gamma_1[0, t_1]$ and $\gamma_2[0, t_2]$ are subsegments of $\gamma_1$ and $\gamma_2$ respectively and thus $f_1(M)$-Morse by Lemma \ref{lemma:monster}\ref{prop:subsegments}. Applying Lemma \ref{lemma:monster}\ref{lemma:triangles} to the triangle with sides $\gamma_1[0, t_1]$, $\gamma_2[0, t_2]$ and $\eta$ we get that $\eta$ is $M_1 = f_1^2(M)$-Morse. Since the endpoints of $\eta$ lie in $\pi^{-1}(\alpha)^+$, their distance is at most $K_1 = D_3(M_1, 1)$ by Lemma \ref{lemma:properties_of_morseless_stars}\eqref{constant:edge_group_intersection2} about the intersection of Morse geodesics and edge groups. Thus, $\gamma_1[0, t_1]\eta$ is a $K_1$-quasi-geodesic with endpoints on $\gamma_2$ and for every $s_1\leq t_1$ there exists a real number $s_1\geq 0$ such that $d_X(\gamma_1(s_1), \gamma_2(s_2))\leq M(K_1) = K_2$. Since $\gamma_1$ and $\gamma_2$ are geodesics, we have that $s_1- K_2\leq s_2\leq s_1 + K_2$ and hence the statement holds for $D_\ver(M) = 2K_2$.
\end{proof}

Let $v\in V(T)$ and $x\in \pi^{-1}(v)$ be vertices. We can go from $v^*$ to $x$ using only edges labelled with elements from the group $G_{\check{v}}$. Thus the path from $v^*$ to $x$ corresponds to a group element in $G_{\check{v}}$. We want to formalise this with the following definition. 

\begin{definition}\label{def:xg}
    Let $x$ be a vertex in $\pi^{-1}(v)$ for some vertex $v\in V(T)$. We have that $v^* = (g, j)$ for some $g\in G$ and $j\in V(\Gamma)$. We define $x^G$ as the element of $G$ which satisfies $(gx^G, j) = x$.
\end{definition}

\begin{remark}
Since $x$ is in $\pi^{-1}(v)$ the element $x^G$ is actually in $G_{\check{v}}$. 
\end{remark}

Next we want to formalise the correspondence between edges in $T$ and cosets of vertex groups. Let $v\in V(T)$ be a vertex and let the projection of $v$ to $\Gamma$ be $i = \check{v}$. Furthermore, let $\beta\in E(\Gamma)$ be an edge whose source is $i$ and denote by $A$ the set of edges of $E(T)$ whose source is $v$ and whose projection to $\Gamma$ is $\beta$. We can use the definition from above to construct a map $\tilde\psi : A \to G_i$. Namely, we will send outgoing edges $\alpha\in A$ to $(\bar\alpha^*)^G$ and non-outgoing edges $\alpha\in A$ to $(\alpha^*)^G$. Since edges of $A$ correspond to cosets of $H_\beta$, the map $\tilde\psi$ actually induces a bijection $\psi : A\to G_i/H_\beta$.

\begin{lemma}\label{lemma:properties_of_xg}
Let $\mc G$ be a Morseless star. There exists an increasing function $f_\xgg: \mc M\to \mc M$ such that the following holds. Let $v\in V(T)$ be a vertex and $\alpha\in E(T)$ be an outgoing edge whose source is $v$. If $\alpha$ is $M$-Morse, then $(\bar\alpha^*)^G$ is $(f_\xgg(M), G_{\check{v}})$-Morse. Conversely, if $(\bar\alpha^*)^G$ is $(M, G_{\check{v}})$-Morse, then $\alpha$ is $f_\xgg(M)$-Morse.
\end{lemma}

In other words, if $\alpha$ is $M$-Morse, $\tilde{\psi}(\alpha)$ is $(f_\xgg(M), G_{\check{v}})$-Morse and if $\tilde{\psi}(\alpha)$ is $(M, G_{\check{v}})$-Morse, then $\alpha$ is $f_\xgg(M)$-Morse. 

\begin{proof}
Let $i = \check {v}$ and let $g = (\bar\alpha^*)^G$. The geodesic $[v^*, \bar\alpha^*]$ is a translation of the geodesic $[(1, i), (g, i)]$. So, if one of them is $M$-Morse, so is the other. Furthermore, the distance of $\alpha^*$ and $\bar\alpha^*$ is $1$. The Lemma thus follows from Lemma \ref{lemma:monster}\ref{prop:adapted2.5} about geodesics with endpoints in the neighbourhood of Morse geodesics and Lemma \ref{lemma:properties_of_morseless_stars}\eqref{lemma:vertex_group_morseness}, which gives us control over Morse gauges when switching from one viewpoint to another.
\end{proof}

Given a group $H$, an element $x\in H$ and a biinfinite Morse geodesic line $\lambda$ satisfying $\lambda(0) = e$, called the \emph{baseline}, one can construct a ray $\lambda_x$ starting at the basepoint and corresponding to $x$ as follows. If $x\cdot \lambda (t)$ is a closest point on $x\cdot \lambda$ to $e$ and $t>0$, then replace $\lambda$ by $\lambda^{-1}$ (here $\lambda^{-1}$ denotes the geodesic line defined as $\lambda^{-1}(t) = \lambda(-t)$). Choose $\lambda_x$ as a realisation of $x\cdot \lambda[0, \infty)$. This construction is used to do the following; given a $G_i$-geodesic line $\lambda$ (the baseline) and an element $x\in G_i$, we get a $G_i$-geodesic ray $\lambda_x$ corresponding to $x$. Note that with this setup, $\lambda_x$ is a $C_2$-quasi-geodesic. Furthermore, for a vertex $v\in V(T)$ and vertex $x\in \pi^{-1}(v)$ we want to have a (quasi)-geodesic ray starting at $v^*$ and passing close to $x$. To do so, we will use the above construction in the vertex group and then translate the ray. The following definition makes this more precise.

\begin{definition} \label{def:corresponding_rays}
Let $v\in V(T)$ be a vertex, let $i= \check v$ and let $\lambda$ be a Morse $G_i$-geodesic line going through $(1, i)$. For a vertex $x\in \pi^{-1}(v)$ and $y = x^G$ define $\lambda_{x, v}$ as $v^*\cdot \lambda_y$. Furthermore, we define the Morse direction of $x$, denoted by $\gamma_x$, as $\gamma_{x} = [\lambda_y]\in \mb G_i$. 
\end{definition}

With this definition, $\lambda_{x, v}$ is a $C_2$-quasi-geodesic and $x\cdot \gamma_{x} = [\lambda_{x, v}]$. Furthermore, $\lambda_{x, v}$ is a $\pi^{-1}(v)$-geodesic realisation  of $x\cdot \lambda[0, \infty)$ starting at $v^*$. The following lemma states some properties of $\lambda_{x, v}$. Property \eqref{prop:morseness} and \eqref{prop:contained_neighbourhoods} are equivalents of properties proven for $\lambda_y$ in Lemma 2.27  of \cite{graph_of_groups1}. However, because of the slightly different setting it is easier to prove them from scratch then use Lemma 2.27 of \cite{graph_of_groups1} to  prove them.

\begin{lemma}\label{lemma:corresponding_geodesic}
There exist increasing functions $f_\lam : \mc M\to \mc M$ and $D_\dlam : \mc M\times \R_{\geq 1}\to \R$ such that the following holds. Let $v\in V(T)$ be a vertex, let $i = \check v$, let $x\in \pi^{-1}(v)$ be a vertex and let $\lambda$ be a $G_i$-geodesic line satisfying $\lambda(0) = (1, i)$. 
\begin{enumerate}
    \item If $\lambda$ and one of $[v^*, x]$ and $\lambda_{x, v}$ are $M$-Morse, the other is $f_8(M)$-Morse. The same holds when replacing $\lambda_{x, v}$ by $\gamma_x$.\label{prop:morseness} 
    \item If $\gamma_x$ and $[v^*, x]$ are both $M$-Morse and $d_X(v^*, x)\geq D_\dlam(M, k)$, then 
    \begin{align}
        x\in \tilde O_k^M([\lambda_{x, v}], v^*) \qquad \text{and} \qquad [\lambda_{x, v}]\in O_k^M(x, v^*).
    \end{align}\label{prop:contained_neighbourhoods}
    \item Assume $[v^*, x]$ is $M$-Morse, $\lambda$ is $N$-Morse, and $y$ is on $\lambda_{x, v}$. If $z$ is on $[v^*, y]$ and $d_X(z, v^*)\geq D_\dlam(M,  d_X( v^*, x) )$, then $[x, z]$ is $f_8(N)$-Morse.\label{prop:nice_connections}
\end{enumerate}
\end{lemma}

\begin{proof}
\eqref{prop:morseness}: If $\lambda$ is $M$-Morse, then the subray $\lambda[0, \infty)$ is $M_1 = f_2(M, C_2)$-Morse by Lemma \ref{lemma:monster}\ref{prop:subsegments}. Observe that $x\cdot\lambda[0, \infty)$, $[v^*, x]$ and $\lambda_{x, v}$ form a $C_2$-quasi-geodesic triangle. Thus, if one of $[v^*, x]$ and $\lambda_{x, v}$ is $M$-Morse, then the other is $f_2(M_1, C_2)$-Morse by Lemma \ref{lemma:monster}\ref{lemma:triangles}. Thus \eqref{prop:morseness} holds for $\lambda_{x, v}$. Let $\gamma$ be a realisation of $v^*\cdot \gamma_x$ starting at $v^*$. If $\lambda_{x, v}$ is $N$-Morse, then $\gamma$, and hence $\gamma_x$ is $f_2(N, C_2)$-Morse by Lemma \ref{lemma:monster}\ref{prop:realization1}. On the other hand, if $\gamma_x$ is $N$-Morse, there exists an $N$-Morse realisation of $v^*\cdot\gamma_x$ starting at $v^*$. Then, by Lemma \ref{lemma:monster}\ref{prop:realization1}, $\lambda_{x, v}$ is $f_2(N, C_2)$-Morse. Thus \eqref{prop:morseness} still holds if we replace $\lambda_{x, v}$ by $\gamma_x$.

\eqref{prop:contained_neighbourhoods}: Assume that $\gamma_x$ and $[v^*, x]$ are $M$-Morse. Let $x_0 = x\cdot \lambda(t)$ be a closest point on $x\cdot \lambda$ to $v^*$. Lemma \ref{quasi-geodesic} implies that $\gamma = [v^*, x_0]x\cdot\lambda[t, \infty)$ is a $3C_2$-quasi-geodesic. We can assume that $t\leq 0$ or equivalently, that $x$ lies on $\gamma$. Let $\eta$ be an $M$-Morse realisation of $[\lambda_{x, v}] = v^*\cdot\gamma_x$ starting at $v^*$. Applying Lemma \ref{lemma:monster}\ref{prop:realization1} to $\eta$ and $\gamma$ we get that $\gamma$ is $M_1 = f_2(M, 3C_2)$-Morse. Applying Lemma \ref{lemma:monster}\ref{lemma:close_representatives} to $\eta$ and $\gamma$ we get that their Hausdorff distance is at most $K_1 = D_1(M_1, 3C_2)$. In particular, there exists $s_0\in [0, \infty)$ such that $d_X(x, \eta(s_0))\leq K_1$. Since $\eta$ is $M$-Morse, the $K_1$-quasi-geodesic $[v^*, x][x, \eta(s_0)]$ is contained in the $K_2 = M(K_1)$-neighbourhood of $\eta$. Since $\eta$ and $[v^*, x]$ are geodesics with the same starting poiont, $d_X(\eta(s), [v^*, x](s))\leq 2K_2$ for all $s\leq d_X(v^*, x)$ by the triangle inequality. Thus we can set $D_\dlam(M, k)= \max\{12K_2, k+4K_2\}$ and Lemma \ref{lemma:get_close_again} concludes the proof. 

\eqref{prop:nice_connections}: By Lemma \ref{lemma:monster}\ref{prop:subsegments} about subsegments of quasi-geodesics, $\lambda[0, \infty)$ is $N_1 = f_2(N, C_2)$-Morse. Let $x_1 = \lambda_{x ,v}(t_0)$ be the closest point on $\lambda_{x, v}$ to $x$. Observe that $d_X( x, x_1)\leq d_X(v^*, x)$. Hence by the triangle inequality, $d_X(v^*, x_1)\leq 2 d_X( v^*,x)$. Since $\lambda_{x, v}$ is a $C_2$-quasi-geodesic we have that $t_0\leq C_2(2d_X(v^*, x) + C_2) = K_1$. In particular, if $y = \lambda_{x, v}(t)$ lies on $\lambda_{x,v}$ and $d_X(v^*, y)\geq C_2(K_1 +C_2) = K_2$, then $t\geq t_0$. By Lemma \ref{quasi-geodesic}, the concatenation $\eta = [x, x_1]\lambda_{x, v}[t_0, \infty)$ is a $3C_2$-quasi-geodesic. Since $[\eta] = [x\cdot\lambda[0,\infty)]$, Lemma \ref{lemma:monster}\ref{prop:realization1} about the Morseness of realisations implies that $\eta$ is $N_2 = f_2(N_1, 3C_2)$-Morse. Lemma \ref{lemma:monster}\ref{lemma:close_representatives} about the distance of realisations implies that the Hausdorff distance between $\eta$ and $x\cdot \lambda[0, \infty)$ is at most $K_3 = D_1(N_2, 3C_2)$. Define $D_\dlam(M, d_X(v^*, x)) = K_2$. Let $y$ be on $\lambda_{x, v}$ and let $z$ be a point on $[v^*, y]$ such that $d_X(v^*, z)\geq D_\dlam(M, d_X(v^*, x))$. Since $d_X(v^*, y)\geq d_X(v^*, z)\geq K_2$ the proof above shows that $d_X(y, x\cdot \lambda[0, \infty))\leq K_3$. Let $x_2$ be the closest point to $x$ on $[v^*, y]$. We have that $d_X(x_2, x)\leq d_X(v^*, x)$ and hence $d_X(v^*, x_2)\leq 2d_X(v^*, x)$. Let $\gamma$ the subsegment of $[v^*, y]$ from $x_2$ to $y$. Let $q$ be a closest point to $y$ on $x\cdot \lambda[0, \infty)$. By Lemma \ref{quasi-geodesic} the concatenation $[x, x_2]\gamma$ is a $(3,0)$-quasi-geodesic. Since $d_X(y, q)\leq K_3$, the concatenation $\eta' = [x, x_2]\gamma[y, q]$ is a $3K_3$-quasi-geodesic. It has endpoints on $x\cdot\lambda[0, \infty)$ and hence is contained in the $K_4 =N_1(3K_3)$-neighbourhood of $x\cdot \lambda[0, \infty)$. Since $d_X(v^*, z)\geq K_2\geq 2d_X(v^*, x)$, the point $z$ lies on $\gamma$ and hence on $\eta'$. Hence both $x$ and $z$ lie in the $K_4$-neighbourhood of $x\cdot \lambda[0, \infty)$, Lemma \ref{lemma:monster}\ref{prop:adapted2.5} implies that $[x, z]$ is $f_\lam(N) = f_2(N_1, K_4)$-Morse. 
\end{proof}

\section{Homeomorphisms of the Morse boundary}\label{section:homeo_tools}

In this section we develop tools to determine when maps $q: \mb G \to \mb H$, for some finitely generated groups $G$ and $H$ are homeomorphisms. If $q$ is a map from $\mb G$ to $\mb H$ then $q$ is continuous if and only if there exists a function $f : \mc M\to \mc M$ such that the following holds. For every $M$-Morse direction $z\in \mb X$ and integer $n$ there exists an integer $m$ such that $q(O_m^M(z))\subset O_n^{f(M)}(q(z))$. In Section 2.4 of \cite{graph_of_groups1}, this is explained in more detail. To show that $q$ is a homeomorphism we only need to prove that it is bijective and both $q$ and its inverse $q^{-1}$ are continuous.

\subsection{Neighbourhoods of the combinatorial Morse boundary} 
Let $\mc G$ be a Morse preserving Morseless star. Here, we define neighbourhoods for the combinatorial Morse boundary $\comb X$ and show how they correspond to neighbourhoods of the Morse boundary. Let $r$ be a combinatorial $M$-Morse ray. 
If $r$ is of finite type let $n= l(r)$ and $v = \omega(r)$ and define 
\begin{align}
    V_k^M(r) = \{r'\in \comb ^M X | \text{ $l(r')\geq n$ and $\gamma_n(r')\in \tilde{O}_k^{f_1(M)} (\gamma_n(r), v^*)$}\}.
\end{align}
If $r$ is of infinite type, define
\begin{align}
    V_k^M(r) = \{r'\in \comb ^M X\mid \text{$l(r')\geq k$ and  $\alpha_k(r') =\alpha_k(r)$} \}.
\end{align}

Observe that if $\alpha_k(r) = \alpha_k(r')$, then $\alpha_j(r) = \alpha_j(r')$ for all $1\leq j \leq k$. 

Let $r\in \comb^M X$ be a combinatorial $M$-Morse ray. Recall that $\gamma_j(r)$ is $f_1(M)$-Morse for all $0\leq j \leq l(r)$ and that we have a bijection $\Phi : \comb X \to \mb X$ defined by $r\mapsto [\bar{r}]$, which sends $M$-Morse combinatorial rays to $f_\bij(M)$-Morse directions and $M$-Morse directions to $f_\bij(M)$-Morse rays.  

\begin{lemma}\label{lemma:neighbourhood_inclusions}
Let $\mc G$ be a Morse preserving Morseless star. Let $r\in \comb ^M X$ and let $z\in \mb ^M X$. For every integer $n$ there exists an integer $m$ such that 
\begin{enumerate}[label= \roman*)]
    \item $\Phi^{-1}(O_m^M(z))\subset V_n^{f_\bij(M)}(\Phi^{-1}(z))$.
    \item $\Phi(V_m^M(r))\subset O_n^{f_\bij(M)}(\Phi(r))$.
\end{enumerate}
\end{lemma}

\begin{figure}\centering
\includegraphics[width= \linewidth]{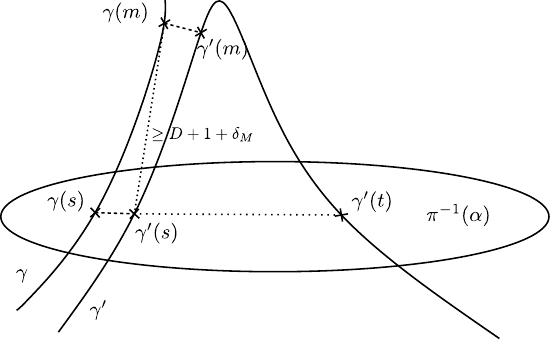}
\caption{Proof of Lemma \ref{lemma:neighbourhood_inclusions}}
\label{picture:proof_of_lemma_4.1}
\end{figure}

\begin{proof}
(i): Case 1: The combinatorial ray $\Phi^{-1}(z)$ is of infinite type. This part of the proof is depicted in Figure \ref{picture:proof_of_lemma_4.1}. Let $\gamma$ be an $M$-Morse realisation of $z$ and let $\alpha = \alpha_n(\Phi^{-1}(z))$. In light of Lemma \ref{lemma:representatives_in_tree} it is enough to find $m$ such that $\alpha$ lies on $\gamma'$ for every geodesic ray $\gamma'\in O^M_m(\gamma)$, because it implies that $\alpha_n(\Phi^{-1}([\gamma']))=\alpha$ and we know that $\Phi^{-1}([\gamma'])$ is $f_\bij(M)$-Morse. Recall that Lemma \ref{lemma:properties_of_morseless_stars}\ref{constant:edge_group_intersection2} states that $D_3$ bounds the intersection of neighbourhoods of edge groups with Morse geodesics. Let $D = D_3(M, \delta_M)$, let $s\geq0$ such that $\gamma(s)\in \pi^{-1}(\alpha)^-$ and choose $m\geq s$ such that $\pi(\gamma(m))\in T_\alpha$ and $d_X(\gamma(m), \pi^{-1}(\alpha)^-)\geq D+ \delta_M+1$. Let $\gamma'\in O_m^M(\gamma)$ and assume $\alpha$ does not lie on $\gamma'$. We have that $d_X(\gamma(m), \gamma'(m))\leq \delta_M$ and hence $\pi(\gamma'(m))\in T_\alpha$. Thus, there exists $t\geq m$ such that $\gamma'(t)\in\pi^{-1}(\alpha)^-$. Since $\gamma'(s)$ and $\gamma'(t)$ are both in the $\delta_M$-neighbourhood of $\pi^{-1}(\alpha)$ we have by Lemma \ref{lemma:properties_of_morseless_stars} (\ref{constant:edge_group_intersection2}) that $d_X(\gamma'(s), \gamma'(t))\leq D$. Observe that $d_X(\gamma'(m), \gamma'(t))\geq d_X(\gamma'(m), \pi^{-1}(\alpha))\geq D+1$. Together this gives $D\geq d_X(\gamma'(s), \gamma'(t))\geq D+1$, which is a contradiction. Hence, our choice of $m$ works in this case.

Case 2: The combinatorial ray $\rho = \Phi^{-1}(z)$ is of finite type. Let $\gamma$ be an $M$-Morse realisation of $z$, let $l= l(\rho)$, $\gamma_l = \gamma_l(\rho)$ and $\alpha = \alpha_l(\rho)$. By the same argumentation as above, there exists $m_1$ such that $\alpha$ lies on $\gamma'$ for every $\gamma'\in O_{m_1}^M(\gamma)$. Let $K_1 = \max\{C_3, \delta_M\}$ (we get $C_3$ by applying Lemma \ref{lemma:distance_of_geo} to $p(\rho)$) and $K_2 = 2K_1 + \delta_M$. Let $m_2 = \max\{n+2K_2, 6K_2\}$ and let $m \geq \max\{m_1, m_2 + D_3(M, K_1)\}$ and such that $\gamma(m)$ is in the vertex group $\pi^{-1}(\alpha^+)$. Such an $m$ exists by Lemma \ref{lemma:vertex_group_intersection}. 

Let $\gamma'\in O_m^M(\gamma)$ be a geodesic ray $\rho'=\Phi^{-1}([\gamma'])$ a combinatorial Morse ray and $\gamma_l' = \gamma_l(\rho')$.

We have that $\rho$ and $\rho'$ are $f_\bij(M)$-Morse. By Lemma \ref{lemma:monster}\ref{prop:realization1} or \ref{prop:geodesic_segments} (depending on whether $\gamma_l$ and $\gamma_l'$ are rays or segments) $\gamma_l$ and $\gamma_l'$ are $M_1 = f_1(f_\bij(M))$-Morse. Moreover, by Lemma \ref{lemma:distance_of_geo} we have that 
\begin{align}
d_X(\gamma(s), \gamma_l(s))\leq K_1 \quad \text{and} \quad d_X(\gamma'(s'), \gamma_l'(s'))\leq K_1,
\end{align}
for all $0\leq s \leq l(\gamma_l)$ and $0\leq s'\leq l(\gamma_l')$.

We first prove that $l(\gamma_{l}')\geq m_2$. If $l(\gamma_l')\geq m$ we are done so we assume $l(\gamma_l')\leq m$. Let $\beta = \alpha_{l+1}(\rho')$. Recall that $\gamma(m)\in \pi^{-1}(\alpha^+)$. Thus since $\gamma'\in O_m^M(\gamma)$ we have that if $\gamma'(m)\in T_{\beta}$, then $d_X(\gamma'(m), \pi^{-1}(\beta))\leq \delta_M$. If $\gamma'(m)\not \in T_{\beta}$ there exists $t\geq m$ such that $\gamma'(t)\in \pi^{-1}(\beta)^+$ since $\beta$ lies on $\gamma'$. So in either case, there exists $t\geq m$ such that $d_X(\gamma'(t), \pi^{-1}(\beta))\leq \delta_M\leq K_1$. We also have that $\gamma_l'^+$ lies in $\pi^{-1}(\beta)^+$. Hence, by Lemma \ref{lemma:distance_of_geo}, $d_X(\gamma'(l(\gamma_l')), \pi^{-1}(\beta))\leq C_3\leq K_1$. Lemma \ref{lemma:properties_of_morseless_stars} (\ref{constant:edge_group_intersection2}) implies that $d_X(\gamma'(l(\gamma_l')), \gamma'(t)) = t  - l(\gamma_l') \leq  D_3(M, K_1)$. In particular, $l(\gamma_l') \geq m  - D_3(M, K_1) \geq m_2$.

By the triangle inequality and Lemma \ref{lemma:distance_of_geo}, $d_X(\gamma_l(s), \gamma_l'(s))\leq 2C_3 +\delta_M\leq 2K_1 + \delta_M = K_2$ for $0\leq s\leq m_2$. Thus, by Lemma \ref{lemma:get_close_again}, $d_X(\gamma_l(s), \gamma_l'(s))\leq \delta_{M_1}$ for all $0\leq s\leq n$, which implies $\gamma_l'(s)\in O_n^{M_1}(\gamma_l)$ and thus $\rho'\in V_n^{f_\bij(M)}(\rho)$.

(ii): Case 1: $r$ is of infinite type. We have that $\Phi(r)$ is $M_1 = f_\bij(M)$-Morse. Let $\gamma$ be an $M_1$-Morse realisation of $\Phi(r)$. Recall that $D_\ver$ as defined in Lemma \ref{lemma:join_vertex_group} bounds the distance of two $M$-Morse realisations with endpoints in the same edge group. Define $m = \max\{n+2D_\ver(M_1), 6D_\ver(M_1)\}$. Let $r'\in V_m^M(r)$ be a combinatorial ray. We have to show that the $M_1$-Morse realisation $\gamma'$ of $r'$ is in $O_n^{M_1}(\gamma)$. We have that $\alpha_m(r)$ lies on $\gamma$ and $\gamma'$. Hence there exist $t_1, t_2\geq m$ such that $\gamma(t_1)$ and $\gamma(t_2)$ are in $\pi^{-1}(\alpha_m(r))^+$. By Lemma \ref{lemma:join_vertex_group} we have that $d_X(\gamma(t), \gamma'(t))\leq D_\ver(M_1)$ for all $0\leq t \leq m$ and hence Lemma \ref{lemma:get_close_again} concludes the proof.

Case 2: $r$ is of finite type. Again, $\Phi(r)$ is $M_1 = f_\bij(M)$-Morse. Let $\gamma $ be an $M_1$-Morse realisation of $\Phi(r)$ and let $l = l(r)$. We apply Lemma \ref{lemma:distance_of_geo} to the path $p(r)$ and the Morse gauge $M_1$ to get a constant $C_3$. Define $D= 2C_3 + \delta_{f_1(M)}$ and define $m = \max\{n+ 2D, 6D\}$.  Let $r'\in V_m^M(r)$ be a combinatorial ray. We have to show that the $M_1$-Morse realisation $\gamma'$ of $\Phi(r')$ is in $O_n^{M_1}(\gamma)$. By Lemma \ref{lemma:distance_of_geo} we have that $d_X(\gamma(s), \gamma_l(r)(s))\leq C_3$ and $d_X(\gamma'(s), \gamma_l(r')(s))\leq C_3$ for $0\leq s \leq m$. We also have that $d_X(\gamma_l(r)(s), \gamma_l(r')(s))\leq\delta_{f_1(M)}$ for all $0\leq s \leq m$. Hence by the triangle inequality, $d_X(\gamma(s), \gamma'(s))\leq D$ for all $0\leq s\leq m$. Lemma \ref{lemma:get_close_again} concludes the proof.

\end{proof}

\subsection{Homeomorphisms from Bass-Serre tree maps}\label{sec:local-bijections}

Let $\mc G$ and $\mc G'$ be Morse preserving Morseless stars. We use the notation from Section \ref{section:graph_of_group} and just add a $'$ to the objects corresponding to the graph of groups $\mc G'$. That is, we denote the Bass-Serre space of $\mc G'$ by $X'$, the Bass-Serre tree of $\mc G'$ by $T'$ and so on. 

\begin{definition}[Bass-Serre tree map]
    We say a map $\phi : V(T)\to V(T')$ is a Bass-Serre tree map if it satisfies the following conditions.
    \begin{enumerate}
        \item $\phi(G_0) = G_0'$.\label{cond:start_at_G0}
        \item (Injectivity) The map $\phi$ is injective.\label{cond:injectivity}
        \item (Nestedness) For vertices $v, w\in V(T)$ with $v\in T_w$ we have that $\phi(v)\in T_{\phi(w)}'$. Conversely, if $v, w\in V(T)$ and $\phi(v)\in T_{\phi(w)}'$, then $v\in T_w$.\label{cond:paths_exist}
        \item (Coarse surjectivity) There exists a constant $C_\tree$ such that for every vertex $v\in V(T)$ and vertex $w\in T_{\phi(v)}'$ with $d_{T'}(w, \phi(v))\geq C_4$ there exists an outgoing edge $\alpha\in E(T)$ with $\alpha^- = v$ and $w\in T_{\phi(\alpha^+)}'$. \label{cond:bounded_jumps}
        \item (Partial surjectivity) Every vertex $v\in V(T')$ where the vertex group $G_{\check v}'$ has non-empty Morse boundary is in the image of $\phi$. \label{cond:surj_on_mores}
        \item (Boundary homeomorphisms) For every vertex $v\in V(T)$, there exists a homeomorphism $p_v : \mb G_{\check{v}} \to \mb  G_{\check{\phi(v)}}'$. \label{cond:homeo}
        \item (Morse conditions) There exists an increasing function $f_\ftree : \mc M\to \mc M$ such that the following conditions hold. \label{cond:morseness}
        \begin{enumerate}
            \item If $\alpha\in E(T)$ is outgoing and $M$-Morse, then the edges on the geodesic $[\phi(\alpha^-), \phi(\alpha^+)]_{T'}$ are $f_\ftree(M)$-Morse. Conversely, if $\alpha\in E(T)$ is an outgoing edge and all edges on the geodesic $[\phi(\alpha^-), \phi(\alpha^+)]_{T'}$ are $M$-Morse, then $\alpha$ is $f_\ftree(M)$-Morse. \label{cond:morseness:edges}
            \item Let $v\in V(T)$ be a vertex. If $z\in \mb  G_{\check {v}}$ is $M$-Morse, then $p_v(z)$ is $f_\ftree(M)$-Morse. Conversely if $z\in  \mb  G_{\check {v}}$ and $p_v(z)$ is $M$-Morse, then $z$ is $f_\ftree (M)$-Morse.\label{cond:morseness:tails}
        \end{enumerate}
    \end{enumerate}
\end{definition}

\begin{remark}\label{remark:no_injectivity_needed}
Injectivity is implied by the other conditions. Namely, if $u, v\in V(T)$ are vertices such that $\phi(u) = \phi(v)$, then in particular $\phi(u)\in T'_{\phi(v)}$ and $\phi(v)\in T'_{\phi(u)}$. Hence nestedness implies that $u\in T_v$ and $v\in T_u$. The latter can only be true if $u = v$. Thus, to prove a map $\phi : V(T)\to V(T')$ is a Bass-Serre tree map, we do not need to prove injectivity. 
\end{remark}

Given a Bass-Serre tree map $\phi$ we can define a map, which we denote by $\bar{\phi}$, from $\comb X$ to $\comb X'$ as follows: Let $r\in \comb X$ be a Morse combinatorial ray. 
\begin{itemize}
    \item Assume $r= (\alpha_1, \ldots, \alpha_n, z)$ is of finite type. Let $(\beta_1, \ldots, \beta_n)$ be the geodesic path from $G_0'$ to $\phi(\alpha_n^+)$. Define 
    \begin{align}
        \bar{\phi}(r) = (\beta_1, \ldots, \beta_n; p_{\alpha_n^+}(z)).
    \end{align}
    \item Assume $r = (\alpha_1, \alpha_2, \ldots)$ is of infinite type. Conditions \eqref{cond:start_at_G0} and \eqref{cond:paths_exist} imply that the path $\gamma$ in $T'$ described by the infinite concatenation 
    \begin{align}
        \gamma = [G_0', \phi(\alpha_1^+)]_{T'}[\phi(\alpha_1^+), \phi(\alpha_2^+)]_{T'}[\phi(\alpha_2^+), \phi(\alpha_3^+)]_{T'}\ldots
    \end{align}
    is in fact a geodesic path. Let $\gamma = (\beta_1, \beta_2, \ldots)$ for edges $\beta_i \in T'$. Define 
    \begin{align}
    \bar{\phi}(r) = (\beta_1, \beta_2, \ldots).
    \end{align}
    
\end{itemize}

Condition \eqref{cond:morseness} implies that if $r\in \comb X$ is an $M$-Morse combinatorial ray, then $\bar{\phi}(r)$ is $f_\ftree(M)$-Morse and that conversely, if $r$ is a combinatorial Morse ray and $\bar{\phi}(r)$ is $M$-Morse, then $r$ is $f_\ftree(M)$-Morse. 

\begin{lemma}
Let $\phi : V(T)\to V(T')$ be a Bass-Serre tree map. The induced map $\bar\phi : \comb X\to \comb X'$ is bijective. 
\end{lemma}

\begin{proof}
Injectivity: Let $r, r'\in \comb X$ be combinatorial rays with $\bar\phi(r)=\bar\phi(r')$. By the definition of $\phi$ this implies that they are of the same type.

Case 1: The combinatorial rays $r, r'$ are of finite type. Then, Conditions \eqref{cond:injectivity} and \eqref{cond:homeo} imply that $\omega(r) = \omega(r')$ and $z(r) = z(r')$, which in turn implies that $r = r'$.

Case 2: The combinatorial rays $r$ and $r'$ are of infinite type. For all $j>0$ define $v_j = \phi(\alpha_j(r)^+)$ and $v_j' = \phi(\alpha_j(r')^+)$. Assume by contradiction that $r\neq r'$ and let $k$ be the smallest index such that $\alpha = \alpha_k(r)\neq \alpha_k(r') = \beta$. By the definition of $\bar \phi (r)$ we know that both $v_k$ and $v_k'$ lie on the geodesic path $p(\bar\phi(r))$. Without loss of generality, $d_{T'}(G_0', v_k)\leq d_{T'}(G_0', v_k')$ and hence $v_k'\in T_{v_k}'$. Condition \eqref{cond:paths_exist} implies that $\beta^+\in T_{\alpha^+}$. The vertices $\alpha^+$ and $\beta^+$ both have distance $k$ from $G_0$. Therefore, $\beta^+\in T_{\alpha^+}$ implies that $\beta^+ = \alpha^+$. By the minimality of $k$, we know that $\beta^- = \alpha^-$. Hence $\beta = \alpha$. This is a contradiction and thus concludes the proof of injectivity.

Surjectivity: Let $s\in \comb X'$ be a combinatorial ray. If $s$ is of finite type, then Conditions \eqref{cond:surj_on_mores} and \eqref{cond:homeo} imply that there exists some combinatorial ray $r\in \comb X$ with $\bar\phi(r) = s$. If $s$ is of infinite type, then Condition \eqref{cond:bounded_jumps} implies that at least every $C_\tree$-th vertex on $p(s)$ is in the image $\phi(V(T))$. Hence, there exists a combinatorial ray $r$ such that $\bar\phi(p(r))$ induces the path $p(s)$. Thanks to condition \eqref{cond:morseness:edges} we know that $r$ is Morse and hence $\bar\phi(r) = s$.

\end{proof}

Since we now know that $\bar\phi: \comb X\to \comb X'$ is bijective, we can talk about its inverse $\bar\phi^{-1} : \comb X'\to \comb X$. In particular, we are able to state the following definition. 

\begin{definition}\label{def:bass_serre_continuity}
    We say that a Bass-Serre tree map $\phi$ is continuous if the following conditions hold. 
    \begin{enumerate}[label = (\roman*)]
        \item For every combinatorial $M$-Morse ray $r\in \comb X$ and every integer $n$ there exists an integer $m$ such that $\bar\phi(V_m^M(r))\subset V_n^{f_\ftree(M)}(\bar\phi(r))$.\label{cond:phi}
        \item For every combinatorial $M$-Morse ray $r'\in \comb X'$ and every integer $n$ there exists an integer $m$ such that $\bar\phi^{-1}(V_m^M(r'))\subset V_n^{f_\ftree(M)}(\bar\phi^{-1}(r'))$.\label{cond:phi_inverse}
    \end{enumerate}
\end{definition}

Being continuous is exactly the condition we need from a Bass-Serre tree map $\phi$ to show it induces a homeomorphism on the Morse boundary.

\begin{proposition}\label{prop:bass_serre_implies_homeo}
    Let $\phi: V(T)\to V(T')$ be a continuous Bass-Serre tree map. Then, the map $\phi_*  = \Phi'\circ \bar\phi\circ\Phi^{-1}: \mb X\to \mb X'$ is a homeomorphism. 
\end{proposition}

Recall that $\Phi : \delta_* X\to \partial_* X$ (and hence $\Phi'$) is defined in Proposition~\ref{prop:equivalence_morse_boundary} and is a bijection from the combinatorial Morse boundary to the Morse boundary.

\begin{proof}
We know that $\Phi, \Phi'$ and $\bar\phi$ are all bijective. So we only have to show that $\phi_*$ and $\phi^{-1}_*$ are continuous. 

Define $f= f_\bij \circ f_\ftree \circ f_\bij$, let $M$ be a Morse gauge and define $M_1 = f_\bij(M)$, $M_2 = f_\ftree(M_1)$ and $M_3 = f(M)=f_\bij(M_2)$. Let $z\in \mb X$ be an $M$-Morse direction and let $n$ be an integer. By Lemma \ref{lemma:neighbourhood_inclusions} there exists $m_1$ such that for $A = V_{m_1}^{M_2}(\bar\phi (\Phi^{-1}(z)))$ we have that $\Phi'(A)\subset O_n^{M_3}(\phi_*(z))$. Since $\phi$ is continuous there exists $m_2$ such that for $B = V_{m_2}^{M_1}(\Phi^{-1}(z))$ we have $\bar\phi(B)\subset A$. Then again by Lemma \ref{lemma:neighbourhood_inclusions} there exists $m$ such that for $C = O_m^M(z)$ we have that $\Phi^{-1}(C)\subset B$. This implies $\phi_*(C) \subset O_n^{M_3}(\phi_*(z))$. Hence, $\phi_*$ is continuous. The proof that $\phi_*^{-1}$ is continuous works analogously.
\end{proof}

\subsection{Constructing continuous Bass-Serre tree maps}

In Section \ref{section:relativemap} and \ref{section:doublemap} we will construct some maps that can be viewed as local bijections of the Bass-Serre tree. Here, we will show that if these local bijections are nicely behaved they induce a continuous Bass-Serre tree map and hence a homeomorphism of the Morse boundary.

Let $v\in V(T)$ be a vertex. The denote the set of edges $\alpha\in E(T)$ which are outgoing and have $\alpha^- = v$ by $E_v$. For an edge $\beta\in E(\Gamma)$ we denote the set of edges $\alpha\in E(T)$ which are outgoing, have $\alpha^-=v$ and have $\check \alpha = \beta$ by $E_{v, \beta}$. With this notation, we have 

\begin{align}
    E_v = \sqcup_{\substack{\beta\in E(\Gamma),\\ \beta^- = \check v}} E_{v, \beta}.
\end{align}

Let $v\in V(T)$ be a vertex, a local bijection $q_v$ of depth $C$ and Morseness $f$ is a map $q_v : E_v^+\cup \{v\}\to V(T')$ together with a constant $C$ and an increasing function $f: \mc M\to \mc M$ such that the following conditions hold. 
\begin{enumerate}[label = (\Roman*)]
\item (Nestedness) For every edge $\alpha\in E_v$ we have that $q_v(\alpha^+)\in T_{q_v(v)}'$ and $q_v(\alpha^+)\neq v$. \label{cond2:nestedness}
\item (Coarse surjectivity) If $w\in T_{q_v(v)}'$ is a vertex with $d_{T'}(w, q_v(v))\geq C$, then there exists an edge $\alpha \in E_v$ such that $w\in T_{q_v(\alpha^+)}$. \label{cond2:coars_surj}
\item (Non-nestedness) If $\alpha, \beta\in E_v$ are edges and $q_v(\alpha^+)\in T_{q_v(\beta^+)}'$, then $\alpha = \beta$. \label{cond2:injectivity}
\item (Partial surjectivity) If there exists an edge $\alpha\in E_v$ and a vertex $u\in V(T')$ with non-empty Morse boundary $\mb G_{\check u}$ and such that $u \in T_{q_v(v)}'$ and $q_v(\alpha^+)\in T_u'$, then either $u=q_v(v)$ or $u = q_v(\alpha^+)$.\label{cond2:partial_surj}
\item (Morse condition) If $\alpha\in E_v$ is $M$-Morse, then every edge on the geodesic from $q_v(v)$ to $q_v(\alpha^+)$ is $f(M)$-Morse. Conversely, if $\alpha\in E_v$ is an edge and every edge on the geodesic from $q_v(v)$ to $q_v(\alpha^+)$ is $M$-Morse, then $\alpha$ is $f(M)$-Morse. \label{cond2:morseness}
\end{enumerate}

\begin{figure}
    \centering
    \includegraphics{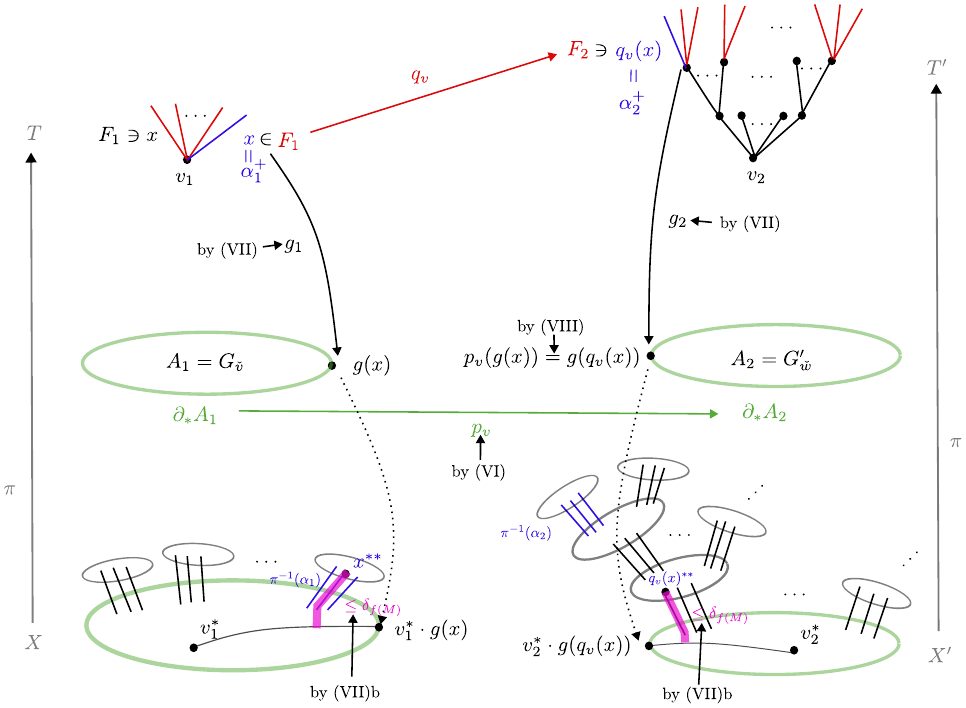}
    \caption{Depiction of a geodeisc local bijection.}
    \label{fig:geolo}
\end{figure}

Let $v\in V(T)$ be a vertex. Let $q_v$ be a local bijection of depth $C$ and Morseness $f$. We need to set up some notation which is more symmetrical. Let $w = q_v(v)$, we denote the vertex group $G_{\check{v}}$ by $A_1$ and the vertex group $G_{\check w}'$ by $A_2$. Sometimes we denote $v$ by $v_1$, $w$ by $v_2$, $d_X$ by $d_1$ and $d_{X'}$ by $d_2$. Furthermore, we denote $E_v^+$ by $F_1$ and $q_v(E_v^+)$ by $F_2$. For $i \in \{1, 2\}$ and a Morse gauge $M$, define $F_i^M\subset F_i$ as the set of all vertices $x\in F_i$ for which every edge on the geodesic path from $v_i$ to $x$ is $M$-Morse. For $i\in \{1, 2\}$ and $x\in F_i$, let $\alpha$ be the first edge on the geodesic path from $v_i$ to $x$. We define $x^{**}$ as $\alpha^*$. With this notation, if $i=1$, then $x^{**}$ is equal to $x^*$. Even though not technically true, one can think of $x^{**}$ as a projection of $x$ to $v_i^*\cdot A_i$.

We say that the local bijection $q_v$ is \emph{geodesic} if $\mb A_1$ is homeomorphic to $\mb A_2$ and either $\mb A_1$ is empty or all of the following conditions hold: 
\begin{enumerate}[resume, label = (\Roman*)]
    \item (Homeomorphism) There exist a homeomorphism $p_v : \mb A_1\to \mb A_2$, such that 
    \begin{enumerate}
        \item If $z\in \mb A_1$ is $M$-Morse, then $p_v(z)$ is $f(M)$-Morse. 
        \item If $p_v(z)\in \mb A_2$ is is $M$-Morse, then $z$ is $f(M)$-Morse. 
    \end{enumerate}\label{cond2:morseness:tails}
    \item (Geodesics) There exists an increasing function $D_v : \mc M\times \R_{\geq 1}\to \R_{\geq 1}$ an maps $g_j: F_j \to \mb A_j$ for $j\in \{1, 2\}$ whose union $g_1\cup g_2$ we denote by $g$ and such that the following conditions hold for every integer $k$, vertex $x\in F_i$ and $i\in \{1, 2\}$.
    \begin{enumerate}
        \item  If $x\in F_i^M$, then $g(x)$ is $f(M)$-Morse. Conversely, if $g(x)$ is $M$-Morse, then $x\in F_i^{f(M)}$.\label{cond2:closeness1}
        \item If $x\in F_i^M$ and $d_i(x^{**}, v_i^*)\geq D_v(M, k)$, then 
        \begin{align}
        x^{**}\in \tilde{O}_k^{f(M)}(v_i^*\cdot g(x), v_i^*) \quad \text{and}\quad v_i^*\cdot g(x)\in O_k^{f(M)}(x^{**}, v_i^*).
        \end{align}\label{cond2:closeness}
    \end{enumerate}\label{cond2:geodesics}
    \item (Commutativity) For every $x\in F_1$ we have that $p_v(g(x)) = g(q_v(x))$.\label{cond2:commutativity}
\end{enumerate}

Roughly speaking, the conditions above state that $q_v$ ``factors through'' a homeomorphism 
$p_v$, which is a homeomorphism between the boundaries of $A_1$ and $A_2$. This is depicted in Figure \ref{fig:geolo}. The maps $g_i$ (whose union is $g$) map edges $x\in F_i$ to points $g(x)\in \partial_* A_i$ such that ``Morseness is respected'' (Condition \ref{cond2:closeness1}) and any $v_i^*$-realisation of $v_i^*\cdot g(x)$ passes ``close'' by $x^{**}$ (\ref{cond2:closeness}).

\begin{proposition}\label{prop:local_bij_imply_bass_serre_map}
Let $C$ be a constant, $f: \mc M\to \mc M$ be an increasing function and let $\{q_v\}_{v\in V(T)}$ be a collection of geodesic local bijections of depth $C$ and Morseness $f$. 
If $q_{G_0}(G_0) = G_0'$ and $q_{\alpha^+}(\alpha^+) = q_{\alpha^-}(\alpha^+)$ for every outgoing edge $\alpha \in E(T)$, then the map $\phi : V(T)\to V(T')$ defined by $\phi(v) = q_v(v)$ is a continuous Bass-Serre tree map.
\end{proposition}

\begin{proof}
We will first show that $q$ is a Bass-Serre tree map. Hence, we need to show that Conditions \eqref{cond:start_at_G0} -  \eqref{cond:morseness} are satisfied. \\
\eqref{cond:start_at_G0}: This follows from the assumption that $q_{G_0}(G_0) = G_0'$.\\
\eqref{cond:paths_exist}: Let $v, w\in V(T)$ be vertices such that $w\in T_v$. Applying \ref{cond2:nestedness} inductively to the vertices of the geodesic path from $v$ to $w$, we get that $\phi(w) \in T_{\phi(v)}'$. If $v\neq w$, then \ref{cond2:nestedness} also implies that $\phi(v)\neq \phi(w)$.

Conversely, let $v, w\in V(T)$ be vertices such that $\phi(w)\in T_{\phi(v)}'$. We have to show that $w\in T_v$. Let $u$ be the lowest common ancestor of $v$ and $w$, or in other words, let $u\in V(T)$ be the vertex furthest away from $G_0$ which satisfies $v, w\in T_u$. If $v = u$, we are done. If $w = u$, then $v\in T_w$ and by the argumentation above, $\phi(v)\in T_{\phi(w)}'$. Together with $\phi(w)\in T_{\phi(v)}'$ this implies that $\phi(v) = \phi(w)$. The observation above shows that if $v\in T_w - \{w\}$, then $\phi(v)\neq \phi(w)$. Hence, $v = w$ and we are done. 

From now on we assume that neither $v$ nor $w$ are equal to $u$. We show that this can never happen. Let $\alpha$ and $\beta$ be edges in $E_u$ lying on the path from $u$ to $v$ and from $u$ to $w$ respectively. Since $u$ is the lowest common ancestor of $v$ and $w$, the edges $\alpha$ and $\beta$ are distinct. Denote $\alpha^+$ by $x$ and $\beta ^+$ by $y$.  Condition \ref{cond2:injectivity} implies that $\phi(x)\not\in T_{\phi(y)}'$ and $\phi(y)\not\in T_{\phi(x)}'$. Hence, $T_{\phi(x)}'$ and $T_{\phi(y)}'$ are disjoint. By the proof from above we also know that $\phi(v)\in T_{\phi(x)}'$ and $\phi(w)\in T_{\phi(y)}'$. But $T_{\phi(v)}'$ is a subset of $T_{\phi(x)}'$ and contains $\phi(w)$. This is a contradiction to $T_{\phi(x)}'$ and $T_{\phi(y)}'$ being disjoint. \\
\eqref{cond:injectivity}: Remark \ref{remark:no_injectivity_needed} shows that this follows from \eqref{cond:paths_exist}. \\
\eqref{cond:bounded_jumps}: Condition \ref{cond2:coars_surj} implies that \eqref{cond:bounded_jumps}  is satisfied for the constant $C_4 = C$. \\
\eqref{cond:surj_on_mores}: We first show that for every vertex $w\in V(T')$ there exists an outgoing edge $\alpha\in E(T)$ such that $w\in T_{\phi(\alpha^-)}'$ and $\phi(\alpha^+)\in T_w$.

Let $w\in V(T')$ be a vertex and let $v\in V(T)$ be the vertex furthest away from $G_0$ such that $w\in T_{\phi(v)}'$. Let $w'$ be a vertex in $T_w'$ at distance at least $C$ from $w$. By \ref{cond2:coars_surj} there exists an edge $\alpha\in E_v$ such that $w'\in T_{\phi(\alpha^+)}'$. Since $T_{\phi(\alpha^+)}'$ and $T_w'$ both contain $w'$ they are not disjoint. Hence, we either have $w\in T_{\phi(\alpha^+)}'$ or $\phi(\alpha^+)\in T_w'$. By the definition of $v$, we know that $w\not\in T_{\phi(\alpha^+)}'$. Thus $\phi(\alpha^+)\in T_w'$.

If $w\in V(T')$ is a vertex such that the group $G_{\check w}'$ has non-empty Morse boundary, then Condition \ref{cond2:partial_surj} implies that $w =  \phi(v)$ or $w = \phi(\alpha^+)$. In particular, $w$ is indeed in the image of $\phi$. \\
\eqref{cond:homeo}: This holds since we require that the local bijections are geodesic. If for a vertex $v\in V(T)$ the Morse boundary $\mb G_{\check v}$ is not empty, the bijection $p_v$ comes from Condition \ref{cond2:morseness:tails}.\\
\eqref{cond:morseness}: Any function $f_\ftree  \geq f$ works. Condition \eqref{cond:morseness:edges} follows from Condition \ref{cond2:morseness} and Condition \eqref{cond:morseness:tails} follows from Condition \ref{cond2:morseness:tails}. Later on, we will see that $f_\ftree$ has more conditions it has to satisfy and determine $f_\ftree$ accordingly. \\

Next, we need to prove that the Bass-Serre tree map $\phi$ is continuous. More precisely, we need to prove that the conditions in Definition \ref{def:bass_serre_continuity} hold. This proof is similar to the proof of Proposition 4.3 in \cite{graph_of_groups1}, although the notation is quite different. We first show that these conditions hold for combinatorial rays of infinite type. 

Let $r\in \comb X$ be an $M$-Morse combinatorial ray of infinite type and let $n$ be an integer. If $s\in V_n^M(r)$, then both $p(\bar\phi(s))$ and $p(\bar\phi(r))$ pass through $\phi(\alpha_n(r))$ and $d_{T'}(\phi(\alpha_n(r)), G_0')\geq n$. Hence $\bar\phi(V_n^M(r))\subset V_n^{f_\ftree(M)}(\bar\phi(r))$.

Let $r'\in \comb X'$ be an $M$-Morse combinatorial ray of infinite type and let $n$ be an integer. Let $v =\alpha_n(\bar\phi^{-1}(r'))$ and let $m = d_{T'}(G_0', \phi(v))$. With this definition we have that $\phi(v) = \alpha_m(r')$. Let $w = \alpha_m(r')$. If $s'\in V_m^M(r')$ and $s = \bar\phi^{-1}(s)$, then there exists some vertex $u$ on the geodesic $p(s)$ such that $\phi(u)\in T_{\phi(v)}$. By Condition \eqref{cond:paths_exist} we have that $u\in T_{v}$ and hence $s\in V_n^{f_\ftree(M)}(\bar\phi^{-1}(r'))$. Hence $\bar\phi^{-1}(V_m^M(r'))\subset V_n^{f_\ftree(M)}(\bar\phi^{-1}(r'))$. 

To show that the conditions of Definition \ref{def:bass_serre_continuity} hold for Morse combinatorial rays of finite type, we first need to prove the following claims.

\textbf{Claim 1.} Let $i\in \{1, 2\}$, let $l$ be a positive integer and let $z\in \mb A_i$ be an $M$-Morse direction. There exists an integer $k$ such that 
\begin{enumerate}[label=(\roman*)]
    \item  for all $x\in F_i^M$ with $x^{**}\in \tilde{O}_k^M(v_i^*\cdot z, v_i^*)$ we have $v_i^*\cdot g(x)\in \tilde{O}_l^{f(M)}(v_i^*\cdot z, v_i^*)$.
    \item for all $x\in F_i^M$ with $d_i(v_i^*, x^{**})\geq k$ and $v_i^*\cdot g(x)\in  \tilde{O}_k^{f(M)}(v_i^*\cdot z, v_i^*)$ we have that $x^{**}\in \tilde{O}_l^{f(M)}(v_i^*\cdot z, v_i^*)$.
\end{enumerate}

\textbf{Proof.} Claim 1 follows from Condition \ref{cond2:closeness} and Lemma 2.24 of \cite{graph_of_groups1}. More precisely, Lemma 2.24 of \cite{graph_of_groups1} implies that there exists an integer $k'$ such that if $y\in \tilde{O}_{k'}^{f(M)}(v_i^*\cdot z, v_i^*)$, then $\tilde{O}_{k'}^{f(M)}(y, v_i^*)\subset \tilde{O}_l^{f(M)}(v_i^*\cdot z, v_i^*)$. Taking $k= \max\{k', D_v(M, k)\}$ works; if $x^{**}\in \tilde{O}_k^{f(M)}$, then $\tilde{O}_{k'}^{f(M)}(x^{**}, v_i^*)\subset\tilde{O}_l^{f(M)}(v_i^*\cdot z, v_i^*)$ and Condition \ref{cond2:closeness} implies that $v_i^*\cdot g(x)\in \tilde{O}_{k'}^{f(M)}(x^{**}, v_i^*)$, which yields $(i)$. Analogously, if $v_i^*\cdot g(x)\in \tilde{O}_k^{f(M)}$ and $d_i(v_i^*, x^{**})\geq D_v(M, k)$, then $\tilde{O}_{k'}^{f(M)}(v_i^*\cdot g(x), v_i^*)\subset\tilde{O}_l^{f(M)}(v_i^*\cdot z, v_i^*)$ and Condition \ref{cond2:closeness} implies that $x^{**}\in \tilde{O}_{k'}^{f(M)}(v_i^*\cdot g(x), v_i^*)$, which yields $(ii)$. $\blacksquare$

\textbf{Claim 2.} Let $i\in \{1, 2\}$. For every Morse gauge $M$ and integer $k$ there exist at most finitely many $x\in F_i^M$ with $d_i(v_i^*, x^{**})\leq k$.

\textbf{Proof.} The graph $X$ is locally finite. Hence, there are only finitely many edges $\alpha\in E_v$ with $d_1(v^*, \alpha^*)\leq k$. Since $x\in F_1$ correspond to edges $\alpha \in E_v$ and $x^{**} = \alpha^*$ the statement follows for $i=1$. The case for $i=2$ is slightly trickier, since for any vertex $x\in F_2$, infinitely many vertices $y\in F_2$ might have $x^{**} = y^{**}$. We need to show that this does not happen when considering vertices in $F_2 ^M$ instead of $F_2$. Every vertex $x\in F_2$ has distance at most $C$ from $v_2$. By Condition \ref{cond2:coars_surj} we know that the vertex group of every vertex $y$ on the geodesic path from $v_2$ to $x$ (not including the endpoints) has empty Morse boundary. Hence, every of these vertices $y$ has only finitely many outgoing edges that are $M$-Morse. So, for every outgoing edge $\beta\in E_{v_2}$ there are only finitely many $x\in F_2^M$ whose geodesic path from $v_2$ to $x$ starts with $\beta$. Now the same observations as for $i=1$ conclude the proof for $i=2$. $\blacksquare$

Now we are ready to prove that the conditions of Definition \ref{def:bass_serre_continuity} hold for combinatorial rays of finite type and suitably chosen $f_\ftree$. 

Let $r\in \comb X$ be an $M$-Morse combinatorial ray of finite type of length $l$ and denote its tail $z(r)$ by $z$ and denote $f_1(M)$ by $N$ and let $n' = n + 18\delta_{f^2(M_1)}$. Since $p_v$ is a homeomorphism, there exists an integer $k_1$ such that $p_v(O_{k_1}^{N}(z))\subset O_{n'}^{f(N)}(p_v(z))$. If $m\geq k_1$, then we have for all combinatorial rays $s\in V_m^M(r)$ with $l(s) = l(r)$ that $p_v(z(s))\in O_{n'}^{f(N)}(p_v(z))$ and hence by Remark \ref{remark:property_of_neighbourhoods}, $\bar\phi(s)\in V_n^{f(N)}(\bar\phi(r))$.  By Claim~1, there exists an integer $k_2$ such that all $y\in F_2^{f(N)}$
with $d_2(y^{**}, v_2^*)\geq k_2$ and $g(y)\in O_{k_2}^{f^2(N)}(p_v(z))$ satisfy $y^{**}\in\tilde{O}_{n'}^{f^2(N)}(v_2^*\cdot p_v(z), v_2^*)$. Since $p_v$ is a homeomorphism, there exists an integer $k_3$ such that $p_v(O_{k_3}^{f (N)}(z))\subset O_{k_2}^{f^2(N)}(p_v(z))$. Lastly, Claim 1 implies that there exists an integer $k_4$ such that if $x\in F_1^{N}$ and $x^{**}\in \tilde{O}_{k_4}^{N}(v_1^*\cdot z, v_1^*)$, then $g(x)\in O_{k_3}^{f(N)}(z)$. Claim 2 implies that there exist only finitely many $y\in F_2^{f(N)}$ with $d_2(y^{**}, v_2^*)\leq k_2$, hence we can choose $k_4$ large enough such that all $x\in F_1^{N}$ with $d_1(x^{**}, v_1^*)\geq k_4$ satisfy $d_2(\phi(x)^{**}, v_2^*)\geq k_2$. Let $m = \max\{k_4, k_1\}$ and let $s\in V_m^M(r)$ be a combinatorial ray with $l(s) > l(r)$. There exists a unique vertex $x\in F_1^{N}$ which lies on $p(s)$. We have that $x^{**} = \alpha_{l+1}^*(s)$ and hence $x^{**} \in \tilde{O}_{k_4}^N(v_1^*\cdot z, v_1^*)$. Claim~1 implies that $g(x)\in  O_{k_3}^{f(N)}(z)$ and hence $p_v(g(x)) = g(\phi(x)) \in O_{k_2}^{f^2(N)}(p_v(z))$. The choice of $k_4$ implies that $d_2(\phi(x)^{**}, v_2^*)\geq k_2$ and hence by Claim~1, we have that $\phi(x)^{**}\in \tilde{O}_{n'}^{f^2(N)}(v_2^*\cdot p_v(z), v_2^*)$. Let $l'$ be the length of $\bar\phi(r)$. We have that $\phi(x)^{**} = \alpha_{l'+1}^*(\bar\phi (s))$  and hence by Remark \ref{remark:property_of_neighbourhoods}, $\bar\phi(s)\in V_n^{f^2(N)}(\bar\phi (r))$.

This concludes the proof of Condition \ref{cond:phi} for $f_\ftree = f^2\circ f_1$. Condition \ref{cond:phi_inverse} is proven analogously. 
\end{proof}

\section{The map} \label{section:relativemap}

Recall that Morseless stars are defined in Section \ref{section:definitoin_of_morseless_stars} and correspond to the graphs of groups of Theorem \ref{cor:rel_hyp_graph_of_groups}, that have exactly one distinguished vertex. Thus the following theorem is a special case of Theorem \ref{cor:rel_hyp_graph_of_groups}, where the distinguished set of vertices $W$ contains exactly one vertex. In this section, we prove Theorem \ref{theorem:star_of_groups} by constructing local bijections and in Section \ref{section:main_result} we use Theorem \ref{theorem:star_of_groups} to inductively prove Theorem \ref{cor:rel_hyp_graph_of_groups}. 

\begin{theorem}\label{theorem:star_of_groups}
Let $\mc{G}$ be a relatively hyperbolic Morseless star. Let $\mc G'$ be the graph of groups with underlying graph $\Gamma' = \Gamma$, vertex groups $G_i' = G_i$ and trivial edge groups. Then $\mb \pi_1(\mc G) \cong \mb\pi_1(\mc G ')$.
\end{theorem}

Observe that the graph of groups $\mc{G}'$ is a Morseless star but not necessarily a relatively hyperbolic Morseless star.\\

We will prove Theorem \ref{theorem:star_of_groups} by constructing a collection of local bijections that satisfy the assumptions of Proposition \ref{prop:local_bij_imply_bass_serre_map}. More precisely, we construct a map $q: V(T)\to V(T')$ and local bijections $\{q_v\}_{v\in V(T)}$ such that $q_v$ and $q$ agree on the intersection of their domains. We start by defining $q(G_0) = G_0'$. Then we repeat the following process until $q$ is defined for all vertices $v\in V(T)$. 

\begin{enumerate}
    \item Choose a vertex $v\in V(T)$ such that $q(v)$ is defined but $q_v$ is not. 
    \item Construct the local bijection $q_v$. \label{step:2}
    \item For every vertex $w\in E_v^+$, define $q(w) = q_v(w)$.
\end{enumerate}

In our process we will make sure that the projections to $\Gamma$ of $q(v)$ and $v$ respectively agree for all vertices $v\in V(T)$. The construction of $q_v$ in the case that $\mb G_{\check{v}}$ is empty differs significantly from the construction in the other case. Thus we first show how to construct $q_v$ if $\mb G_{\check{v}}$ is empty and then show how to consturct $q_v$ if $\mb G_{\check v}$ is non-empty.

\subsection{Constructing the local bijection $q_v$ if $\mb G_{\check{v}}$ is empty} 
Let $i\in V(\Gamma)$ be a vertex such that $\mb G_i$ is empty. For every edge $\beta\in E(\Gamma)$ whose source is $i$ fix a bijection $\nu_\beta : G_i / H_\beta \to G_i$ with $\nu_\beta(H_\beta) = 1$. While any possible choice of $\nu_\beta$ works, different choices will induce different local bijections. To ensure that the local bijections we construct have a uniform Morseness, it is important that we fix the bijections $\nu_\beta$ once and for all. 

Let $v\in V(T)$ be a vertex with $\check v = i$ and let $\beta \in E(\Gamma)$ be an edge whose source is $i$. We denote the set of edges $\alpha\in E(T)$ whose source is $v$ and whose projection $\check \alpha$ to $\Gamma$ is $\beta$ by $A_{v, \beta}$. Recall that we can define a bijection $\psi_{v, \beta} : A_{v, \beta} \to G_i/H_\beta$ via $\psi_{v,\beta}(\alpha) = (\bar\alpha^*)^G H_\beta$. We can use this bijection to define the Morseness of cosets of $G_i/H_\beta$.

Observe the following. The sets $A_{v, \beta}$ and $E_{v, \beta}$ are defined very similarly. The only difference is that in $E_{v, \beta}$ we require that edges are outgoing. Since $T$ is a tree, there is at most one edge $\alpha\in E(T)$ with $\alpha^- = v$ that is not outgoing, namely the one whose inverse $\bar\alpha$ is on the path from $G_0$ to $v$.

\begin{definition}
A coset $hH_\beta\in G_i/H_\beta$ is $M$-Morse if there exists a closest point $p\in \pi^{-1}(\psi_{v,\beta}^{-1}(hH_\beta))^-$ to $v^*$ such that $[v^*, p]$ is $M$-Morse.
\end{definition}

Note that $v^*$ is equal to $(g_1, i)$ for some $g_1\in G$ and any point $p\in \pi^{-1}(\psi_{v, \beta}^{-1}(hH_\beta))$ has the form $p = (g_1hh_1, i)$ for some $h_1\in H_\beta$. As a consequence, this definition does not depend on the vertex $v$. Namely, if $u\in V(T)$ is another vertex with $\check u = i$ and $u^* = (g_2, i)$ for some $g_2\in G$, then $p' = (g_2hh_1, i)$ is a closest point on $\pi^{-1}(\psi_{u, \beta}^{-1}(hH_\beta))^-$ to $u^*$ if and only if $p = (g_1hh_1, i)$ is a closest point on $\pi^{-1}(\psi_{v, \beta}^{-1}(hH_\beta))^-$ to $v^*$. Similarly, the geodesic $[u^*, p']$ is $M$-Morse if and only if $[v^*, p]$ is $M$-Morse. In the following if it is clear which vertex $v$ of $V(T)$ we are working with, we will denote $A_{v, \beta}$ by $A_\beta$ and $\psi_{v, \beta}$ by $\psi_\beta$.

Although $\nu_\beta$ could be any bijection, it turns out that it satisfies the following property. This works only because we assume that $\mb G_i$ is empty.

\begin{lemma}\label{lemma:morseness_of_empty_bijecton}
There exists an increasing function $f_\empbij : \mc M\to \mc M$ such that for every edge $\beta\in E(\Gamma)$ whose source is $i$ the following holds. If a coset $hH_\beta\in G_i/H_\beta$ is $M$-Morse, then $\nu_\beta(hH_\beta)$ is $(f_\empbij(M), X')$-Morse. Conversely, if $g\in G_i$ is $(M, X')$-Morse, then the coset $\nu_\beta^{-1}(g)$ is $f_\empbij(M)$-Morse.
\end{lemma}

Recall that since $\Gamma = \Gamma'$, we can also think of $\beta$ as an edge of $\Gamma'$.

\begin{proof}
The Morse boundary $\mb G_i$ is empty. Thus for every Morse gauge $M$, there exist at most finitely many elements of $G_i/H_\beta$ and $G_i$ which are $M$-Morse. Hence for any Morse gauge $M$, the Morse gauge $f_\empbij(M)$ only has finitely many conditions it needs to satisfy and thus exists.
\end{proof}

Let $v\in V(T)$ be a vertex with $\check{v} = i$ and assume that $q(v)=w$ for some vertex $w\in V(T')$. We are now ready to construct the local bijection $q_v$. To do so, we will construct bijections $q_\beta: E_{v, \beta}\to E_{w, \beta}$ for all edges $\beta \in E(\Gamma)$ with $\beta^- = i$. Then we can define $q_v : E_v^+\cup \{v\}\to E_w^+\cup \{w\}$ via $q_v(v) = w$ and $q_v(\alpha^+) = q_\beta(\alpha)^+$ for edges $\alpha\in E_{v, \beta}$. Observe that this construction ensures that $\check {q_v(x)} = \check x$ for all vertices $x\in E_v^+$.

Fix an edge $\beta\in E(\Gamma)$ with $\beta^- = i$. Similar to $A_\beta$ we can define $A_\beta'$ as the set of edges $\alpha\in E(T')$ whose projection to $\Gamma$ is $\beta$ and whose source is $v$. Analogous to $\psi_\beta$, we have a bijection $\psi_\beta' : A_\beta'\to G_i/H_\beta'$. Since $H_\beta' =\{1\}$ we can view $\psi_\beta'$ as a bijection $\psi_\beta' : A_\beta'\to G_i$. Define the bijection $\tilde{q}_\beta : A_\beta \to A_\beta'$ as the composition $\psi_\beta'^{-1}\circ \nu_\beta \circ \psi_\beta$.

Recall that the sets $A_\beta$ and $E_{v, \beta}$ differ by at most one edge. Edges in $E_{v, \beta}$ have to be outgoing. Since $T$ is a tree, there is at most one edge $\alpha\in E(T)$ with $\alpha^- = v$ that is not outgoing, namely the one whose inverse $\bar\alpha$ is on the path from $G_0$ to $v$. Since we construct the local bijection $q_v$ via the bijections $q_\beta$ (both here and in the case that $\mb G_i$ is non-empty) we know that if $\alpha\in V(T)$ with $\alpha ^- =v$ is not outgoing and $\alpha'\in V(T')$ with $\alpha' = w$ is not outgoing, then their projections to $\Gamma$ coincide. This observation together with the fact that $\nu_\beta (H_\beta) =1$ implies that $\tilde{q}_\beta$ restricted to $E_{v, \beta}$ induces a bijection $q_\beta : E_{v, \beta} \to E_{w, \beta}$.

\begin{lemma}
The map $q_v$ constructed above is a geodesic local bijection of depth $1$ and Morseness $f:  \mc M\to \mc M$ for any increasing map $f$ such that $f(M)\geq \max\{f_2(f_\empbij(f_\emp(M, 1)),1), f_8(f_\empbij(f_2(M, 1)))\}$.
\end{lemma}

\begin{proof}
We need to prove that Conditions \ref{cond2:nestedness} - \ref{cond2:commutativity} hold. Conditions \ref{cond2:nestedness}  -\ref{cond2:partial_surj} follow directly from the fact that $q_v : E_v^+\to E_w^+$ is a bijection. 

Next we check that Condition \ref{cond2:morseness} holds. Let $\beta\in E(\Gamma)$ with $\beta^- = i$. Let $\alpha\in E_{v, \beta}$ be $M$-Morse. Lemma \ref{lemma:cosets_morseness_vs_closeness} implies that the coset $\psi_\beta(\alpha)$ is $M_1 = f_\emp(M, 1)$-Morse. Lemma \ref{lemma:morseness_of_empty_bijecton} implies that $g = \nu_\beta(\psi_\beta(\alpha))$ is $(M_2, X') = (f_\empbij(M_1), X')$-Morse. Let $\alpha' = {\psi_{\beta}^{-1}}'(g)$. Since $\mc G'$ is a trivial Morseless star, the endpoints of $[w^*, \alpha^*]$ and $w^*\cdot[(1, i), (g, i)]$ are at distance $1$ from each other. Lemma \ref{lemma:monster} \ref{prop:adapted2.5} implies that $\alpha'$ is $f_2(M_2, 1)$-Morse. Conversely, let $\alpha'\in E_{w ,\beta}$ be $N$-Morse and let $g = \psi_\beta'(\alpha')$. By the same argument as above, Lemma \ref{lemma:monster} \ref{prop:adapted2.5} implies that $g$ is $(N_1, X') = (f_2(N, 1), X')$-Morse. By Lemma~\ref{lemma:morseness_of_empty_bijecton}, $\nu_\beta^{-1}(g)$ is $N_2 = f_\empbij(N_1)$-Morse. Lemma~\ref{lemma:compare_morseness_of_repr} then implies that $\alpha$ is $f_8(N_2)$-Morse. Hence, $q_v$ indeed satisfies Condition \ref{cond2:morseness}. 

We know that $v$ and $w$ project to the same vertex $i$ in $\Gamma$. We also assumed that $\mb G_i$ is empty. Hence, $q_v$ is indeed geodesic and we do not need to check Conditions \ref{cond2:morseness:tails} - \ref{cond2:commutativity}.
\end{proof}

\subsection{Constructing the local bijection $q_v$ if $\mb G_{\check{v}}$ is non-empty} Let $i\in V(\Gamma)$ be a vertex such that $\mb G_i$ is not empty. We fix a Morse $G_i$-geodesic line $\lambda$ and a Morse gauge $M_0$ such that $\lambda$ is $(M_0, X)$, $(M_0, X')$ and $(M_0, G_i)$-Morse. 

The group $G_i$ embeds both into $X$ and $X'$. Hence certain paths can be viewed as living either in $X$ or in $X'$. Next, we show what happens to the Morseness of these paths when we switch the viewpoint. To do so we will repeatedly use Lemma \ref{lemma:properties_of_morseless_stars}. Let $g\in G_i$ be an element. If the geodesic $[(1, i), (g, i)]_X$ is $M$-Morse, then $[1, g]_{G_i}$ is $(f_3(M), G_i)$-Morse. Hence, $[(1, i), (g, i)]_{G_i}$ viewed as a quasi-geodesic in $X'$ is a $M_1 = f_3^2(M)$-Morse $C_2$-quasi-geodesic. By Lemma \ref{lemma:monster}\ref{prop:adapted2.5}, the geodesic $[(1, i), (g, i)]_{X'}$ is $f_2(M_1, C_2)$-Morse. Similarly, if $z\in \mb G_i$ is $M$-Morse viewed as an element of $\mb X$, then $z$ is $(f_3(M), G_i)$-Morse. If $\gamma$ is a $(f_3(M), G_i)$-Morse $G_i$-geodesic realisation of $z$, then $\gamma$ is a $M_1 = f_3^2(M)$-Morse $C_2$-quasi-geodesic. By Lemma \ref{lemma:monster}\ref{prop:realization1}, there exists a $f_2(M, C_2)$-Morse geodesic $\gamma'$ in $X'$ with $[\gamma'] = [\gamma]$. This motivates the following definition. 
\begin{definition}
The map $f_\mov: \mc M\to \mc M$ is the map defined by $f_\mov(M) = f_2(f_3^2(M), C_2)$ for all Morse gauges $M\in \mc M$.
\end{definition}

As shown above, the map $f_\mov$ is a control on the Morse gauge when changing the viewpoint from $X$ to $X'$ and vice versa.

Now we are ready to start the construction of the local bijection $q_v$ for a fixed vertex $v\in V(T)$ with $\check v = i$ and where $q(v) = w$ is defined. As in the case where $\mb G_i$ was empty, we construct bijections $q_\beta: E_{v, \beta} \to E_{w, \beta}$ for every edge $\beta\in E(\Gamma)$ whose source is $i$ and define $q_v$ as the map induced by these bijections. We will additionally construct maps $ E_{v, \beta}^+\to \mb G_i$ and $ E_{w, \beta}^+\to \mb G_i$ both denoted by $g$. Let $\beta\in V(\Gamma)$ be an edge with $\beta^-=i$. Enumerate the edges of $E_{w, \beta}$. We construct $q_\beta$ iteratively and the enumeration will determine in which order we construct $q_\beta$. The construction of $q_\beta$ is depicted in Figure \ref{picture:mapping_amalgam}. Repeat the following steps.

\begin{figure}\centering
\includegraphics[width= \linewidth]{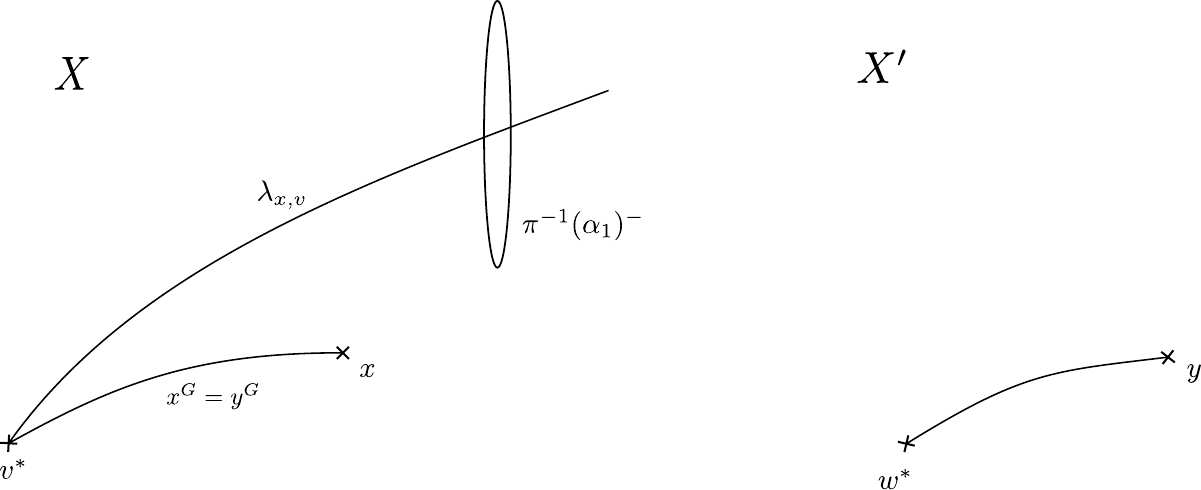}
\caption{Construction of $q_\beta(\alpha_2)^{-1}$}
\label{picture:mapping_amalgam}
\end{figure}

\begin{itemize}
    \item Let $\alpha_2$ be the lowest index element of $E_{w, \beta}$ which is not yet in the image of $q_\beta$ and denote $\bar\alpha_2^*$ by $y$. Let $x\in \pi^{-1}(v)$ be the vertex such that $x^G = y^G$. Recall that $x^G$ is defined in Definition \ref{def:xg}. 
    \item If there exists an edge $\alpha_1\in E_{v, \beta}$ for which $q_\beta$ is not defined and such that $\bar\alpha_1^* = x$, then define $q_\beta(\alpha_1) = \alpha_2$. If this happens, we say that $\alpha_1$ is exact.
    \item Otherwise pick any edge $\alpha_1\in E_{v, \beta}$ for which the following conditions hold and define $q_\beta (\alpha_1) = \alpha_2$.
    \begin{enumerate}
        \item $q_\beta(\alpha_1)$ is not defined. \label{cond3:not_matched}
        \item $\pi^{-1}(\alpha_1)$ and $\lambda_{x, v}$ have a non-empty intersection. Recall that $\lambda_{x, v}$ and $\gamma_x$ are defined in Definition \ref{def:corresponding_rays}\label{cond3:intsersect}
        \item $d_X(v^*, \pi^{-1}(\alpha_1))\geq d_X(v^*, x)  + D_\dlam(M, d_X(v^*, x)) + T(M)$, where $M$ is a Morse gauge such that $[v^*, x]$ is $M$-Morse and $T:\mc M\to \R_{\geq1}$ is a map. Note that this will allow us to apply Lemma \ref{lemma:corresponding_geodesic}\eqref{prop:nice_connections}. The map $T$ should be sufficiently large and in the proof of Lemma \ref{lemma:amalgam_local_bijection} the exact condition $T$ needs to satisfy is stated.  \label{cond3:far_away}
    \end{enumerate}
    In this case, we say that $\alpha_1$ is not exact.
    \item Define $g(\alpha_2^+) = g(q_\beta^{-1}(\alpha_2)^+)  =  \gamma_{x}$ and observe that $\gamma_x = \gamma_y$. 
\end{itemize}

\begin{lemma}
The map $q_\beta : E_{v, \beta}\to E_{w, \beta}$ is a well defined bijection. 
\end{lemma}

\begin{proof}
The sets $\pi^{-1}(\alpha)$ for $\alpha\in A_\beta$ partition $\pi^{-1}(v)$. By Lemma \ref{lemma:properties_of_morseless_stars}\eqref{constant:edge_group_intersection2}, the intersection of $\lambda_{x, v}$ and any edge group $\pi^{-1}(\alpha)$ is finite. Hence, there are infinitely many edges $\alpha\in E_{v, \beta}$ that satisfy \eqref{cond3:intsersect}. Only finitely many edges  $\alpha\in E_{v, \beta}$ violate \eqref{cond3:not_matched} or \eqref{cond3:far_away}. Hence, an edge $\alpha_1\in E_{v, \beta}$ satisfying \eqref{cond3:not_matched} - \eqref{cond3:far_away} exists.

From the construction it is clear that the map $q_\beta$ is injective and surjective. The only thing left to argue is that $q_\beta$ is defined on the whole domain $E_{v, \beta}$. This is indeed the case; the edge groups $\pi^{-1}(\alpha_2)^-$ for edges $\alpha_2\in E_{w, \beta}$ consist of single vertices. Hence, for every edge $\alpha_1\in E_{v, \beta}$ there exists an edge $\alpha_2\in E_{w, \beta}$ such that $(\bar\alpha_2^*)^G = (\bar\alpha_1^*)^G$. Observing that $\bar\alpha_2^*\neq \bar\alpha^*$ for distinct $\alpha, \alpha_2\in E_{v ,\beta}$ concludes the proof.
\end{proof}

\begin{lemma}\label{lemma:amalgam_local_bijection}
For sufficiently large $T$, the map $q_v$ constructed above is a geodesic local bijection of depth $1$ and Morseness $f:  \mc M\to \mc M$ for some increasing function $f$.
\end{lemma}

We will see in the proof that if $f$ satisfies a finite set of conditions, which do not depend on the vertex $v$, then $q_v$ is a geodesic local bijection of depth $1$ and Morseness $f$.

\begin{proof}

We need to prove that Conditions \ref{cond2:nestedness} - \ref{cond2:commutativity} hold. Again, Conditions \ref{cond2:nestedness}  -\ref{cond2:partial_surj} follow directly from the fact that $q_v : E_v^+\to E_w^+$ is a bijection. We can define $p_v$ as the identity map from $\mb G_i$ to itself. Condition \ref{cond2:morseness:tails} holds if $f(M)\geq f_\mov(M)$. Condition \ref{cond2:commutativity} follows directly from the definition of $p_v$ and $g$.\\

Next we prove Condition \ref{cond2:morseness} and Condition \ref{cond2:geodesics}. It is enough to show that Condition \ref{cond2:closeness1} holds for all elements of $F_2$, since Condition \ref{cond2:morseness} then implies that for large enough $f$, Condition \ref{cond2:closeness1} also holds for elements of $F_1$. Also observe that $F_1 = E_{v}^+$ and $F_2 = E_{w}^+$. In the following, we will use the notation from the construction of $q_\beta$. 

Let $\alpha_2\in E_{w}$ be $N$-Morse. Lemma \ref{lemma:monster}\ref{prop:adapted2.5} implies that $[w^*, y]$ is $N_1 = f_2(N, 1)$-Morse. Lemma \ref{lemma:corresponding_geodesic}\eqref{prop:morseness} implies that $g(\alpha_2^+)$ is $N_2 = f_\lam(\max\{M_0, N_1\})$-Morse. Conversely, if $g(\alpha_2^+)$ is $N'$-Morse, then Lemma \ref{lemma:corresponding_geodesic}\eqref{prop:morseness} implies that $[w^*, y]$ is $N_1' = f_\lam(\max\{N', M_0\})$-Morse and by Lemma \ref{lemma:monster}\ref{prop:adapted2.5}, the edge $\alpha_2$ is $N_2' = f_2(N_1', 1)$-Morse. Hence Condition \ref{cond2:closeness1} holds for all elements of $F_2$ for large enough $f$. Furthermore, since $[w^*, y]$ is $N_1$-Morse, the geodesic $[v^*, x]$ is $N_3 = f_\mov(N_1)$-Morse. Thus Condition \ref{cond2:morseness} holds for large enough $f$ if $\alpha_1 = q_\beta^{-1}(\alpha_2)$ is exact. If $\alpha_1$ is not exact we know that $\pi^{-1}(\alpha_1)^-$ intersects $\lambda_{x, v}$. Let $a$ be a vertex in that intersection. Lemma \ref{lemma:corresponding_geodesic}\eqref{prop:morseness} implies that $\lambda_{x,v}$ is $N_4 = f_\lam(\max\{N_3, M_0\})$-Morse. Hence, by Lemma \ref{lemma:monster}\ref{prop:adapted2.5} we have that $[v^*, a]$ is $N_5 = f_2(N_4, C_2)$-Morse. Lemma \ref{lemma:compare_morseness_of_repr} implies that $\alpha_1$ is $N_6 = f_8(N_5, 1)$-Morse and thus in this case, Condition \ref{cond2:morseness} holds if $f(N)\geq N_6$.

Conversely, let $\alpha_1\in E_{v}$ be $N$-Morse. Let $\alpha_2 = q_\beta(\alpha_1)$. Lemma \ref{lemma:monster}\ref{prop:adapted2.5} implies that $[v^*, \bar\alpha_1^*]$ is $N_1 = f_2(N, 1)$-Morse. If $\alpha_1$ is exact, then $x = \bar\alpha_1^*$. Hence $[w^*, y]$ is $N_2 = f_\mov(N_1)$-Morse. By Lemma \ref{lemma:monster}\ref{prop:adapted2.5}, the edge $\alpha_2$ is then $N_3 = f_2(N_2, 1)$-Morse. So if $f(N)\geq N_3$ Condition \ref{cond2:morseness} holds if $\alpha_1$ is exact. In the following we assume that $\alpha_1$ is not exact. In this case the proof of \ref{cond2:morseness} is more complicated and depicted by Figure \ref{picture:mapping_amalgam_proof}.

\begin{figure}\centering
\includegraphics[width= \linewidth]{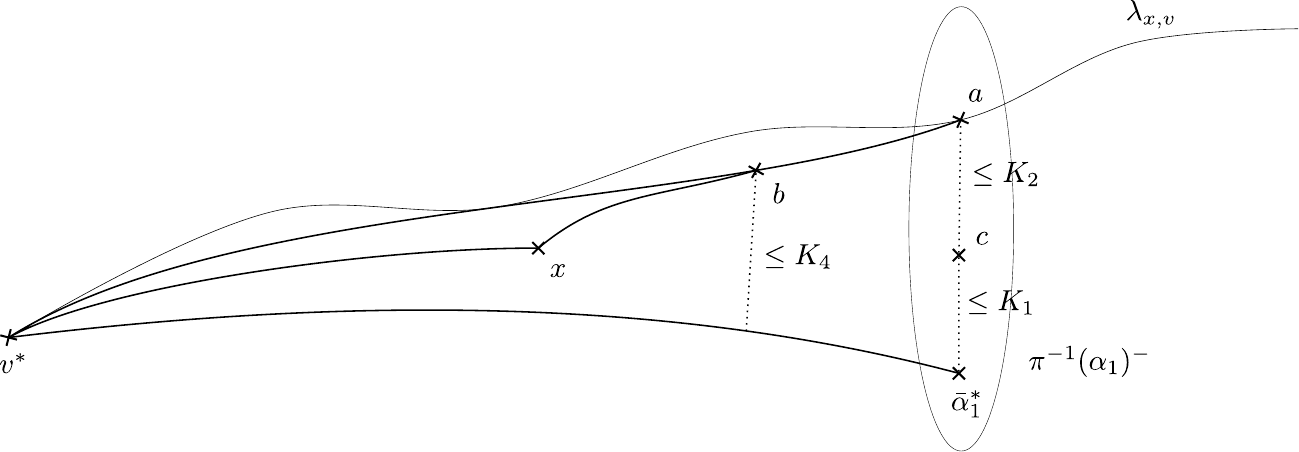}
\caption{Proof of Condition \ref{cond2:morseness} for elements of $F_1$}
\label{picture:mapping_amalgam_proof}
\end{figure}

If $\alpha_1$ is not exact we have that  $\pi^{-1}(\alpha_1)^-\cap \lambda_{x, v}$ is not empty and hence contains a vertex, say $a$. In the construction of $\alpha_1$, we assumed that $[v^*, x]$ was $M$-Morse. It is not clear how (and if) $M$ and $N$ relate to each other. Lemma \ref{lemma:corresponding_geodesic} \eqref{prop:morseness} implies that $\lambda_{x, v}$ is $M_1 = f_\lam(\max\{M, M_0\})$-Morse. Since $[v^*, a]$ has its endpoints on $\lambda_{x, v}$ it is $M_2 = f_2(M_1, C_2)$-Morse by Lemma \ref{lemma:monster}\ref{prop:adapted2.5}. Let $c$ be a vertex in $\pi^{-1}(\alpha_1)$ closest to $v^*$. Lemma \ref{lemma:cosets_morseness_vs_closeness} implies that $d_X(\bar\alpha_1^*, c)\leq K_1 = D_5(N_1, 1)$ and that $d_X(c, a)\leq K_2 = D_5(M_2, 1)$. If $K_2\leq K_1$, define $b = a$. Otherwise, the concatenation $[v^*, \bar\alpha_1^*][\bar\alpha_1^*, a]$ is a $2K_2$-quasi-geodesic with endpoints on $[v^*, a]$. Hence, $d_X([v^*, \bar\alpha_1^*](t), [v^*, a])\leq M_2(2K_2) = K_3$ for all $t\in [0, d_X( v^*, \bar\alpha_1^*)]$. Since $[v^*, a]$ and $[v^*, \bar\alpha_1^*]$ are geodesics with the same starting point, $d_X([v^*, \bar\alpha_1^*](t), [v^*, a](t))\leq 2K_3$ for all $t\in [0, d_X( v^*, a)- K_3]$. If $d_X(v^*, a)\geq 13K_3$ we can apply Lemma \ref{lemma:get_close_again} to get that $d_X([v^*, a](t), [v^*, \bar\alpha_1^*](t))\leq \delta_{N_1}$ for all $t\in [0, d_X(v^*, a) - 5K_3]$ and we choose $b$ on $[v^*, a]$ such that $d_X(v^*, b) = d_X(v^*, a) - 5K_3$. 

So, if $T(M)\geq 13K_3$, then no matter whether $K_2\leq K_1$ or not, $b$ is in the $K_4 = 2K_1 + \delta_{N_1}+1$ neighbourhood of $[v^*, \bar\alpha_1^*]$. Note that $K_3$ depends only on $M$ and this is the only condition we impose on $T$. By \eqref{cond3:far_away} we know that $d_X(v^*, b)\geq d_X(v^* ,x) + D_\dlam(M, d_X(v^*, x))$. Lemma \ref{lemma:corresponding_geodesic}\eqref{prop:nice_connections} implies that $[b, x]$ is $f_\lam(M_0)$-Morse. Since $[v^*, b]$ has endpoints in a $K_4$-neighbourhood of $[v^*, \bar\alpha_1^*]$ it is $N_4 = f_2(N_1, K_4)$-Morse by Lemma \ref{lemma:monster}\ref{prop:adapted2.5}. Lemma \ref{lemma:monster}\ref{lemma:triangles} applied to the triangle $([v^*, x], [x,b], [b, v^*])$ implies that $[v^*, x]$ is $N_5 = f_1(\max\{N_4, f_\lam(M_0)\})$-Morse. Thus, $[w^*, y]$ is $N_6 = f_\mov(N_5)$-Morse and Lemma \ref{lemma:monster}\ref{prop:adapted2.5} implies that $\alpha_2$ is $N_7 = f_2(N_6, 1)$-Morse. So if $f(N)\geq N_7$, Condition \ref{cond2:morseness} holds in this case. Also note that even though we use the Morse gauge $M$ in the proof, $N_7$ does not depend on $M$. Hence for large enough $f$, Condition \ref{cond2:morseness} holds as a whole.

The only thing left to prove is Condition \ref{cond2:closeness}. To do so we assume that the map $f$ is large enough, that is if $\tilde f$ is a map satisfying all conditions posed on $f$ so far, we assume that $f\geq \tilde f ^2$. In particular, $f$ satisfies Condition \ref{cond2:closeness1}.

Let $\alpha_2\in E_w$ be an $N$-Morse edge. We know that $\gamma_y$ and $[w^*, y]$ are both $f(N)$-Morse. If $D_v(N, k)\geq D_\dlam(f(N), k)+1$, then Lemma \ref{lemma:corresponding_geodesic}\eqref{prop:contained_neighbourhoods} implies that Condition \ref{cond2:closeness} holds in this case.

Let $\alpha_1\in E_v$ be an $N$-Morse edge which is exact. We know that $[v^*, x]$ and $\gamma_x$ are $f(N)$-Morse. If $D_v(N, k)\geq D_\dlam(f(N), k)+1$, then Lemma \ref{lemma:corresponding_geodesic}\eqref{prop:contained_neighbourhoods} implies that Condition \ref{cond2:closeness} holds in this case.

Let $\alpha_1\in E_v$ be an $N$-Morse edge that is not exact. We have shown that $[v^*, \bar\alpha_1^*]$ and $q_\beta(\alpha_1)$ are $\tilde f(N)$-Morse and $\gamma_x$ is $f(N)$-Morse. Thus, $ [v^*, a]$ is $f(N)$-Morse, where $a$ is a vertex in the intersection $\pi^{-1}(\alpha_1)^-\cap \lambda_{x, v}$. Let $\gamma$ be an $f(N)$-Morse realisation of $[\lambda_{x, v}] = v^*\cdot \gamma_x$ starting at $v^*$. Lemma \ref{lemma:monster}\ref{lemma:close_representatives} implies that $d_X(a, \gamma)\leq K_1 = D_1(f(N), C_2)$. Lemma 2.7 of \cite{Cor16} implies that $\gamma\in O_{k_1}^{f(N)}([v^*, a], v^*)$ for $k_1 \leq d_X(a, v^*) - \delta_{f(N)} - K_1 $. Let $c\in \pi^{-1}(\alpha_1)^-$ be a closest point to $v^*$. Applying Lemma \ref{lemma:cosets_morseness_vs_closeness} twice we get that $d_X(a, \bar\alpha_1^*)\leq 2D_5(f(N), 1) = K_2$. Lemma 2.7 of \cite{Cor16} implies that $[v^*, a]\in O_{k_2}^{f(N)}([v^*, \alpha_1^*], v^*)$ for $k_2 \leq d_X(v^*,a) - \delta_{f(N)} - K_2 $. Now we can apply Lemma \ref{lemma:get_close_again} to $[v^*, \alpha_1^*]$ and $\gamma$. We get that if $k_1$ and $k_2$ are at least $12\delta_{f(N)}$, then $\gamma\in O_{k_3}^{f(N)}(\alpha_1^*, v^*)$ and conversely $\alpha_1^*\in O_{k_3}^{f(N)}(\gamma, v^*)$ for $k_3\leq \min\{k_1, k_2\} - 4\delta_{f(N)}$. Hence, if $D_v(N, k)\geq k + 20 \delta_{f(N)} + K_1 + 2K_2 +1$, then Condition \ref{cond2:closeness} is satisfied in this case.

Choosing $D_v$ large enough that the conditions in all three cases are satisfied concludes the proof. 

\end{proof}

\begin{proof}[Proof of Theorem \ref{theorem:star_of_groups}] 
The constructed maps $q_v$ are geodesic local bijections of depth $1$ and Morseness $f$ for some increasing map $f: \mc M\to \mc M$. Moreover, the collection $\{q_v\}_{v\in V(T)}$ satisfies all assumptions of Proposition \ref{prop:local_bij_imply_bass_serre_map} and hence induces a continuous Bass-Serre tree map. Proposition \ref{prop:bass_serre_implies_homeo} concludes the proof. 
\end{proof}

\section{Free product with empty Morse boundary} \label{section:doublemap}

In this section we prove Theorem \ref{theorem:empty_morse_boundary}. That is, we prove that if $G$ and $H$ are finitely generated groups such that $G$ is not hyperbolic and $H$ is infinite and has empty Morse boundary, then $\mb (G\ast G) \cong \mb (G\ast H)$. To do so, we need to construct a bijection from $G$ to $G\times H - \{1\}$ that can be viewed as a geodesic local bijection. 

In the first part of this section, we construct a sequence of geodesic lines $\lambda^1, \lambda^2, \lambda^3,$ and so on whose Morseness is increasingly bad. In the second part of this section we then use these geodesic lines to construct a bijection $q$ from $G$ to $G\times H -\{1\}$. Namely, we require that for (most) elements $(g, h)$, the preimage $q^{-1}(g, h)$ lies on the to $g$ corresponding ray $\lambda^i_g$, where $i$ depends only on $h$. We then show that $q$ indeed can be viewed as a geodesic local bijection. Lastly, we apply the results from Section \ref{section:homeo_tools} to show that the constructed bijection induces a homeomorphism of the Morse boundaries of $G*G$ and $G*H$.

Theorem 1.2 of \cite{graph_of_groups1} states that the Morse boundary of a free product only depends on the Morse boundary of its factors. Hence, if the Morse boundary of $G$ is empty, the statement follows directly. Thus for the rest of this section we assume that $\mb G$ is not empty. Let $\mc G$ and $\mc G'$ be the graph of groups corresponding to the free products $G*G$ and $G*H$ respectively. In particular, their underlying graphs $\Gamma$ and $\Gamma'$ consist of two vertices $0$ and $1$ joined by an edge. The graph of groups $\mc G$ and $\mc G'$ are trivial Morseless stars and hence are Morse preserving. This allows us to use the tools developed in previous sections. Also observe that the vertex groups $G_0$ and $G_1$ are both isomorphic to $G$ and we assume that $G_0'$ is isomorphic to $G$ while $G_1'$ is isomorphic to $H$.

Also, since both $\mc G$ and $\mc G'$ are trivial Morseless stars, their vertex groups are not only quasi-isometrically embedded in their the Bass-Serre space $X$ and $X'$ but even isometrically embedded. As a consequence we do not need to distinguish whether paths are (quasi)-geodesics with respect to the metrics $d_G$ or $d_X$. In this section, unlike the rest of the paper, we say that a quasi-geodesic is $M$-Morse if it is $(M, G)$-Morse. 

As in Section \ref{section:relativemap} we will construct a collection of local bijections $\{q_v\}$ and then use the results of Section~\ref{section:homeo_tools} to conclude that the Morse boundaries $\mb (G*G)$ and $\mb (G*H)$ are homeomorphic. But first, we need to set up the construction.

\begin{lemma}\label{lemma:existance_of_bad_geodesic_segments} There exists a constant $C\in \R$ and a sequence of geodesic segments $(s_i)_{i\in \N}$ in $G$ such that the following holds. For every $i\in \N$ there exists a $ C$-quasi-geodesic $\gamma_i$ with endpoints on $s_i$ but not contained in the $i$-neighbourhood of $s_i$.
\end{lemma}

\begin{proof}
For every constant $C\in \R$, either there exists such a sequence $s_i$ or there exists a constant $K = K(C)$ such that all $C$-quasi-geodesics with endpoints on a geodesic segment are contained in the $K$-neighbourhood of that geodesic. So assume by contradiction that for all constants $C$, the latter is true. Then, for any Morse gauge $M$ with $M(\lambda, \eps) \geq  K(\max\{\lambda, \eps\})$ for all quasi-geodesic constants $(\lambda, \eps)$, every geodesic in $G$ is $M$-Morse. Hence, $G$ is hyperbolic, which we have assumed it is not. 
\end{proof}

\begin{definition}
Let $\gamma$ be a Morse quasi-geodesic and $C\in \R$ a constant. We say that the critical value of $\gamma$ at $C$ is the minimal value $t\in \R$ such that every $C$-quasi-geodesic $\lambda$ with endpoints on $\gamma$ is contained in a $t$-neighbourhood of $\gamma$. 
\end{definition}

Note that the minimum exists since we used closed neighbourhoods in the definition of Morse geodesics.

For the rest of this section, we fix a Morse geodesic line $\lambda$ with $\lambda(0) = 1$ in $G$ and an element $h_0\in H - \{1\}$. Furthermore we fix a Morse gauge $M_0$ such that $\lambda$, $\lambda[0, \infty)$ and $\lambda[0, -\infty)$ are $M_0$-Morse and $h_0$ is $(M_0, X')$-Morse. 

\begin{lemma}\label{lemma:existance_dual_segments}
Let $N$ be a Morse gauge. There exists a constant $C\in \R$ and sequences of geodesic segments $(a_n)_{n\in \N}$ and $(b_n)_{n\in N}$ with $a_n^+ = b_n^-$ and such that for every $n\in \N$ the concatenation $a_nb_n$ is a geodesic and the following holds. The critical values of $a_n$ and $b_n$ at $C$ are at least $n$, the concatenation $a_nb_n$ is not $f_1(M_0)$-Morse and the segments $a_n$ and $b_n$ are not $N$-Morse. 
\end{lemma}

\begin{proof}
By Lemma \ref{lemma:existance_of_bad_geodesic_segments} there exists a constant $C'$, a sequence of geodesics $s_n: [0, l_n]\to G$ and a sequence of $C'$-quasi-geodesics $\gamma_n : [0, t_n]\to G$ such that $\gamma$ has endpoints on $s_n$ but is not in the $n$-neighbourhood of $s_n$. 

\begin{figure}\centering
\includegraphics[width= \linewidth]{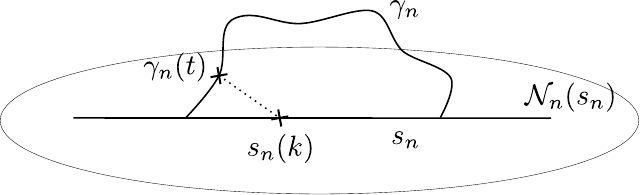}
\caption{Proof of Lemma \ref{lemma:existance_of_bad_geodesic_segments}}
\label{picture:lemma6.4}
\end{figure}

Fix an integer $n$. Let $k\in [0, l_n]$ be a positive integer (or let $k = l_n$). We can consider the subsegments $s_n[0, k]$ and $s_n[k, l_n]$. If both endpoints of $\gamma_n$ are on the same subsegment, that subsegment has critical value at $C'$ (and hence at $3C'$) of at least $n$. Assume that $\gamma_n$ has an endpoint on each of the subsegments $s_n[0, k]$ and $s_n[k, l_n]$. As depicted in Figure \ref{picture:lemma6.4}, let $\gamma_n(t)$ be the closets point on $\gamma_n$ to $s_n(k)$. Lemma \ref{quasi-geodesic} implies that the concatenations $[s_n(k), \gamma_n(t)]\gamma_n[t, l_n]$ and $\gamma_n[0, t][\gamma_n(t), s_n(k)]$ are $3C'$-quasi-geodesics. At least one of those quasi-geodesics is not contained in the $n$-neighbourhood of $s_n$. Hence, one of the subsegments $s_n[0, k]$ and $s_n[k, l_n]$ has critical value at $3C'$ of at least $n$. 

Thus, no matter where the endpoints of $\gamma$ are, at least one of the subsegments $s_n[0, k]$ and $s_n[k, l_n]$ has critical value at $3C'$ of at least $n$. Assume there exists an integer $k_0\in [0 ,l_n]$ such that both $s_n[k_0, l_n]$ and $s_n[0, k_0+1]$ (respectively $s_n[0, l_n]$, if $k_0+1\geq l_n$) have critical value at $3C'$ of at least $n$. Any $3C'$-quasi-geodesic with endpoints on $s_n[0, k_0+1]$ can be extended to a $(3C'+2)$-quasi-geodesic with endpoints on $s_n[0, k_0]$. Thus, the critical value of $a_n = s_n[0, k_0]$ and $b_n = s_n[k_0, l_n]$ at $3C'+2$ is at least $n$. If this is true for every integer $n$, the critical values of $a_n$ and $b_n$ at $3C'+2$ are arbitrarily large for large enough $n$. In particular, for any Morse gauge $M$, they are larger than $M(3C'+2)$. Therefore, by passing to a subsequence of $s_n$ we can make sure that all the required conditions are satisfied.

Next we prove that either $s_n[0, 0]$ has critical value at $3C'$ of at least $n$ (in which case we can choose $k=0$ and also have that $s_n[0, k]$ and $s_n[k, l_n]$ have critical value at $3C'+2$ of at least $n$) or that such an integer $k_0$ indeed exists. Assume that $s_n[0, 0]$ does not have critical value at $3C'$ of at least $n$ and let $k_1\geq 0$ be the smallest integer such that $s_n[0, k_1+1]$ (or $s_n[0, l_n]$ if $k_1+1\geq l_n$) has critical value at $3C'$ of at least $n$. Since $s_n[0, k_1]$ does not have critical value at $3C'$ of at least $n$, $s_n[k_1, l_n]$ does. Hence $k_1$ is a possible choice for $k_0$.
\end{proof}

Let $N = f_2(f_2(f_2(M_0, M_0(3, 0)), 3), 3)$. For the rest of this section we fix a sequence $a_n$ and $b_n$ that satisfy Lemma \ref{lemma:existance_dual_segments} for $N$. We denote the concatenation $a_nb_n$ by $s_n$. We will use this sequence to construct a sequence of Morse geodesic lines $\lambda^n$. The construction is depicted in Figure \ref{picture:construction_of_lambda_n}.

\begin{figure}\centering
\includegraphics[width= \linewidth]{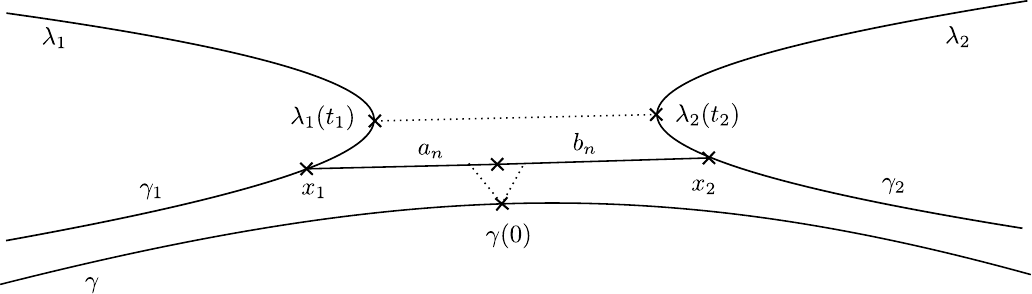}
\caption{Construction of $\lambda^n$}
\label{picture:construction_of_lambda_n}
\end{figure}

\textbf{Construction:} Recall that we have fixed an $M_0$-Morse geodesic line $\lambda$. Let $n$ be a positive integer. Let $x_1$ and $x_2$ be the endpoints of $s_n$. Consider the geodesic lines $\lambda_1 = x_1\cdot \lambda$ and $\lambda_2 = x_2\cdot \lambda$. Let $\lambda_1(t_1)$ and $\lambda_2(t_2)$ be closest points on $\lambda_1$ and $\lambda_2$. For $i\in \{1, 2\}$, let $\gamma_i$ be one of the subrays $\lambda_i[0, \infty)$ and $\lambda_i[0, -\infty)$ whose interior does not contain $\lambda_i(t_i)$. If $[\gamma_1] = [\gamma_2]$, then Lemma \ref{lemma:monster} \ref{lemma:triangles} about triangles applied to the triangle with sides $\gamma_1, \gamma_2$ and $s_n$ would imply that $s_n$ is $f_1(M_0)$-Morse, which it is not. Hence we know that $[\gamma_1]\neq [\gamma_2]$ and there exists a Morse geodesic line $\gamma$ such that $[\gamma[0, -\infty) ]= [\gamma_1]$ and $[\gamma[0, \infty)] = [\gamma_2]$. We reparametrize $\gamma$ such that the distance of $\gamma(0)$ to $a_n$ and $b_n$ is equal. A priori, it is not clear that such a point always exist but we will prove the existence of such a point together with the next Lemma. We denote $\gamma(0)^{-1}\cdot\gamma$ by $\lambda^n$.

\begin{lemma}\label{lemma:constructing_geodesics}
There exists a function $f_\blam: \mc M\to \mc M$ such that for every positive integer $n$ the following properties hold.
\begin{enumerate}[label=\roman*)]
    \item If $s_n$ is $M$-Morse, then $\lambda^n$ is $f_\blam(M)$-Morse.\label{cond3:segment_to_ray}
    \item If $\lambda^n[0, \infty)$ is $M$-Morse, then $b_n$ is $f_\blam(M)$-Morse.\label{cond3:ray_to_semment_right}
    \item If $\lambda^n[0, -\infty)$ is $M$-Morse, then $a_n$ is $f_\blam(M)$-Morse.\label{cond3:ray_to_segment_left}
    \item There exists an integer $T_n$ such that for all $t\geq T_n$ $\lambda^n[t, \infty)$ and $\lambda^n[-t,-\infty)$ are $f_\blam(M_0)$-Morse.\label{cond3:nice_at_some_point}
\end{enumerate}
\end{lemma}

Note that the map $f_\blam$ is allowed to depend on $M_0$. 

\begin{proof}
We first prove Property \ref{cond3:nice_at_some_point} (without using any properties of $\gamma(0)$) and then use this to prove that we can indeed reparametrize $\gamma$ such that the distance from $\gamma(0)$ to $a_n$ and $b_n$ is the same.

Let $K$ be an integer, let $\gamma(t)$ be the point on $\gamma$ closest to $\gamma_2(0)$ and let $\gamma(s)$ be a point on $\gamma$ closest to $\gamma_2(K)$. The rays $\gamma[0, \infty)$ and $\gamma_2$ have finite Hausdorff distance. Thus there exists $K$ large enough, so that we can apply Lemma \ref{quasi-geodesic}, implying that the concatenation $[\gamma_2(0),  \gamma(t)]\gamma[t, s][\gamma(s), \gamma_2(K)]$ is a $(3, 0)$-quasi-geodesic. Hence, $\gamma[t, s]$ is in the $K_1 = M_0(3, 0)$-neighbourhood of $\gamma_2$. Since we can use this argument for any $K$ which is large enough we get that the whole ray $\gamma[t,  \infty)$ is in the $K_1$-neighbourhood of $\gamma_2$ and hence is $N_1 = f_2(M_0, K_1)$-Morse by Lemma \ref{lemma:monster}\ref{prop:adapted2.5}. We can repeat the same argument for $\gamma[0, \infty)$ and $\gamma_1$. Thus Property \ref{cond3:nice_at_some_point} holds for $\gamma$ and its translate $\lambda^n$. 

Now we prove that $\gamma(0)$ is well-defined. Consider the function $h(t) = d_G(\gamma(t), a_n) - d_G(\gamma(t), b_n)$. This function is continuous and $\gamma(0)$ being well-defined is equivalent to $0$ being in the image of $h$. Since $h$ is continuous it is enough to prove that $h$ assumes both non-negative and non-positive values. Let $t\geq T_n$ be an integer and let $p$ be a closest point to $\gamma(t)$ on $s_n$. By Lemma \ref{quasi-geodesic} the concatenation $\eta = [\gamma(t), p][p, b_n^+]_{s_n}$ is a $(3, 0)$-quasi-geodesic. Here $[p, b_n^+]_{s_n}$ denotes the subsegment of $s_n$ from $p$ to $b_n^+$. Lemma \ref{lemma:monster}\ref{lemma:triangles} applied to the quasi-geodesic triangle $(\gamma[t, \infty), \gamma_2^{-1}, \eta^{-1})$ implies that $\eta$ is $f_2(N_1, 3)$-Morse. If $h(t)$ was negative, then $b_n$ would have its endpoints on $\eta$ and hence be $f_2(f_2(N_1, 3), 3)$-Morse by Lemma \ref{lemma:monster}\ref{prop:adapted2.5}. This is not the case, so $h(t)$ indeed assumes non-negative values. Similarly, one can prove that $h(t)$ assumes non-positive values. Thus, $\gamma(0)$ is indeed well-defined.

Let $M$ be a Morse gauge and let $M_1 = \max\{M, M_0\}$. The geodesics $\gamma_1^{-1}$, $s_n$, $\gamma_2$ and $\gamma^{-1}$ form a 4-gon. If $s_n$ is $M$-Morse, then by Lemma \ref{lemma:monster}\ref{cond:n_gon} applied to the $4$-gon above $\gamma$ is $f_1^2(M_1)$-Morse, which shows that Property \ref{cond3:segment_to_ray} holds.

Assume that $\lambda^n[0,\infty)$ and hence $\gamma[0, \infty)$ is $M$-Morse. Let $p$ be the closest point to $\gamma(0)$ on $s_n$. We constructed $\gamma(0)$ in a way we can assume that $p$ lies on $a_n$. The concatenation $\eta = [\gamma(0), p][p, b_n^+]_{s_n}$ is a $(3, 0)$-quasi-geodesic by Lemma \ref{quasi-geodesic}. Here $[p, b_n^+]_{s_n}$ denotes the subsegment of $s_n$ from $p$ to $b_n^+$. Lemma \ref{lemma:monster}\ref{lemma:triangles} applied to the quasi-geodesic triangle $(\gamma[0, \infty), \gamma_2^{-1}, \eta^{-1})$ shows that $\eta$ is $M_2 = f_2(M_1, 3)$-Morse. Since $b_n$ has its endpoints on $\eta$ it is $M_3 = f_2(M_2, 3)$-Morse by Lemma \ref{lemma:monster}\ref{prop:adapted2.5}. This concludes the proof of Property \ref{cond3:ray_to_semment_right}. Property \ref{cond3:ray_to_segment_left} can be proven analogously.
\end{proof}

Recall that we construct  the ray $\lambda_g^n$ ray corresponding to an element $g\in G$ using the base line $\lambda^n$ as follows. We translate $\lambda^n$ by $g$ and then choose a closest point, say $x_0 = g\cdot \lambda^n(t)$, of $1$ on $g\cdot\lambda^n$. If $t> 0$ we say that $\lambda_g^n$ is a realisation of $g\cdot \lambda^n[0, -\infty)$ and otherwise we define $\lambda_g^n$ as a realisation of $g\cdot \lambda^n[0, \infty)$.

\begin{lemma}\label{lemma:corresponding2}
There exists an increasing function $f_\clam: \mc M\to \mc M$ and an increasing function $D_\dclam: \mc M\to \R_{\geq1}$ such that for every element $g\in G$ and positive integer $n\in \N$ the following conditions hold.
\begin{enumerate}[label=\roman*)]
    \item If $\lambda_g^n$ is $M$-Morse, then $g$ is $f_\clam(M)$-Morse and either $\lambda^n[0, \infty)$ or $\lambda^n[0, -\infty)$ is $f_\clam(M)$-Morse.\label{lemma:retrieve_Morseness}
    \item If $g$ and $\lambda^n$ are $M$-Morse, then $\lambda_g^n$ is $f_\clam(M)$-Morse.\label{lemma:to_corresponding}
    \item There exists an integer $T_g^n$ such that if $x$ is on $\lambda_g^n$, and $d_G(1, x)\geq T_g^n$ and $[1, x]$ is $M$-Morse, then $\lambda_g^n$ is $f_\clam(M)$-Morse.\label{lemma:far_away_enough2}
    \item If $g$ and $\lambda^n$ are $M$-Morse, and $d_G(1, g)\geq D_\dclam(M, k)$, then 
    \begin{align}
    g\in \tilde{O}^M_k(\lambda^n_g)\qquad\text{and}\qquad \lambda_g^n\in O_k^M(g).
    \end{align}\label{prop:neighbourhoods2}
\end{enumerate}
\end{lemma}

\begin{proof} 

For the rest of the proof we assume that $t\leq0$. The proof if $t>0$ goes analogously by replacing $\lambda^n[t, \infty)$ with $\lambda^n[-t, -\infty)$ and so on. 

The concatenation $\gamma = [1, x_0]g\cdot\lambda^n[t, \infty)$ is a $(3, 0)$-quasi-geodesic by Lemma \ref{quasi-geodesic} and it satisfies $[\gamma] = [\lambda_g^n]$ by construction. Thus, if $\lambda_g^n$ is $M$-Morse, $\gamma$ is contained in a $D_1(M, 3)$-neighbourhood of $\lambda_g^n$ by Lemma \ref{lemma:monster}\ref{lemma:close_representatives}. Lemma \ref{lemma:monster}\ref{prop:adapted2.5} about quasi-geodesics in neighbourhoods of Morse quasi-geodesics concludes the proof of Property \ref{lemma:retrieve_Morseness}. Property \ref{lemma:to_corresponding} and \ref{prop:neighbourhoods2} are proven in Lemma 2.27 of \cite{graph_of_groups1}.

Lemma \ref{lemma:constructing_geodesics} \ref{cond3:nice_at_some_point} states that there exists a constant $T_n$ such that $\lambda^n[T_n, \infty)$ is $f_\blam(M_0) = M_1$-Morse. Let $y_1 = \lambda_g^n(s)$ be a point on $\lambda_g^n$ closest to $g\cdot\lambda^n(T_n)$. Lemma \ref{quasi-geodesic} implies that $\gamma_1 = [g\cdot \lambda^n(T_n), y_1]\lambda_g^n[s, \infty)$ is a $(3, 0)$-quasi-geodesic with $[\gamma_1] = [g\cdot\lambda^n[T_n, \infty)]$. Hence, $\lambda_g^n[s', \infty)$ is $M_2 = f_2(M_1, 3)$-Morse by Lemma \ref{lemma:monster}\ref{prop:realization1} for any $s'\geq s$. Define $T_g^n =s$ and let $x = \lambda_g^n(t)$ be an $M$-Morse point on $\lambda_g^n$ with $d_G(1, x)\geq T_g^n$ or in other words such that $t\geq T_g^n$. Lemma \ref{lemma:monster}\ref{prop:geodesic_segments} about geodesics with the same endpoints shows that $\lambda_g^n[0, t]$ is $f_1(M)$-Morse. Lemma \ref{lemma:monster}\ref{prop:multiple_concatenation} about concatenating geodesics concludes the proof of \ref{lemma:far_away_enough2}. 
\end{proof}

\begin{lemma}\label{lemma:morseness_comp1}
Let $g\in G$ and $h\in H$. There exists an increasing function $f_\comp : \mc M\to \mc M$ such that the following conditions hold. 
\begin{enumerate}[label=\roman*)]
    \item If $g$ is $M$-Morse and $h$ is $(M, H)$--Morse, then $\lambda_g^{\abs{h}}$ is $f_\comp(M)$-Morse. \label{prop:from_gourp_to_ray}
    \item If $\lambda_g^{\abs{h}}$ is $M$-Morse, then $g$ and $h$ are $f_\comp(M)$ and $(f_\comp(M), H)$-Morse respectively. \label{prop:from_ray_to_group}
\end{enumerate}
\end{lemma}

\begin{proof}
We first prove Property \ref{prop:from_gourp_to_ray}. For every Morse gauge $M$ there exist only finitely many $h'\in H$ which are $(M, H)$-Morse. Thus there exists a Morse gauge $M_1$ such that $\lambda^{\abs{h'}}$ is $M_1$-Morse for all of these. Let $M_2 = \max\{M_1, M\}$. Lemma \ref{lemma:corresponding2}\ref{lemma:to_corresponding} shows that $\lambda_g^{\abs{h}}$ is $f_\clam(M_2)$-Morse.

Next we prove Property \ref{prop:from_ray_to_group}. Let $n = \abs{h}$. If $\lambda_g^n$ is $M$-Morse, then Property \ref{lemma:retrieve_Morseness} of Lemma \ref{lemma:corresponding2} implies that $g$ and at least one of $\lambda^n[0, \infty)$ and $\lambda^n[0, -\infty)$ is $M_1 = f_\clam(M)$-Morse. Hence, at least one of $a_n$ and $b_n$ is $M_2 = f_\blam(M_1)$-Morse by Lemma \ref{lemma:constructing_geodesics}. Let $C$ be the constant from Lemma \ref{lemma:existance_dual_segments} where we have some control over the critical value of $a_n$ and $b_n$ at $C$. In particular, for all $m >  M_2(C)$ we know that $a_m$ and $b_m$ are not $M_2$-Morse. There are only finitely many elements of $H$ which have length at most $M_2(C)$ and hence there exists some Morse gauge $M_3$ such that all of them are $(M_3, H)$-Morse. Thus for $f_\comp(M) = M_3$ Property \ref{prop:from_ray_to_group} holds.
\end{proof}

\begin{definition}
Let $w$ be a vertex of $V(T')$. We define 
\begin{align}
    H_w = \cup_{\alpha\in E_w} E_{\alpha^+}^+.
\end{align}
That is, $H_w$ is the set of all vertices in $V(T')$ that can be reached from $w$ by going along two outgoing edges.
\end{definition}

Now we are ready to start with the construction of the collection of geodesic local bijections $\{q_v\}_{v\in V(T)}$. As in Section \ref{section:relativemap} we also construct an injective map $q: V(T)\to V(T')$ which agrees with the local bijections $q_v$ on their domains. As before, we start with defining $q(G_0) = G_0'$. Next we repeatedly choose a vertex $v\in V(T)$ such that $q(v)$ is defined but $q_v$ is not, construct the local bijection $q_v$ and finally define $q(w) = q_v(w)$ for every $w\in E_v^+$. We will make sure that $q(v)$ projects to $0$ in $\Gamma'$ for every vertex $v\in V(T)$.

\subsection{Constructing the local bijection $q_v$.}

For every vertex $v\in V(T)$, there is a natural bijection $\psi$ between $G$ and edges of $T$ whose source is $v$ by sending and edge $\alpha$ to $(\bar\alpha^*)^G$. Note that $(\bar\alpha^*)^G$ was defined in \ref{def:xg} and can be viewed as the path needed to travel from $v^*$ to $\bar\alpha^*$. This bijection $\psi$ induces a bijection from $E_v$ to $G$, if $v = G_0$ and from $E_v$ to $G - \{1\}$ otherwise. Similarly, for a vertex $w\in V(T')$ whose projection to $\Gamma'$ is equal to $0$ we have a bijection from $E_w$ to either $G$ (if $w = G_0'$) or $G - \{1\}$ (otherwise). For a vertex $u\in V(T')$ whose projection to $\Gamma$ is $1$, we have a bijection from $E_u$ to $H - \{1\}$.

Define  
\begin{align}
Y_1 = G \qquad &\text{and}\qquad Z_1  = G\times(H - \{1\}).\\
Y_2 = G- \{1\}\qquad &\text{and}\qquad Z_2 = (G-\{1\})\times (H - \{1\}).
\end{align}

As we have seen above, for a vertex $v\in V(T)$ we can identify elements of $E_v$ (and hence $E_v^+$) with elements of $Y_i$, where $i=1$ if and only if $v = G_0$. Moreover, for a vertex $w\in V(T')$ we can identify elements of $H_w$ with $Z_i$, where $i=1$ if and only if $w = G_0'$. The identification works as follows, if $u\in H_w$ and $(\alpha_1, \alpha_2)$ is a path from $w$ to $u$, then we identify $u$ with $((\bar\alpha_1^*)^G,(\bar\alpha_2^*)^G)$. 

Thus, if we have bijections $q_i : Y_i\to Z_i$ for $i\in \{1, 2\}$, and we know that $q(v) = w$ for some vertex $v\in V(T)$ then the bijections $q_i$ induce a bijection $q_v : E_v^+ \to H_w$. Thus to define the maps $q_v$ for all vertices $v$ it is enough to construct bijections $q_1$ and $q_2$. Note that with this construction, we indeed ensure that $\check{q(u)} = 0$ for all $u\in E_v^+$.

Recall that we fixed an $(M_0, X')$-Morse element $h_0\in 
H-\{1\}$. Let $i\in \{1, 2\}$. We construct a bijection $q_i : Y_i\to Z_i$ and maps $ Y_i \to \mb G$ and $ Z_i\to \mb G$, both denoted by $g$, as follows. Enumerate the elements of $Z_i$; this enumeration determines the order in which we define $q_i$. Repeat the following steps: 

\begin{enumerate}
    \item Let $y = (x, h)$ be the lowest index element of $Z_i$ which is not yet in the image of $q$. Let $n$ be the length of $h$.
    \item If $q_i(x)$ is not yet defined and $h = h_0$, set $q_i(x) = y$ and say that $x$ is exact.\label{constr:step2}
    \item Otherwise, choose an element $x_1$ on $\lambda^n_{x}$ such that the following hold and define $q_i(x_1) = y$.
    \begin{itemize}
        \item $q_i(x_1)$ is not yet defined. 
        \item $d_G(1, x_1)\geq T_x^n$, that is we can apply Lemma \ref{lemma:corresponding2}\ref{lemma:far_away_enough2}.
    \end{itemize}
    We say that $x_1$ is not exact.
    \label{constr:step3}
    \item Define $g(y) = [\lambda_x^n]$ and define $g(q_i^{-1}(y)) = g(y)$.
\end{enumerate}

\begin{lemma}
For $i\in \{1, 2\}$ the map $q_i$ as defined above is a well-defined bijection
\end{lemma}

\begin{proof}
Step \eqref{constr:step2} makes sure that $q_i$ is defined on the whole domain. It is clear that a vertex $x_1$ in Step \eqref{constr:step3} satisfying the desired conditions always exists. Thus, $q_i$ is well-defined. Injectivity and surjectivity follow directly from the construction.
\end{proof}

\begin{lemma}
The map $q_v$ induced by $q_i$ for either $i=1$ or $i=2$ is a geodesic local bijection of depth $2$ and Morseness $f$ for sufficiently large functions $f: \mc M\to \mc M$.
\end{lemma}

\begin{proof}
Conditions \ref{cond2:nestedness} - \ref{cond2:injectivity} hold since $q_i : E_v^+\to H_w$ is a bijection. Condition \ref{cond2:partial_surj} holds since we always make sure that $\check{w} = 0$. Namely, if $u$ is a vertex in $T_w'$ whose vertex groups has trivial Morse boundary, then either $u = w$ or $u$ is at distance at least $2$ from $w$. In the latter case, $q_i$ being a bijection from $E_v^+$ to $H_w$ implies that there exists an edge $\alpha\in E_v$ such that $u\in T_{q_v(\alpha^+)}'$. 

The vertex groups of $v$ and $w$ are both equal to $G$. We define $p_v: \mb G\to \mb G$ as the identity map. Recall that Lemma \ref{lemma:properties_of_morseless_stars}\eqref{lemma:vertex_group_morseness} relates the Morseness of geodesics when viewed as a subset of $G$ to the Morseness of geodesics when viewed as a subset of $X$ (resp $X'$). If $z\in \mb G$ is $(M, X)$-Morse, then it is $(f_3(M), G)$-Morse and thus $(f_3^2(M), X')$-Morse. Similarly, if $z\in \mb G$ is $(M, X')$-Morse, it is $(f_3^2(M), X)$-Morse. Thus, Condition \ref{cond2:morseness:tails} holds if $f\geq f_3^2$. Condition \ref{cond2:commutativity} follows directly from the definition of $g$.

Thus, we only have Conditions \ref{cond2:morseness} and \ref{cond2:geodesics} left to prove. Let $u_1\in F_1 = E_v^+$ and let $u_2 \in F_2 = H_w$ satisfy $q_v(u_1) = u_2$. The geodesic path from $w$ to $u_2$ consists of two edges, say $\alpha_1$ and $\alpha_2$ (in this order). Let $x = (\bar{\alpha_1}^*)^G\in G$ and $h  = (\bar{\alpha_2}^*)^G\in H$. Let $\alpha_3$ be the edge from $v$ to $u_1$. Let $x_1 = (\bar{\alpha_3}^*)^G\in G$. With these definitions we have that $q_i(x_1) = (x, h)$.

Recall that Lemma \ref{lemma:properties_of_xg} shows that there is a relation between the Morseness of a vertex $\alpha$ and the element $(\bar{\alpha}^*)^G$. Assume that $u_2\in F_2^M$, that is both $\alpha_1$ and $\alpha_2$ are $(M, X')$-Morse. Lemma \ref{lemma:properties_of_xg} shows that $x$ is $M_1 = f_\xgg(M)$-Morse and $h$ is $(M_1, H)$-Morse. Lemma \ref{lemma:morseness_comp1}\ref{prop:from_gourp_to_ray}, implies that $\lambda_x^{\abs{h}}$ and hence $g(u_2)$ is $M_2 = f_\comp(M_1)$-Morse. Thus $g(u_2)$ is $(f_3(M_2), X')$-Morse and the first implication of Condition \ref{cond2:closeness1} holds for $u_2\in F_2$ for large enough $f$. 

Note that $\alpha_1^* = u_2^{**}$. If $d_{X'}(w^*, \alpha_1^*)\geq D_\dclam(M_3, k)+1$, then $d_G(1, x)\geq D_\dclam(M_3, k)$ and Lemma \ref{lemma:corresponding2}\ref{prop:neighbourhoods2} implies that $x\in \tilde {O}_k^{M_3}(g(u_2))$ and $g(u_2)\in O_k^{M_3}(x)$. Viewing the geodesics as subsets of $X'$ and translating by $w^*$ gives that $\bar\alpha_1^*$ (and hence $\alpha_1^*$) is in $\tilde {O}_k^{f_3(M_3)}(g(u_2), w^*)$ and conversely $w^*\cdot g(u_2)\in O_k^{f_3(M_3)}(\alpha_1^*, w^*)$. Hence Condition \ref{cond2:closeness} holds for $u_2\in F_2^M$ for large enough $f$ and if $D_v(M, k)\geq D_\dclam(M_3, k)+1$. Moreover, if $x_1$ is exact, then $x_1 = x$ is $M_1$-Morse and hence $\alpha_3$ is $(f_\xgg(M_1), X)$-Morse by Lemma \ref{lemma:properties_of_xg}. If $x_1$ is not exact, then $x_1$ lies on $\lambda_x^{\abs{h}}$ and is thus $M_3 = f_1(M_2)$-Morse by Lemma \ref{lemma:monster}\ref{prop:subsegments}. Using Lemma \ref{lemma:properties_of_xg} again we get that $\alpha_3$ is $(f_\xgg(M_3), X)$-Morse. In both cases, Condition \ref{cond2:morseness} holds for $u_2\in F_2$ for large enough $f$.

What is left to prove is the second implication of Condition \ref{cond2:closeness1} for $u_2\in F_2$ and Conditions \ref{cond2:geodesics} and \ref{cond2:morseness} for $u_1\in F_1$. 

Assume that $g(u_2)$ is $(M, X')$-Morse and hence $M_4 = f_3(M)$-Morse. Lemma \ref{lemma:morseness_comp1} implies that $x$ is $M_5 = f_\comp(M_4)$-Morse and $h$ is $(M_5, H)$-Morse. Hence Lemma \ref{lemma:properties_of_xg} implies that $\alpha_1$ and $\alpha_2$ are $(f_\xgg (M_5), X')$-Morse and thus the second implication of Condition \ref{cond2:closeness1} holds for $u_2\in F_2$.

Assume that $\alpha_3$ is $(M, X)$-Morse. Lemma \ref{lemma:properties_of_xg} implies that $x_1$ is $M_6 = f_\xgg(M)$-Morse. If $x_1$ is exact, then $h=h_0$ is $(M_0, H)$-Morse and $x=x_1$ is $M_6$-Morse. By Lemma \ref{lemma:properties_of_xg}, $\alpha_1$ and $\alpha_2$ are $(f_\xgg(M_6), X')$-Morse. If $x_1$ is not exact, then $x_1$ lies on $\lambda_x^{\abs{h}}$ and most importantly, we can apply Lemma \ref{lemma:corresponding2}\ref{lemma:far_away_enough2} to get that $\lambda_x^{\abs{h}}$ is $M_7 = f_\clam(M_6)$-Morse. Combining this with Condition \ref{cond2:closeness1} for $u_2\in F_2$, we get that for large enough $f$, Condition \ref{cond2:morseness} is satisfied for $u_1\in F_1$ both if $x_1$ is exact and also if it is not. Condition \ref{cond2:morseness} for all vertices and Condition \ref{cond2:closeness1} for $u_2\in F_2$ implies that Condition \ref{cond2:closeness1} holds for $u_1\in F_1$ as well if $f$ is large enough. 

Lastly we prove Condition \ref{cond2:closeness} for $u_1\in F_1$. If $x_1$ is exact, the result follows from Condition \ref{cond2:closeness} for $u_2$. If $x_1$ is not exact, then $\bar\alpha_3^{*}$ lies on $\lambda_x^{\abs{h}}$. Hence, Condition \ref{cond2:closeness} is satisfied if $D_v(M, k)\geq k+1$. 

We have shown that we can satisfy all conditions for large enough functions $D_v$ and $f$, which concludes the proof.
\end{proof}

\begin{proof}[Proof of Theorem \ref{theorem:empty_morse_boundary}]
The conditions we had on the function $f$ in the previous proof do not depend on the vertex $v$. Thus, the collection of geodesic local maps $\{q_v\}$ satisfies the Conditions of Proposition \ref{prop:local_bij_imply_bass_serre_map}, which implies that they induce a continuous Bass-Serre tree map. Proposition \ref{prop:bass_serre_implies_homeo} states that a continuous Bass-Serre tree map induces a homeomorphism of the Morse boundaries $\mb X$ and $\mb X'$ and thus concludes the proof. 
\end{proof}

\section{Proof of the main results}\label{section:main_result}
First we prove Theorem \ref{cor:rel_hyp_graph_of_groups} and then use this to prove Theorem \ref{theorem:3-manifold}, which is a more detailed version of Theorem \ref{theorem:manifold_light}. 

The main idea of the proof of Theorem \ref{cor:rel_hyp_graph_of_groups} is to collapse certain parts of the graph of groups to single vertices and then use Theorem \ref{theorem:star_of_groups} to show that the Morse boundary of the groups does not change when replacing the edge groups of the non-collapsed edges with trivial groups.

\begin{proof}[Proof of Theorem \ref{cor:rel_hyp_graph_of_groups}] 
Let $W = \{w_1, \ldots , w_n\}$. Let $\mc G^0 = \mc G$. For $i\in\{1, \ldots, n\}$ we iteratively define $\mc G^i$ as a copy of $\mc G^{i-1}$ with trivial edge groups for edges adjacent to $w_i$. Since every edge is adjacent to at least one vertex of $W$, we have $\mc G^n = \mc G'$. To show that $\mb \pi_1 (\mc G) \cong \mb \pi_1 (\mc G')$ it suffices to show that $\mb \pi_1 (\mc G^{i-1}) \cong \mb \pi_1 (\mc G^i)$ for every index $i\in\{ 1, \ldots n\}$. Fix some index $i$ and let $X_1, \ldots X_m$ be the connected components of $\Gamma - \{w_i\}$. Define $\mc H = \mc G_{X_1, \ldots, X_m}^{i-1}$ and observe that $\mc H$ is a relatively hyperbolic Morseless star. Therefore, by Theorem \ref{theorem:star_of_groups}, $\mb \pi_1(\mc H)\cong \mb \pi_1(\mc H')$, where $\mc H'$ is actually the same graph of groups as $\mc G^i _{X_1, \ldots X_m}$. Lastly, Lemma \ref{lemma:graph_of_groups_trivia} implies $\pi_1(\mc G^{i-1}) = \pi_1( \mc H) $ and $\pi_1(\mc G^i) = \pi_1(\mc H')$, yielding $\pi_1(\mc G^{i-1})\cong \pi_1(\mc G^i)$.
\end{proof}

\begin{definition}
Let $S$ be a topological space and let $G$ be a finitely generated infinite group such that $\mb G \cong S$. We define $S\st S$ as $\mb (G*G)$.
\end{definition}

In light of Theorem 1.2 of \cite{graph_of_groups1}, the homeomorphism type of $S\st S$ does not depend on $G$.

\begin{theorem}\label{theorem:3-manifold}
Let $M$ be a closed and connected 3-manifold and $G$ its fundamental group. Then the Morse boundary $\mb \pi_1(M)$ is homeomrphic to 
\begin{itemize}
    \item \textbf{Totally disconnected Morse boundary:}
    \begin{enumerate}
        \item $\emptyset$: if $M$ is prime and has one of the geometries $S^3, \R^3, \mathbf{H}^2\times \R, Nil, Sol, \widetilde{PSL_2 (\R)}$.
        \item $\{\cdot, \cdot\}$: if $M$ is prime and has geometry $S^2\times \R$ or if $M=\R P^3 \#\R P^3$.\label{2}
        \item Cantor set: if $M$ is not (\ref{2}) and all prime factors have geometry $S^2\times \R$ or $S^3$.
        \item $\omega$-Cantor set: if $M$ is none of the above and neither contains a finite volume hyperbolic component nor a prime factor with geometry $\mathbf{H}^3$.
    \end{enumerate}
    \item \textbf{Connected components are spheres:} For all of these, we require that $M$ has a prime factor with geometry $\mathbf{H}^3$ but no finite volume hyperbolic component.
    \begin{enumerate}[resume]
        \item $S^2$: if $M$ is prime (and has geometry $\mathbf{H}^3$).
        \item $S^2\st S^2$: if $M$ is not prime and all prime factors have geometry $\mathbf{H}^3, S^3$ or $S^2\times \R$.
        \item $S^2 \st \emptyset$: if $M$ is not prime and at least on of its prime factors is non-geometric or has a geometry which is not $\mathbf{H}^3, S^3$ or $S^2\times \R$.
       
    \end{enumerate}
    \item \textbf{Contains $\omega$-Sierpi\'nski curves:} For all of these, we require that $M$ has a finite volume hyperbolic component.
    \begin{enumerate}[resume]
        \item $\omega$-Sierpi\'nski $\st\omega$-Sierpi\'nski: if $M$ does not contain a prime factor with geometry $\mathbf{H}^3$.
         \item $S^2 \st \omega$-Sierpi\'nski: if at least one of the prime factors of $M$ has geometry $\mathbf{H}^3$. 
    \end{enumerate}    
\end{itemize}
Furthermore, the 9 homeomorphism types listed above are pairwise distinct. 
\end{theorem}

The proof strategy of Theorem \ref{theorem:3-manifold} is as follows. We use the geometrization theorem to show that $G$ is isomorphic to the fundamental group of a graph of groups which satisfies Theorem \ref{cor:rel_hyp_graph_of_groups}. Thus, the Morse boundary of $G$ is homeomrphic to a free product. To reduce the classification of the Morse boundary to finitely many cases, we use Theorem 1.1 of \cite{graph_of_groups1}, which essentially tells us that we only care about which spaces occur as Morse boundaries of factors, but not how many times they occur. Lastly we use Theorem \ref{theorem:empty_morse_boundary}, to get rid of factors which have empty Morse boundary in case there are other factors which are not hyperbolic.  

\begin{proof}
Observe that by potentially passing to a finite index subgroup of $\pi_1(M)$ we can assume that $M$ is oriented. 
We start with recalling the geometrization theorem, see \cite{perelman2003ricci}, \cite{perelman2002entropy}, \cite{kleiner2008notes}, \cite{morgan2007ricci} and \cite{cao2006complete}. Let $M$ be as in the statement of Theorem \ref{theorem:3-manifold}. Then $M$ is the connected sum $M_1 \# M_2\# \ldots \# M_n$ of some prime manifolds $M_1, \ldots M_n$ with fundamental groups $G_i = \pi_1(M_i)$. By the van Kampen Theorem, $G= G_1 * G_2 *\ldots *G_n$. 

For each prime manifold $M_i$ we have either 
\begin{itemize}
    \item \textbf{$M_i$ is geometric:} In this case, $M_i$ has one of the following geometries: $S^3, \R^3, \mathbf{H}^3, S^2\times \R, \mathbf{H}^2\times \R, Nil, Sol , \widetilde{PSL_2 \R}$. See \cite{scott1983geometries} for details about those geometries. In particular, $\widetilde{PSL_2 \R}$ is quasi-isometric to $\mathbf{H}^2\times \R$ and $Nil$ and $Sol$ are non-virtually cyclic and solvable (and hence satisfy a law). Any non-virtually cyclic group satisfying a law has empty Morse boundary by \cite{DS05}. The Morse boundary $\mb G_i$ is (in the same order): $\emptyset,\emptyset,  S^2, \{\cdot, \cdot\}, \emptyset, \emptyset, \emptyset, \emptyset$. Moreover, $G_i$ is infinite in all cases but the first. 
    \item \textbf{$M_i$ is non-geometric:} In this case, there exists some disjoint union of embedded tori $\{T_j\}\subset M_i$ and a disjoint union $\{H_k\}$ of finite volume hyperbolic components and $\{S_l\}$ of Seifert-fibered components such that 
    \begin{align}
        M_i - \bigsqcup T_j = \bigsqcup H_k \bigsqcup S_l.
    \end{align}
    Moreover, we can define a graph of groups $\mc G$ with underlying graph $\Gamma$ as follows.
    \begin{itemize}
        \item The components $S_l$ and $H_k$ are the vertices of $\Gamma$. The vertex group corresponding to a component $C$ is $\pi_1(C)$
        \item The edges of $\Gamma$ correspond to the embedded tori $T_j$ with edge groups $\pi_1(T_j)\cong\Z^2$ and maps induced by inclusions $T_i\into C$ for the components $C$ it is adjacent to.
    \end{itemize}
    The graph of groups $\mc G$ has the following properties:
    \begin{itemize}
        \item The image of every edge group is undistorted in $\pi_1(\mc G)$. This follows from \cite{kapovich19963} and \cite{leeb19943}, where they show that $\pi_1(M_i)$ is quasi-isometric to a CAT(0) space and the edge groups are sent to flats under the quasi-isometry they construct. 
        \item The edge groups have infinite index in the vertex groups and are relatively wide in the vertex groups. The former holds because vertex groups are virtually the product of $\Z$ and the fundamental group of a surface $S$ and the edge groups correspond to the product of boundary components of $S$ with $\Z$. The latter holds because edge groups are isomorphic to $\Z^2$, which is wide.
        \item For every finite volume hyperbolic component $H$, $\pi_1(H)$ is hyperbolic relative to $\mc A$, where $\mc A$ contains images of all edge groups of adjacent edges (see \cite{dahmani2003combination}).  
    \end{itemize}
    For finite volume hyperbolic components $H$, $\mb \pi_1(H)$ is an $\omega$-Sierpi\'nski curve and for Seifert-fibered components $S$, $\mb\pi_1(S)=\emptyset$ (see \cite{charney2020complete}). 
    
    Let $\mc H = \{H_k\}_k$ be the set of all finite volume hyperbolic components. Let $\mc X$ be the set of connected components of $\Gamma - \mc H$ and let $X_c\in \mc X$ be a connected component. By \cite{charney2020complete} the Morse boundary of $\pi_1(\mc G \mid _{X_c})$ is either empty (if $X_c$ consists of a single vertex and no edge) or an $\omega$-Cantor space. Furthermore, the graph of groups $\mc G_i = \mc G _ {\mc X}$ satisfies all conditions of Theorem \ref{cor:rel_hyp_graph_of_groups} and thus $\mb G_i \cong \mb \pi_1 (\mc G_i')$, where $\mc G_i'$ is a copy of $\mc G_i$ where all edge groups are trivial. Since all the vertex groups are infinite, either $\mc G_i'$ consist of a single vertex or $\pi_1 (\mc G_i')$ is infinitely ended. Hence by Theorem 1.1 of \cite{graph_of_groups1}, $\mb \pi_1( \mc G_i')$ is homeomorphic to $\mb \mc F_i$ where $\mc F_i$ is the free product of all vertex groups of $\mc G_i'$. The Morse boundary of every factor of $\mc F_i$ is empty, an $\omega$-Cantor space or an $\omega$-Sierpi\'nski curve. Furthermore, there exist a factor whose Morse boundary is an $\omega$-Sierpi\'nski curve if and only if $M_i$ contains a finite volume hyperbolic component.
\end{itemize}

Now we have enough information about the structure of the groups $G_i$ to start with the classification. We have seen that $G = G_1*\ldots *G_n$. Using Theorem 1.2 of \cite{graph_of_groups1} and the fact that for non-geometric primes $\mb \mc F_i\cong \mb G_i$ we get that $\mb G$ is homeomorphic to $\mb (E_1*\ldots * E_n)$, where $E_i = G_i$ if $M_i$ is geometric and $E_i = \mc F_i$ otherwise. Next we want to apply Theorem 1.1 of \cite{graph_of_groups1}, which essentially allows us to ignore factors of the free product if another factor already has the same Morse boundary. However, we can only apply it if $E = E_1*\ldots *E_n$ is infinitely ended. Thus we first deal with the cases where $E$ is not infinitely ended. This can only happen if either $n=1$ and $M_1$ is geometric or if $n=2$ and both $E_1$ and $E_2$ are equal to $\Z/2\Z$. If $M = M_1$ is a geometric prime, we know what the Morse boundary of $G$ is, and since $\R P^3$ is the only prime which has fundamental group $\Z/2\Z$, the latter only happens if $M =  \R P^3 \# \R P^3$ and then its Morse boundary is equal to $\{\cdot, \cdot\}$.

From now on we assume that $E$ is infinitely ended. Theorem 1.1 of \cite{graph_of_groups1}  implies that the homeomorphism type of $\mb E$ (and hence $\mb G$) only depends on set of the homeomorphism types $\mc T$ of non-virtually cyclic vertex groups. If $\mc T$ has cardinality one, then one might think that $\mb E\cong X$, but this is only the case if $X = \mb A$ for some infinitely ended group $A$. Otherwise $\mb E \cong X\st X$. If $\mc T$ is empty, or in other words if all factors are virtually cyclic, then $E$ is virtually free and its Morse boundary is a Cantor set. The possible homeomorphism types in $\mc T$ are: 
\begin{itemize}
    \item $S^2$: this only occurs if one of the prime factors has geometry $\mathbf{H}^3$
    \item $\omega$-Sierpi\'nski curve: this only occurs if one of the prime factors is non-geometric and has a finite volume hyperbolic component.
    \item $\omega$-Cantor space: This can only happen if one of the primes is non-geometric. Furthermore, \cite{charney2020complete} show that $\emptyset\st \emptyset \cong \omega$-Cantor. Thus by replacing every factor whose Morse boundary is an $\omega$-Cantor space by a free product of two factors with empty Morse boundaries, we can assure that the Morse boundary of no factor is an $\omega$-Cantor space. Observe that this does not change the homeomorphism type of $E$.
    \item $\emptyset$: this occurs if one of the prime factors has geometry other than $\mathbf{H}^3, S^3$ or $S^2\times \R$. The case of a factor having Morse boundary homeomorphic to an $\omega$-Cantor space got merged into this case. Thus, if one of the prime factors is non-geometric and has a Seifert-fibered component, then $\emptyset$ is also one of the homeomorphism types. By Theorem \ref{theorem:empty_morse_boundary} we can ignore the existence of factors with empty Morse boundary as long as at least one of the other factors is non-hyperbolic.
\end{itemize}

The classification follows directly from these observations and the only thing left to prove is that the spaces (1)-(9) are pairwise distinct. This is indeed the case: (1)-(4) are the only totally disconnected spaces in the list and clearly pairwise distinct. Spaces (5)-(7) are the only ones that are not totally disconnected but every non-singleton connected component is a sphere. Furthermore, (5) and (6) are compact, while (7) is not and (5) is connected while (6) is not. In space (8) every non-singleton connected component is a an $\omega$-Sierpi\'nski curve while (9) has connected components which are spheres. 
\end{proof}

\bibliographystyle{alpha}
\bibliography{sources}
\end{document}